\documentclass[11pt, oneside]{article}   	
\usepackage{geometry}                		
\usepackage[parfill]{parskip}    		
\usepackage[british]{babel}
\usepackage{amsmath,amssymb,amsthm,graphicx,enumerate}
\usepackage{mathrsfs,bbm,latexsym}
\usepackage{hyperref}
\usepackage{cite}
\usepackage{verbatim}

\usepackage{ wasysym }

\catcode`\>=\active \def>{
\fontencoding{T1}\selectfont\symbol{62}\fontencoding{\encodingdefault}}
\newcommand{\assign}{:=}
\newcommand{\cdummy}{\cdot}
\newcommand{\mathd}{\mathrm{d}}
\newcommand{\nobracket}{}

\newcommand{\tmdummy}{$\mbox{}$}
\newcommand{\tmem}[1]{{\em #1\/}}
\newcommand{\tmop}[1]{\ensuremath{\operatorname{#1}}}

\newcommand{\tmtextit}[1]{{\itshape{#1}}}
\newenvironment{enumeratenumeric}{\begin{enumerate}[1.] }{\end{enumerate}}
\newtheorem{theorem}{Theorem}[section]
\newtheorem{corollary}[theorem]{Corollary}
\newtheorem{definition}[theorem]{Definition}
\newtheorem{lemma}[theorem]{Lemma}
\newtheorem{proposition}[theorem]{Proposition}
\newtheorem{remark}[theorem]{Remark}

\newcommand{\zzone}{\text{\resizebox{.7em}{!}{\includegraphics{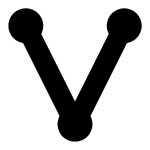}}}}
\newcommand{\zztwo}{\text{\resizebox{.7em}{!}{\includegraphics{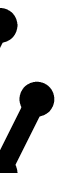}}}}
\newcommand{\zzthree}{\text{\resizebox{.7em}{!}{\includegraphics{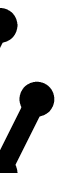}}}}
\newcommand{\zzfour}{\text{\resizebox{1em}{!}{\includegraphics{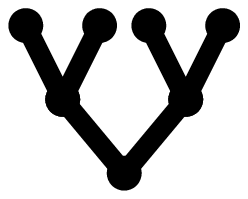}}}}

\newcommand{\CC}{\mathscr{C}}
\usepackage{color} 

\newcommand{\ED}{\mathscr{D} ( \sqrt{-H})}
\newcommand{\ssp}{\mathscr{H}}

\newcommand{\Addresses}{{
		\bigskip
		\footnotesize
		
	 \textsc{ Hausdorff Center for Mathematics \& Institut fur Angewandte 
	 	Mathematik\\ Universit\"at Bonn, Endenicher Allee 60, D-53115 Bonn, Germany \\
	 	E-mail addresses: }
		\textit{gubinelli@iam.uni-bonn.de; bugurcan@uni-bonn.de; zachhuber@iam.uni-bonn.de} }}

\begin{document}
\title{Semilinear evolution equations for the Anderson Hamiltonian in two and
	three dimensions}
\author{
	M. Gubinelli
	\and
	B. Ugurcan
	\and
	I. Zachhuber 
}

\maketitle
\begin{abstract}
	We analyze nonlinear Schr\"odinger and wave equations  whose linear
	part is given by the renormalized Anderson  Hamiltonian in two and three dimensional periodic domains.
\end{abstract}
{\tableofcontents}
\section{Introduction}

The basic aim of this paper is to study the following random Cauchy problems
\begin{align} 
i \partial_t u  &=H u-u|u|^2, \ \ 
u(0) =u_0 \label{equ:NLSwave1} \\ 
\partial^2_t u &=  H u - u^3, \ \  (u, \partial_t u) |_{t = 0}  =  (u_0, u_1) \label{equ:NLSwave2}
\end{align}
on the $d$-dimensional torus $\mathbb{T}^d$ with $d=2,3$. Here $H$ is formally the Anderson Hamiltonian $$H= \Delta + \xi,$$ where
$\xi$ is a space white noise and $\Delta$ the Laplacian with periodic boundary conditions. 

The presence of white noise makes this kind of problems not well posed in classical functional spaces. Indeed it is well known that white noise has sample paths which are only distributions of regularity $-d/2 - \varepsilon$ in H\"older-Besov spaces, where $\varepsilon$ is an arbitrary small but non-zero constant. A sign of this difficulty is the fact that the above equations have to be properly renormalized by subtracting a formally infinite constant to the operator $H$ in order to obtain well defined limits. 

In the parabolic setting there is a, by now, well developed theory of such \emph{singular SPDEs}, thanks to Hairer's invention of the theory of regularity structures~\cite{hairer_theory_2014} and the parallel development of the paracontrolled approach~\cite{gubinelli2015paracontrolled} by Gubinelli, Imkeller and Perkowski. The first results for non parabolic evolution equations have been obtained in~\cite{debussche2016schr} where the authors manage to solve  the linear and the cubic nonlinear (with a range of powers) Schr\"odinger equations with multiplicative noise on $\mathbb{T}^2$ by first applying a transform inspired by~\cite{HairerLabbe15} and then using mass and energy conservation along with certain interpolation arguments. The wave equations in $d=2$ with polynomial non-linearities and additive space-time white noise have been considered in~\cite{gubinelli_renormalization_2017}. The main difficulty is that the absence of parabolic regularization makes the control of the non-linear terms involving the singular noise contributions non-trivial.

Here we exploit the insights of~\cite{allez_continuous_2015} in order to identify an appropriately renormalized version of $H$ as a self-adjoint operator on $L^2(\mathbb{T}^d)$ and use the related spectral decomposition to give a meaning to the above equations as abstract evolution equations in Hilbert space. Our first contribution is then the study of the Anderson Hamiltonian on $\mathbb{T}^3$ and the derivation of some additional results when $d=2$, for example the characterization of the form domain of the operator and some related functional inequalities which are needed in the abstract treatment of the evolution equations. 

For the sake of the reader, and also to illustrate the proof strategy in the $d=3$ case, we pursue a complete treatment of the $d=2$ case showing the self-adjointness of the  Hamiltonian and  the convergence of suitable regularized operators in norm resolvent sense. Resolvent convergence is used in the second part to ``prepare" suitable initial conditions adapted to prove convergence of approximations. We mention also the proof of a version of the classical Brezis-Gallouet inequality~\cite{brga80} for the Anderson Hamiltonian in $d=2$. For $d=3$ we prove that the Anderson Hamiltonian satisfies an inequality   which is analogous to the classical Agmon's inequality, see Lemma \ref{lem:3dagmon}. These functional inequalities are instrumental then in the second part of this work in order to control the non-linear terms of the evolution equations.

An interesting byproduct of our approach is a estimate which expresses the fact that the paraproduct is ``almost" adjoint to the resonant product whose definitions we recall in the Appendix.  This implies in particular that the energy norm with respect to the Anderson Hamiltonian can be estimated from below in a precise way and allows us to characterize (see Proposition \ref{lem:formdom}) both the domain  and the form domain of $H$ by using  certain Sobolev norms. 

The second part of our paper is concerned with the solution of the above equations with different regularities of the initial conditions and with the proof of convergence of solutions of approximate equations where the noise has been regularized and which are then classically well--posed. While the general methodology is the same adopted in~\cite{debussche2016schr}, namely the use of conservation laws and functional inequalities to control the non-linear term, one of the main contributions of our work is to clarify the role of the spectral theory of the Anderson Hamiltonian and of relative functional spaces in the apriori control of the solutions and in the analysis of the non-linear terms. This simplifies and unifies the analysis of the $d=2$ and $d=3$ cases.

After these, having all the  necessary Sobolev and $L^p$-estimates at our disposal along with an analogue of Brezis-Gallouet inequality and proper approximation tools; in Section \ref{sec:section3} we move on to the study of the nonlinear Schr\"odinger and wave equations for the Anderson Hamiltonian (properly shifted for positivity) in dimensions $2$ and $3$.  One important point is that by having undertaken the stochastic analysis of the Anderson Hamiltonian in the paracontrolled setting; we are now in a position to address these PDE problems by using classical techniques which sorts out the stochastic and analytical components and their interaction in a coherent and transparent way.

To officially recap, we study the PDEs~\eqref{equ:NLSwave1} and~\eqref{equ:NLSwave2} (with a range of powers for the nonlinearity) with operator domain and finite energy data, depending on the cases and study their uniqueness.

We also work out the convergence of the solutions of  regularized equations, obtained by suitable approximations of the initial data and the Gaussian white noise, to the solutions of the above PDEs:
\begin{align} 
i\partial_t u_{\varepsilon} &=  H_{\varepsilon} u_{\varepsilon}-u_\varepsilon|u_\varepsilon|^2, u_\varepsilon(0)=u^\varepsilon_0. \label{equ:NLSwaveapp1}\\ 
\partial^2_t u_{\varepsilon}  &=  H_{\varepsilon} u_{\varepsilon} - u_{\varepsilon}^3 \ \tmop{on} \  \mathbbm{T}^d , \ \  (u_{\varepsilon}, \partial_t u_{\varepsilon}) |_{t = 0}  =  (u^{\varepsilon}_0, u^{\varepsilon}_1). \label{equ:NLSwaveapp2}
\end{align}

In Theorem \ref{thm:2dnlsfix}, we establish the well posedness of \eqref{equ:NLSwave1} with  operator domain data in $d=2$.  This is achieved, in part, using our version of the Brezis--Gallouet inequality for the Anderson Hamiltonian.  In Theorem \ref{thm:2ddomnls} we show that the solutions to the regularized equations, namely to \eqref{equ:NLSwaveapp1}, converge to that of equation \eqref{equ:NLSwave1}.  Observe that, in this context, establishing this convergence is important as  the domain of the Anderson Hamiltonian is contained in $\ssp^{1-}$ whereas the domain of the approximations lie in $\ssp^{2}$.  So, there is a drop in smoothness that needs to be addressed carefully. Extension of some of these results to $d=3$ and focusing case  is possible as we prove an analogue of Agmon's inequality in Lemma \ref{lem:3dagmon} to replace the Brezis-Gallouet inequality, please see Remark \ref{rem:nls3dagmon}.  Although we have given the proofs for cubic nonlinearity, one can easily modify this to get a general power nonlinearity, see Remark \ref{rem:nlsdiffpowerfocsmalldata}.

As we characterize the energy domain for the Anderson Hamiltonian in Lemma \ref{lem:formdom}, now we can also make sense of the energy solutions for the NLS, namely \eqref{equ:NLSwave1}.  Accordingly, in Theorem \ref{thm:2dnlsenergy}, we show the existence of solutions.  Observe that, in this case we were not able to show uniqueness; but our result could be considered optimal in view of \cite{BGT04}, as Strichartz estimates for the Anderson Hamiltonian is not known.  Furthermore, as in the domain case, we show in Corollary \ref{corr:energyregconv} the convergence of the regularized solutions. One can generalize the result given in the section to the other  powers of the nonlinearity similar to the domain data case. 

Being able to characterize the energy domain for the Anderson Hamiltonian both in dimensions 2 and 3  enables us to also treat nonlinear stochastic wave equations in both dimensions.  In subsection \ref{subsec:twothreewave}, we prove our results regarding the well posedness of \eqref{equ:NLSwave2} in 2 and 3 dimensions.  In Theorem \ref{thm:wavefix}, we obtain the well posedness with initial data in the domain and the energy domain.  Similar to the Schr\"odinger case we also show convergence of the regularized solutions in Theorem \ref{thm:waveapp}.  Then we conclude  by stating Theorem \ref{thm:waveenergy} which details the well posedness for initial data in the energy domain and the $L^2$ for \eqref{equ:NLSwave2} and  whose proof follows from our earlier considerations in the same section.  By our version of Agmon's inequality and similar methods certain extensions to the different power nonlinearities is possible, please see Remark \ref{rem:wave3ddiffpower} for an elaboration on this topic.

Although we solve the PDEs with an Anderson Hamiltonian which is properly shifted to result in a positive operator,  this does not cause any weaker results.  As known, this shift simply causes a phase shift (i.e. multiplication by $e^{iCt}$ for some constant $C$) in the NLS case, which one can simply rotate back to the solution of the original equation.  In the wave case it simply adds a lower order nonlinearity, in fact a linear term.

In the sequel, we use $\mathscr{H}$ for Sobolev spaces, $L$ for $L^p$-spaces and $\mathscr{C}$ for the Besov-H\"older spaces. As we work either on $\mathbbm{T}^2$ or $\mathbbm{T}^3$ and it is very clear in what setting we consider throughout the paper, we drop the domain parameter i.e. for $\mathscr{H}^2(\mathbbm{T}^3)$ we simply write $\mathscr{H}^2$.  We denote the Gaussian white noise by $\xi$ and enhanced noise by $\Xi$ (see Definition \ref{def:2dDXi} and Theorem \ref{thm:3dren}).

We reserve the letter $A$ for the Anderson Hamiltonian and we use the letter $H$ to denote the  operator shifted by a specific constant $K_\Xi$.  We denote by $C_\Xi$  the constants depending on certain norms (which will be clear from the context) of the (enhanced) noise.  This constant can change value from line to line.  We use the notation $\mathscr{X}$ for the enhanced noise space both in $d=2$ and $d=3$.

For the convenience of the reader we included an Appendix containing mostly  classical results and the results from other papers such as \cite{allez_continuous_2015,gubinelli2015paracontrolled,gubinelli2017kpz}.  In Appendix, we recall the definitions  of relevant function spaces and other harmonic analysis topics such as Littlewood-Paley theory and Bony paraproducts.  The result included in Proposition \ref{lem:circadj} is new and possibly of independent interest.

After the completion of the present work we became aware of recent, still unpublished, work of C. Labb\'e~\cite{labbe_2018} where he constructs the Anderson Hamiltonian in $d=3$ with Dirichlet boundary conditions using regularity structures and produces some results about the law of its eigenvalues. 

\bigskip

\textbf{Acknowledgments.} We gratefully acknowledge partial support  from the German Research Foundation (DFG) via CRC 1060.

\section{The Anderson Hamiltonian in two and three dimensions} \label{sec:section2}

We resume certain concepts and definitions that we will use throughout this section.  First we  recall the definition of Gaussian white noise on $\mathbb{T}^d$.

\begin{definition} \label{def:whiteNoise}
The Gaussian white noise $\xi$ is a family of centered Gaussian random variables $\{ \xi(\phi)~|~ \phi \in L^2(\mathbb{T}^d)\}$, covariance given by
\[
\mathbb{E} \left( \xi(\phi) \xi(\psi) \right) = \langle \phi, \psi \rangle_{L^2(\mathbb{T}^d)}.
\] 
\end{definition}

To get an intuitive description, let  $\hat{\xi}(k)$ be i.i.d centered complex Gaussian random variables with $\hat{\xi}(k)=\overline{\hat{\xi}}(-k)$ and covariance \[
\mathbb{E}(\hat{\xi}(k)\overline{\hat{\xi}}(l))=\delta(k-l).
\]

Formally, the Gaussian white noise, on the torus, can be thought as the following random series

\[
\xi(x)=\sum_{k\in \Lambda}\hat{\xi}(k)e^{2\pi i k\cdot x},
\]

where in this section we will respectively take $\Lambda$ to be  $\mathbb{Z}^2$ and $\mathbbm{Z}^3 \backslash \{ 0 \}$.  That is, in the 3d case, we simply take out the zero mode for ease of computations.

We also define the regularized spatial white
noise as
\begin{equation} \label{eqn:xieps}
\xi_{\varepsilon} (x) = \sum_{k \in \Lambda} m (\varepsilon 
k ) e^{2\pi i k \cdot x} \hat{\xi} (k), \end{equation}
where $m$ is a smooth radial function on $\mathbbm{R} \backslash \{ 0 \}$ with
compact support such that 
\[\underset{}{}_{\underset{}{}} \underset{x \rightarrow 0}{\lim} m (x) = 1.\]
We put
\[ c_{\varepsilon} \assign \underset{k \in \mathbbm{Z}^2}{\sum} \frac{|
	m (\varepsilon  k ) |^2}{1 + | k |^2} \sim \log \left(
\frac{1}{\varepsilon} \right). \]
which will act as one of the renormalization constants in 2d case.  Note that this constant depends on the choice of mollifier. 

We also recall  the Anderson Hamiltonian, which is formally the following operator
\begin{align}
A  = \Delta  + \xi  \end{align}
where $\xi$ is the Gaussian white noise. As we have articulated in the introduction, this operator can not be defined in $L^2(\mathbb{T}^{2,3})$ because of the low H\"older regularity of $\xi$.  The Besov-H\"older regularity of Gaussian white noise on $\mathbb{T}^d$ is $-\frac{d}{2} - \delta$, that is $\xi \in \mathscr{C}^{-\frac{d}{2} - \delta}$ almost surely, for any positive $\delta>0$ \cite{gubinelli2015paracontrolled}.

Therefore, we will consider a renormalization of this operator in the context of paracontrolled distributions which one can formally write as
\begin{align}
A  = \Delta  + \xi - \infty  \end{align}
and to which we will give meaning as a suitable  limit $\varepsilon \rightarrow 0$ of the regularized
Hamiltonians
\begin{align}
A_{\varepsilon} = \Delta + \xi_{\varepsilon} - c_{\varepsilon,} \label{eqn:2dHepsdef} \end{align}
for precise constants $c_{\varepsilon} $.

Accordingly, in this section, we define the Anderson Hamiltonian and introduce suitable regularizations in the setting of paracontrolled distributions in two and three dimensional torus, respectively in the following subsections. Namely, we construct a suitable (dense) domain for the operator and then show closedness, symmetry, self-adjointness and  norm resolvent convergence (of the regularized Hamiltonians).  At the end of both 2d and 3d cases, we prove certain functional inequalities which we will use in the PDE part of the paper, namely in Section \ref{sec:section3}.

\subsection{The two dimensional case}\label{subsec:2dham}

In this part, we work on the 2d torus.  We follow the same line of thought as in \cite{allez_continuous_2015} with important modifications.  In \cite{allez_continuous_2015} authors worked in the 2d case but  our modifications will enable us to use similar proofs in Section \ref{section:3dham}, namely for the 3d case, and also obtain certain functional inequalities such as the Brezis-Gallouet inequality for the Anderson Hamiltonian.  In this section, for paraproducts we use the notations ``$\prec$" and ``$\succ$" and for the resonant product we use ``$\circ$"; please see the Appendix for precise definitions of the function spaces and concepts from harmonic analysis that will be used throughout this section.

\subsubsection{Enhanced noise, the domain and the $\Gamma$-map}

In order to introduce the paracontrolled ansatz (see \cite[Section 3]{GP17}  for  motivation), which will enable us to define the domain of the operator, we need the following definition.

\begin{definition} \label{def:2dencNoise}
	For $\alpha \in \mathbbm{R},$ we define
	 $\mathscr{E}^{\alpha} \assign \CC^{\alpha} \times \CC^{2
		\alpha + 2}$ and 
	$ \mathscr{X}^{\alpha}$ as the closure of the set $\{
	(\eta, \eta \circ (1 - \Delta)^{- 1} \eta + c) : \eta \in \CC^{\infty}
	(\mathbbm{T}^2), c \in \mathbbm{R} \}$ w.r.t. the 
	$\mathscr{E}^{\alpha}$ topology,
	where $\CC^{\alpha} = B_{\infty \infty}^{\alpha}$ denotes the Besov-H\"older space.
\end{definition}

We point out that, $ \mathscr{X}^{\alpha}$ is the space of ``enhanced noise".  In some sense, one needs to lift the singular term white noise into a larger space which, in some sense, also encodes the regularization.  However, one needs to do this consistently; namely the lift should not depend on the mollifier.  This is the content of the following result, which was proved in \cite[Theorem 5.1]{allez_continuous_2015}.

\begin{theorem} 
	\label{thm:2dren}For any $\alpha < - 1$ we have
	\begin{equation} \label{equ:enhancedConv}
	\Xi^{\varepsilon} \assign (\xi_{\varepsilon}, \xi_{\varepsilon} \circ (1
	- \Delta)^{- 1} \xi_{\varepsilon} + c_{\varepsilon}) \rightarrow_{} \Xi=(\Xi_1,\Xi_2)
	\in \mathscr{X}^{\alpha}, 
	\end{equation}
where	the convergence holds as $\varepsilon \rightarrow 0 $  in  $L^p
	(\Omega ; \mathscr{E}^{\alpha})$ for all $p > 1$ and almost surely in $\mathscr{E}^{\alpha}.$ Moreover, the limit is
	independent of the mollifier and $\Xi_1 = \xi$.\end{theorem}

By  this result, one can see that
\[ \| \xi \|_{\CC^{\alpha}}, \quad \|\Xi_2 \|_{\CC^{2 \alpha + 2}},\quad  \| (1 - \Delta)^{- 1} \xi
\|_{\CC^{\alpha + 2}} < \infty \text{ a.s.} \]
by Schauder estimates.

We also recall the following definition which describes the domain of the Anderson Hamiltonian.

\begin{definition}\label{def:2dDXi}
	Assume $- \frac{4}{3} < \alpha < - 1$ and $- \frac{\alpha}{2} < \gamma
	\le \alpha + 2$. Then we define the space of functions
	paracontrolled by the enhanced noise $\Xi$ as follows
	\begin{equation}\label{equ:ansatzmain}
	\mathscr{D}_{\Xi}^{\gamma} \assign \{u \in \ssp^{\gamma} ~\text{s.t.}~ u = u \prec X + B_{\Xi} (u) + u^{\sharp}, ~\text{for}~ u^{\sharp} \in \ssp^2  \}
	\end{equation}
	where $X = (1 - \Delta)^{- 1} \xi \in \CC^{\alpha + 2}$ and
	\[ B_{\Xi} (u) \assign (1 - \Delta)^{- 1} (\Delta u \prec X + 2 \nabla u
	\prec \nabla X + \xi \prec u - u \prec \Xi_2) \in \ssp^{2
		\gamma} . \]
	This space is equipped with the scalar product given by, $u, w \in
	\mathscr{D}_{\Xi}^{\gamma}$,
	\[ \langle u, w
	\rangle_{\mathscr{D}_{\Xi}^{\gamma}} \assign \langle u, w
	\rangle_{\ssp^{\gamma}} + \langle u^{\sharp}, w^{\sharp} \rangle_{\ssp^2}. \]

\end{definition}

We have several remarks now, that explains our modification of this definition.

\begin{remark}\upshape  \label{rem:domainNotationset2d}
	For the rest of the paper, we set \[ \mathscr{D}(A) := \mathscr{D}_{\Xi}^{\gamma}.\] Observe that at this point this is to uniformize notation and we are not stating equality as normed spaces, namely as a normed space we have  $ \mathscr{D}(A) = (\mathscr{D}_{\Xi}^{\gamma}, ||\cdot||_{\mathscr{D}_{\Xi}^{\gamma}})$.  But in the sequel, it will be clear after Proposition \ref{lem:formdom} that we also have $ (\mathscr{D}(A), ||\cdot||_{\mathscr{D}(A)}) = (\mathscr{D}_{\Xi}^{\gamma}, ||\cdot||_{\mathscr{D}_{\Xi}^{\gamma}})$ where $||\cdot||_{\mathscr{D}(A)}$ denotes the standard (functional analytic) domain norm.
\end{remark}

We work out the following  modification of the above ansatz \eqref{equ:ansatzmain} to fit our purposes. Assume $u$ is of the form

\begin{equation}
u = \Delta_{> N} (u \prec X + B_{\Xi} (u)) + u^{\sharp}, \label{eq:ansatz}
\end{equation}
for $2/3<\gamma < 1.$ and  $\Delta_{> N} $ denotes a frequency cut-off at $2^N,$ more precisely,
	\[ \Delta_{> N} f \assign \mathscr{F}^{- 1} \chi_{| \cdummy | > 2^N}
	\mathscr{F}f, \]
with $N \in
\mathbbm{N}$ which will be chosen depending on the (enhanced) noise $\Xi .$ We also
define
\begin{equation}
B_{\Xi} (u) \assign (1 - \Delta)^{- 1} (\Delta u \prec X + 2 \nabla u \prec
\nabla X + \xi \prec u - u \prec \Xi_2) . \label{eq:def-G}
\end{equation}
Note that by Schauder estimates we have the following bound for $B$, 
\[ \| B_{\Xi} (u) \|_{\ssp^s} \lesssim_s C_{\Xi} \| u \|_{\ssp^{s - \gamma}}, \qquad
s \in [0, 2 \gamma] . \]
Recall that, we denote by $C_\Xi$ a constant that depends explicitly on the norm of the realization of the enhanced noise $\Xi$  and can change from line to line.
\begin{remark}\upshape	\upshape \label{rem:equivalentCuttofSp}
	This modification changes the decomposition by a smooth function so it does not change
	the space. Strictly speaking, one obtains a different norm depending on $N,$ which  is equivalent to the
	$\mathscr{D}(A)$ norm above.
	In fact, assume that for a function $f$ and some $N \ge 1$ we have
	\begin{align*}
	f  = & f \prec X + B_{\Xi} (f) + f_1^{\sharp}\\
	\tmop{and} & \\
	f  = & \Delta_{> N} (f \prec X + B_{\Xi} (f)) + f_2^{\sharp} .
	\end{align*}
	Then we readily have the estimate
	\begin{align*}
	\| f^{\sharp}_1 \|_{\ssp^2}  = & \| f - f \prec X + B_{\Xi} (f) \|_{\ssp^2}\\
	= & \| f - \Delta_{> N} (f \prec X + B_{\Xi} (f)) - \Delta_{\le
		N} (f \prec X + B_{\Xi} (f)) \|_{\ssp^2}\\
	\le & \| f_2^{\sharp} \|_{\ssp^2} + C (N, \Xi) \| f \|_{\ssp^{\gamma}}
	\end{align*}
	and analogously $\| f^{\sharp}_2 \|_{\ssp^2} \le \| f_1^{\sharp} \|_{\ssp^2}
	+ C (N, \Xi) \| f \|_{\ssp^{\gamma}}$. This proves the norm equivalence.
\end{remark}

With this modification of the ansatz, we can
write $u$ as a function of $u^{\sharp}$. In order to do so, we define the following linear map $\Gamma$ 
\[ \Gamma f = \Delta_{> N} (\Gamma f \prec X + B_{\Xi} (\Gamma f)) + f, \]
so that $u = \Gamma u^{\sharp}$. For $N$ large enough, depending on the realization of $\Xi$, we can show that this map exists and has useful bounds. 

\begin{remark}\upshape
In the following, we will utilize this map $\Gamma$ to show density of the domain, symmetry and norm resolvent convergence.  The key point is the map $\Gamma$ can also be defined in the 3d case and be used there in a similar manner, which we will do in the 3d section.
\end{remark}

By these considerations, we can bound certain Sobolev norms of $u$ by that of $u^\sharp$, which is the content of the following result.
\begin{proposition}
	\label{lem:gamma}We can choose $N$ large enough depending only on $C_{\Xi}$ and
	$s$ so that
	\begin{equation}
	\| \Gamma f \|_{L^{\infty}} \le 2 \| f \|_{L^{\infty}},
	\label{eq:gamma-bound-Linfty}
	\end{equation}
	\begin{equation}
	\| \Gamma f \|_{\ssp^s} \le D_\Xi \| f \|_{\ssp^s}.
	\label{eq:gamma-bound-sobolev}
	\end{equation}
	for some constant $D_\Xi$ for $s \in [0, \gamma]$  and  $D_\Xi=3$ for $s \in [0, \gamma)$.
\end{proposition}

\begin{proof}
	Let us start with proving the $L^{\infty}$ bound. Choose $\delta > 0$ and
	let $g = \Gamma f$, we have
	\begin{align*}
	\| B_{\Xi} (g) \|_{\CC^{\gamma - \delta}} & \le \| \Delta g \prec X
	\|_{\CC^{\gamma - \delta - 2}} + 2 \| \nabla g \prec \nabla X \|_{\CC^{\gamma
			- \delta - 2}}\\ 
		&+ \| \xi \prec g \|_{\CC^{\gamma - \delta - 2}} + \| g \prec
	\Xi_2 \|_{\CC^{\gamma - \delta - 2}} \\
	& \le 3 \| g \|_{\CC^{- \delta}} \| X \|_{\CC^{\gamma}} + \| \xi
	\|_{\CC^{\gamma - 2}} \| g \|_{\CC^{- \delta}}\\
	 &+ \| g \|_{ \CC^{- \delta}} \| \Xi_2 \|_{\CC^{\gamma - 2}} \lesssim C_{\Xi} \| g \|_{ \CC^{- \delta}}
	\lesssim_{\delta} C_{\Xi} \| g \|_{L^{\infty}} 
	\end{align*}
	by paraproduct estimates and the fact that $\| g \|_{ \CC^{- \delta}}
	\lesssim_{\delta} \| g \|_{L^{\infty}}$ for any small $\delta > 0$. Now we can write,
	\[ \| g \|_{L^{\infty}} \le \| \Delta_{> N} (\Gamma f \prec X +
	B_{\Xi} (\Gamma f)) \|_{L^{\infty}} + \| f \|_{L^{\infty}} \le
	2^{(\delta - \gamma) N} \| g \prec X + B_{\Xi} (g) \|_{\CC^{\gamma -
			\delta}} + \| f \|_{L^{\infty}} \]
	\[ \le 2^{(\delta - \gamma) N} (\| X \|_{\CC^{\gamma - \delta}} +
	C_{\delta} C_{\Xi}) \| g \|_{L^{\infty}} + \| f \|_{L^{\infty}} \le
	C_{\delta} C_{\Xi} 2^{(\delta - \gamma) N} \| g \|_{L^{\infty}} + \| f
	\|_{L^{\infty}} \]
	 and choose $N$ large enough so that $2 C_{\delta} C_{\Xi} 2^{(\delta - \gamma)
		N} \le 1$ which implies $\| g \|_{L^{\infty}} \le 2 \| f
	\|_{L^{\infty}}$.
	
	For the $\ssp^s$ bound we can proceed more simply by noting that
	\[ \| B_{\Xi} (g) \|_{\ssp^{\gamma}} \le (3 \| X \|_{\CC^{\gamma}} + \| \xi
	\|_{\CC^{\gamma - 2}} + \| \Xi_2\|_{\CC^{\gamma - 2}}) \| g
	\|_{L^2} \]
	and if $s \le \gamma$ we have
	\[ \| g \|_{\ssp^s} \le C C_{\Xi} 2^{(s - \gamma) N} \| g \|_{L^2} + \| f
	\|_{\ssp^s} . \]
	If $s = 0$ we can choose $N$ large enough so that $\| g \|_{L^2} \le 2
	\| f \|_{L^2}$ and as a consequence we have also
	\[ \| g \|_{\ssp^s} \le 2 C C_{\Xi} 2^{(s - \gamma) N}  \| f \|_{L^2} +
	\| f \|_{\ssp^s} \]
	for all $s \le \gamma$. If $s < \gamma$ we can have $N$ large enough
	(depending on $s$) so that $\| g \|_{\ssp^s} \le 3 \| f \|_{\ssp^s}$.
\end{proof}

\begin{remark}\upshape\upshape
	Note that $\mathscr{D}^\gamma_\Xi$ is actually independent of $\gamma$, since for $\gamma,\gamma'\in(2/3,1)$ we can compute
	\[
	||u||_{\ssp^\gamma}\lesssim ||u^\sharp||_{\ssp^2}\lesssim||u||_{\mathscr{D}^{\gamma'}_{\Xi}},
	\]
	and vice versa, so the $\mathscr{D}_\Xi^\gamma$ and $\mathscr{D}_\Xi^{\gamma'}$norms are equivalent and we will from now on drop the $\gamma$ and write simply $\mathscr{D}_\Xi$.  That is, we have $\mathscr{D}(A) = \mathscr{D}_\Xi$.
\end{remark}

As a first step we prove that the domain of $A$, now defined to be $\mathscr{D}(A),$ is dense in $L^2.$  Before that we note the following remark and then a lemma.

\begin{remark}\upshape\label{rem:smoothNotation}
In the sequel, we put
\[
 X_\varepsilon = (1 - \Delta)^{- 1} \xi_\varepsilon.
\]
and similar to the operator $\Gamma$ in Lemma \ref{lem:gamma} we define  $\Gamma_\varepsilon$ as follows
\[ \Gamma_\varepsilon u := \Delta_{> N} (\Gamma_\varepsilon u \prec X_\varepsilon + B_{\Xi_\varepsilon} (\Gamma_\varepsilon u)) + u^\sharp, \]
where $\Xi_\varepsilon \rightarrow  \Xi$ in $\mathscr{X}^\alpha$. 
\end{remark}

For the above introduced $\Gamma_\varepsilon$ we prove the following lemma, which will be useful in the sequel.

\begin{lemma} \label{lem:identityGammaconv}
	We have that $|| \text{id} - \Gamma  \Gamma_\varepsilon^{-1}||_{\mathscr{H}^\gamma \rightarrow \mathscr{H}^\gamma} \rightarrow 0$.
\end{lemma}

\begin{proof}
	For $f \in \mathscr{H}^\gamma$, we can write, by using Proposition \ref{lem:gamma}
	\begin{align*}
	|| f -  \Gamma \Gamma_\varepsilon^{-1}(f)||_{ \mathscr{H}^\gamma} &= || \Gamma(f - f\prec X + B_{\Xi} (f))  -   \Gamma(f - f\prec X_\varepsilon + B_{\Xi^\varepsilon} (f)) ||_{ \mathscr{H}^\gamma}\\
	&= || \Gamma(f\prec ( X_\varepsilon - X) + B_{(\Xi^\varepsilon - \Xi)} (f))||_{ \mathscr{H}^\gamma}\\
	& \leq D_\Xi ||f||_{\mathscr{H}^\gamma} || \Xi^\varepsilon - \Xi||_{ \mathscr{X^\alpha}}
	\end{align*}
	which shows that $\text{id} = \Gamma  \Gamma^{-1}$ converges to $\Gamma  \Gamma_\varepsilon^{-1}$ in operator norm.
\end{proof}

\begin{corollary}
	The space $\mathscr{D}(A),$ as defined in Definition \ref{def:2dDXi},  is dense in $\mathscr{H}^\gamma$, therefore dense  in $L^2$.
\end{corollary}  
\begin{proof}
For $f \in \mathscr{H}^\gamma$, by Lemma \ref{lem:identityGammaconv}, simply observe that
\[
|| f -   \Gamma \Gamma_\varepsilon^{-1} f||_{\mathscr{H}^\gamma} \rightarrow 0.
\]
Hence, the result.
\end{proof}

We are in a position to define the operator $A$ in $L^2$ on  its domain $\mathscr{D}(A)$.

\begin{definition}
We define the  operator $A:\mathscr{D}(A) \to L^2$ as
	\begin{align} \label{equ:defandersonHam}
	A u : = \Delta u^{\sharp} + u^{\sharp} \circ \xi + G (u),\end{align}
	 where in $u \xi= u \prec \xi +  \xi \prec u+ u\circ \xi$
	 we have defined
\begin{align*} u \circ \xi&:= (\Delta_{> N}(u\prec X +B_{\Xi}(u))+u^\sharp) \circ \xi\\
&= C_N (u, \xi, X) + u\Xi_2  + (\Delta_{> N}B_{\Xi}) \circ \xi + u^\sharp \circ \xi \ \ \text{and we put} \  \  \\
G (u) &: = \Delta_{\le N} (u \prec \xi + u \succ \xi + u \prec \Xi_{2})\\ &+ \Delta_{>
	N} (- B_{\Xi} (u) - u \prec X + u \succcurlyeq \Xi_2 + C_N (u,
X, \xi) + B_{\Xi} (u) \circ \xi). \end{align*}

Here, $C_N (u, \xi, X) := \left(\Delta_{> N} (u\prec X)\right))\circ \xi- u(X\circ \xi)$ is the modified commutator which satisfies bounds (depending on the fixed $N$) as shown in Proposition \ref{prop:commu2}.

\end{definition}

\begin{remark}\upshape
	By using the regularities in Definition \ref{def:2dDXi}, one can easily check, through Proposition \ref{lem:paraest}, that $Au$ is in fact in $L^2$.  Later, respectively in Proposition \ref{prop:opApp} and Theorem \ref{thm:normResolventMain}  we obtain this operator as a norm and norm resolvent limit of $A_\varepsilon$ which perfectly motivates the informal identity
	\[
	A = \Delta + \xi - \infty.
	\]
\end{remark}

In the following result, we show that the $\ssp^2$-norm of $u^\sharp$ can be bounded above by the (standard) domain norm of $A$.

\begin{proposition} \label{lem:h2bound}
	There exists a constant $C_{\Xi} > 0$ depending on the enhanced
	noise such that
	\begin{equation}\label{equ:h2bound}
	\| u^{\sharp} \|_{\ssp^2} \le 2 \| A u \|_{L^2} + C_{\Xi} \| u
	\|_{L^2}. \end{equation}
\end{proposition}

\begin{proof}

	First, we note that $\Delta u^{\sharp} \in L^2$ by
		assumption. For the resonant term we compute
		\begin{align}
		\| u^{\sharp} \circ \xi \|_{L^2} &\le \| (\Delta_{\le M}
		u^{\sharp}) \circ \xi \|_{L^2} + \| (\Delta_{> M} u^{\sharp}) \circ \xi
		\|_{L^2} \nonumber\\&\le C_{\Xi} 2^{2 M} \| u^{\sharp} \|_{L^2} + \| \Delta_{>
			M} u^{\sharp} \|_{\ssp^{1 + \delta}} \| \xi \|_{\CC^{- 1 - 2\delta}}
		\label{eqn:usharpxi}
		\end{align}
		for $\delta$ sufficiently small, giving, for any $M \ge 0$,
		\[ \| u^{\sharp} \circ \xi \|_{L^2} \lesssim_{\Xi} (2^{2 M} \| u^{\sharp}
		\|_{L^2} + 2^{(\delta - 1) M} \| u^{\sharp} \|_{\ssp^2}), \]
		where we have used Bernstein's inequality (Lemma \ref{lem:bernstein}) and
		Theorem \ref{thm:2dren} for the noise. Using again Bernstein's inequality
		for the low-frequency terms and the paraproduct estimates for the
		high-frequency terms, we obtain the bound
		\[ \| G (u) \|_{L^2} \le C_{\Xi} \| u \|_{\ssp^{\gamma}}, \]
		for $\gamma < 1,$ where the constant can be chosen as
		\[ C_{\Xi} = C 2^{2 N} (\| \xi \|_{\CC^{\alpha}} + \| \Xi_2
		\|_{\CC^{2 \alpha + 2}} ) \]
		with $\alpha < - 1$ as before.
		
		By using these, for the $\ssp^2$ bound, we compute
		\[ \| \Delta u^{\sharp} \|_{L^2} \le \| A u \|_{L^2} + \| u^{\sharp}
		\circ \xi \|_{L^2} + \| G (u) \|_{L^2}. \]
		Now, as above we have
		\[ \| u^{\sharp} \circ \xi \|_{L^2} \lesssim_{\Xi} (2^{2 M} \| u^{\sharp}
		\|_{L^2} + 2^{\gamma M} \| u^{\sharp} \|_{\ssp^2}) \]
		and
		\begin{equation}
		\| G (u) \|_{L^2} \lesssim_{\Xi} \| u \|_{\ssp^{\gamma}} \lesssim_{\Xi}
		\| u^{\sharp} \|_{\ssp^{\gamma}}  \lesssim_\Xi\| \Delta_{> M} u^{\sharp} \|_{\ssp^{\gamma}} + \| \Delta_{\le M} u^{\sharp} \|_{\ssp^{\gamma}}
		\label{eqn:Gubound}
		\end{equation}
		and using again Bernstein's inequality for the low-frequency part we get
		\[ \| G (u) \|_{L^2} \lesssim C_{\Xi} (2^{2 M} \| u \|_{L^2} + 2^{-\gamma M} \| u^{\sharp} \|_{\ssp^2}) \]
		where we have used the straightforward bound
		$\| u^{\sharp} \|_{L^2} \le C_{\Xi} \| u \|_{L^2}$.
		Finally, choosing $M$ large enough (depending on $\Xi$), we obtain
		\[ \| u^{\sharp} \|_{\ssp^2} \le 2 \| A u \|_{L^2} + C_{\Xi} \| u
		\|_{L^2} . \]

	Hence, the result.
\end{proof}

\subsubsection{Density, symmetry, self-adjointness and convergence}

In the following, we show that $A$ is a closed and symmetric operator on $\mathscr{D}(A)$.  We first establish closedness.

\begin{proposition} \label{prop:opclosed}
	We have that $A$ is a closed  operator on its dense domain $\mathscr{D}(A)$.
\end{proposition}

\begin{proof}
	Assume $(u_n) \subset \mathscr{D} (A)$ is
	a sequence s.t.
	\begin{eqnarray*}
		u_n \rightarrow u & \tmop{in} & L^2\\
		& \tmop{and} & \\
		A u_n \rightarrow g & \tmop{in} & L^2
	\end{eqnarray*}
	for some $g \in L^2 .$ Then $u_n^{\sharp}$ forms a Cauchy
	sequence in $\ssp^2$ and thus converges to a limit that we call
	$w^{\sharp}$. Moreover $\Gamma w^{\sharp} = u,$ so
	$u \in \mathscr{D} (A)
	.$ Thus
	\begin{eqnarray*}
		\| A u - g \|_{L^2} & \le & \| A u - A u_n \|_{L^2} + \| A u_n - g
		\|_{L^2}\\
		& \le & C_\Xi \| w^{\sharp}- u_n^{\sharp} \|_{\mathscr{H}^2} + \| A u_n - g
		\|_{L^2}
	\end{eqnarray*}
	where the second step comes from the proof of Proposition \ref{lem:h2bound} . Since
	both terms on the right-hand side tend to zero as $n \rightarrow \infty$ we
	get $A f = g$, namely that $A$ is closed.
	
	\end{proof}

Before we show the symmetry we need the following approximation result.

\begin{proposition} \label{prop:opclosed}
For every $u \in \mathscr{D}(A)$ there exists a sequence $u_\varepsilon \in \mathscr{H}^2$ such that 
\begin{equation} \label{equ:h2app}
|| u - u_\varepsilon  ||_{\mathscr{H}^\gamma} + || u^\sharp - u_\varepsilon^\sharp ||_{\mathscr{H}^2} \rightarrow 0
\end{equation}
as $\varepsilon \rightarrow 0$. 
\end{proposition}

\begin{proof}

	 We take an arbitrary approximation $u_\varepsilon^\sharp \rightarrow u^\sharp$ in $\mathscr{H}^2$ and set $q_\varepsilon := \Gamma_\varepsilon (u_\varepsilon^\sharp), g_\varepsilon := \Gamma(u_\varepsilon^\sharp), $   that is
	\begin{align*}
	q_\varepsilon &=\Delta_{> N}(q_\varepsilon\prec X_\varepsilon +B_{\Xi_\varepsilon}(q_\varepsilon))+u_\varepsilon^\sharp\\
	g_\varepsilon &=\Delta_{> N}(g_\varepsilon\prec X +B_{\Xi}(g_\varepsilon))+u_\varepsilon^\sharp.
	\end{align*}
	We have $|| u^\sharp - u_\varepsilon^\sharp ||_{\mathscr{H}^2} \rightarrow 0$ by construction.  We can write
	\[
	|| u - q_\varepsilon  ||_{\mathscr{H}^\gamma} \leq || u - g_\varepsilon  ||_{\mathscr{H}^\gamma} + || g_\varepsilon  - q_\varepsilon  ||_{\mathscr{H}^\gamma} 
	\]
	We have, by Proposition \ref{lem:gamma} that
	\[
	|| u - g_\varepsilon  ||_{\mathscr{H}^\gamma} = || \Gamma(u^\sharp) - \Gamma(u_\varepsilon^\sharp)  ||_{\mathscr{H}^\gamma}\leq || u^\sharp - u_\varepsilon^\sharp ||_{\mathscr{H}^2} \rightarrow 0
	\]
	Similar to Proposition \ref{lem:gamma}, we also have that
	\[|| g_\varepsilon  - q_\varepsilon  ||_{\mathscr{H}^\gamma}  = || \Gamma_\varepsilon(u^\sharp) - \Gamma_\varepsilon(u_\varepsilon^\sharp)  ||_{\mathscr{H}^\gamma} = || \Gamma_\varepsilon(u^\sharp- u_\varepsilon^\sharp)  ||_{\mathscr{H}^\gamma} \lesssim || u^\sharp- u_\varepsilon^\sharp  ||_{\mathscr{H}^\gamma} \rightarrow 0. \].  This settles \eqref{equ:h2app}.
	
\end{proof}

\begin{remark}\upshape
Observe that, though stated generally, in Proposition \ref{prop:opclosed} one might as well take $u_\varepsilon^\sharp = u^\sharp$ and obtain the approximations $u_\varepsilon = \Gamma_\varepsilon(u^\sharp)$ with the stated properties.
\end{remark}

Now, we are also ready to show the norm convergence of the approximating operators.

\begin{proposition} \label{prop:opApp}
	Let  $u^\sharp \in \mathscr{H}^2$, $u  = \Gamma(u^\sharp )$ and  $u_\varepsilon  = \Gamma_\varepsilon(u^\sharp )$. We have that
	\begin{equation}\label{equ:normOpConv2}
	|| Au - A_\varepsilon u_\varepsilon ||_{L^2} \lesssim_{\Xi}   || \Xi_\varepsilon - \Xi||_{\mathscr{X}^\alpha} ||u^\sharp||_{\mathscr{H}^2}.
	\end{equation}
	Consequently, this implies that
	\begin{equation} \label{equ:normOpConv3}
	|| A\Gamma - A_\varepsilon \Gamma_\varepsilon   ||_{\mathscr{H}^2 \rightarrow L^2}  \rightarrow 0.
	\end{equation}
	That is to say, $A_\varepsilon \Gamma_\varepsilon \rightarrow A\Gamma$ in norm.
	\end{proposition}

\begin{proof}
By using the formula \ref{equ:defandersonHam}, we observe that all terms in $Au  - A_\varepsilon u_\varepsilon$ are bilinear.  For the upper bound, by addition and subtraction of cross terms, one obtains terms of the form
\begin{equation} \label{equ:normOpconv4}
||\Xi_\varepsilon - \Xi||_{\mathscr{X}^\alpha} ||u^\sharp||_{\mathscr{H}^2} + ||u_\varepsilon - u||_{\mathscr{H}^\gamma} ||\Xi||_{\mathscr{X}^\alpha}.
\end{equation}
Now, recall that $u = \Gamma (u^\sharp), u_\varepsilon = \Gamma_\varepsilon (u^\sharp)$.  Then, we obtain terms of the form
\begin{equation} \label{equ:normOpconv4}
||\Xi_\varepsilon - \Xi||_{\mathscr{X}^\alpha} ||u^\sharp||_{\mathscr{H}^2} + ||\Gamma_\varepsilon - \Gamma||_{\mathscr{H}^2 \rightarrow\mathscr{H}^\gamma} ||u^\sharp||_{\mathscr{H}^2} ||\Xi||_{\mathscr{X}^\alpha} 
\end{equation}

By using Lemma \ref{lem:identityGammaconv} and the estimate in its proof, the result  \eqref{equ:normOpConv2} is now immediate.

\end{proof}

After this we immediately obtain the symmetry of the operator.

\begin{corollary}
	Let $u \in \mathscr{D}(A)$ and  $u_\varepsilon \in \mathscr{H}^2$ be as in Proposition \ref{prop:opclosed}. Then, we obtain
	\begin{equation}\label{equ:formapp}
	\langle A_\varepsilon u_\varepsilon, v_\varepsilon \rangle \rightarrow \langle A u, v \rangle.
	\end{equation}
	Consequently, we have that $A$ is a symmetric operator on its dense domain $\mathscr{D}(A)$.
\end{corollary}
\begin{proof}
This directly follows from Proposition \ref{prop:opApp}.  Using the  symmetry of $A_\varepsilon$ implies the symmetry of $A$ through the equalities
\begin{align*}
\langle u,Av\rangle=\lim\limits_{\varepsilon\to0}\langle u_\varepsilon,A_\varepsilon v_\varepsilon\rangle=\lim\limits_{\varepsilon\to0}\langle A_\varepsilon u_\varepsilon, v_\varepsilon\rangle=\langle Au,v\rangle.
\end{align*}
Hence, the result.
\end{proof}

The next result shows that the quadratic form given by $- A$ is, through addition of a constant, bounded from
below by the $\ssp^1$ norm of $u^{\sharp}$.   We  will later use this estimate to bound certain norms by a (conserved)
energy, when we deal with the NLS and the nonlinear wave equations.

\begin{proposition}
	\label{lem:2dh1bound}There exists a constant $C_{\Xi}>0$ such that
	\begin{equation} \label{equ:2dconstant}
	\frac{1}{2} \langle \nabla u^{\sharp}, \nabla u^{\sharp} \rangle
	\le - \langle u, A u \rangle + C_{\Xi} \| u \|_{L^2}^2 .
	 \end{equation}
\end{proposition}

\begin{proof}
	Expanding the Ansatz and integrating by parts we get
	\begin{align*}
    \langle u, A u \rangle =& \langle u, \Delta u^{\sharp} \rangle + \langle
	u, u^{\sharp} \circ \xi \rangle + \langle u, G (u) \rangle\\ 
	 =& \langle \Delta_{> N} (u \prec X), \Delta u^{\sharp} \rangle + \langle
	u^{\sharp}, \Delta u^{\sharp} \rangle + \langle u, u^{\sharp} \circ \xi
	\rangle + \langle u, G (u) \rangle + \langle \Delta_{> N} B_{\Xi} (u),
	\Delta u^{\sharp} \rangle \\
	=& - \langle \Delta_{> N} (u \prec \xi), u^{\sharp} \rangle+\langle\Delta_{>N}(u\prec X),u^\sharp\rangle - \langle
	\nabla u^{\sharp}, \nabla u^{\sharp} \rangle + \langle u, u^{\sharp}
	\circ \xi \rangle \\
	& + \langle u, G (u) \rangle + \langle \Delta_{>N}(\Delta u \prec X), u^{\sharp} \rangle
	+ 2 \langle \Delta_{>N}(\nabla u \prec \nabla X), u^{\sharp} \rangle + \langle
	\Delta_{> N} B_{\Xi} (u), \Delta u^{\sharp} \rangle \\
	 =& D (u, \xi, \Delta_{> N} u^{\sharp}) - \langle \nabla u^{\sharp}, \nabla
	u^{\sharp} \rangle + \langle u, (\Delta_{\le N} u^{\sharp}) \circ
	\xi \rangle+\langle\Delta_{>N}(u\prec X),u^\sharp\rangle \\
	& + \langle u, G (u) \rangle + \langle \Delta_{>N}(\Delta u \prec X), u^{\sharp} \rangle
	+ 2 \langle \Delta_{>N}(\nabla u \prec \nabla X), u^{\sharp} \rangle + \langle
	\Delta_{> N} B_{\Xi} (u), \Delta u^{\sharp} \rangle 
	\end{align*}
	where
	\[ D (u, \xi, \Delta_{> N} u^{\sharp}) \assign \langle u, (\Delta_{> N}
	u^{\sharp}) \circ \xi \rangle - \langle u \prec \xi, \Delta_{> N}
	u^{\sharp} \rangle . \]
	Now fix a sufficiently small $\delta > 0$, then
	\[ | \langle u, (\Delta_{\le N} u^{\sharp}) \circ \xi \rangle |
	\lesssim 2^{2 (1 + 2 \delta) N} \| \xi \|_{\CC^{- 1 - \delta}} \|
	u^{\sharp} \|_{L^2} \| u \|_{L^2} \lesssim C_{\Xi} 2^{2 (1 + 2 \delta) N}
	\| u \|_{L^2}^2 \]
	since from the Ansatz we have easily $\| u^{\sharp} \|_{L^2} \le
	C_{\Xi} \| u \|_{L^2}$. Moreover
	\begin{align*}
	 | \langle u, G (u) \rangle | &\lesssim_{\Xi} \| u \|_{L^2}^2 + \|
	u^{\sharp} \|_{\ssp^{1 - \delta}}^2 \\
	 | \langle \Delta u \prec X, u^{\sharp} \rangle | + | \langle \nabla u
	\prec \nabla X, u^{\sharp} \rangle | &\lesssim \| u \|_{\ssp^{1 - \delta}} \|
	X \|_{\CC^{1 - \delta}} \| u^{\sharp} \|_{\ssp^{2 \delta}} \le C_{\Xi} 
	\| u^{\sharp} \|_{\ssp^{1 - \delta}}^2 \\
	 | \langle \Delta_{> N} B_{\Xi} (u), \Delta u^{\sharp} \rangle | = |
	\langle \Delta_{> N} \Delta B_{\Xi} (u), u^{\sharp} \rangle | &\le
	\| B_{\Xi} (u) \|_{\ssp^{2 - 2 \delta}} \| u^{\sharp} \|_{\ssp^{2 \delta}}
	\le C_{\Xi} \| u^{\sharp} \|_{\ssp^{1 - \delta}}^2 \end{align*}
	and similarly we bound the term $\langle\Delta_{>N}(u\prec X),u^\sharp\rangle$.
	 By proof of Proposition~\ref{lem:circadj}, we have also
	\[ | D (u, \xi, \Delta_{> N} u^{\sharp}) | \lesssim \| \xi \|_{\CC^{- 1 -
			\delta}} \| u \|_{\ssp^{(1 + \delta) / 2}} \| \Delta_{> N} u^{\sharp}
	\|_{\ssp^{(1 + \delta) / 2}} \lesssim C_{\Xi} \| u^{\sharp} \|_{\ssp^{1 -
			\delta}}^2 \]
	Using that
	\[ \| u^{\sharp} \|_{\ssp^{1 - \delta}}^2 \lesssim \| \Delta_{> M} u^{\sharp}
	\|_{\ssp^{1 - \delta}}^2 + \| \Delta_{\le M} u^{\sharp} \|_{\ssp^{1 -
			\delta}}^2 \lesssim 2^{2 M (1 - \delta)} \| u \|_{L^2} + 2^{- 2 \delta M}
	\| u^{\sharp} \|_{\ssp^1}^2 \]
	and choosing $M$ large enough we can obtain that
	\[ \frac{1}{2} \langle \nabla u^{\sharp}, \nabla u^{\sharp} \rangle
	\le - \langle u, A u \rangle + C_{\Xi} \| u \|_{L^2}^2. \]
\end{proof}

Now, we are in a position to define the form domain of the operator.  We first shift the operators $A$ and $A_\varepsilon$ by a constant to obtain a positive operator.

\begin{remark}\upshape\label{rem:2deps} \upshape
	One can readily check that the preceding analysis is valid as well for the approximate Hamiltonians $A_\varepsilon,$ given by \eqref{eqn:2dHepsdef}, simply by replacing the noise $\Xi$ by its regularization $\Xi_\varepsilon.$ Moreover, since all the constants we obtain are polynomials in the $\mathscr{X}^\alpha$ norm of the noise, one sees that they can be chosen to hold uniformly in $\varepsilon,$ since $||\Xi_\varepsilon||_{\mathscr{X}^\alpha}\le||\Xi||_{\mathscr{X}^\alpha}.$  In particular, the result in Proposition \ref{lem:2dh1bound} is true for $A_\varepsilon$ and $\Xi_\varepsilon$ for the same constant $C_\Xi$.
\end{remark}

	\begin{proposition}\label{lem:laxmil}
	There exists a constant $K_\Xi$ which is independent of $\varepsilon$ s.t. 
	\begin{align}
	(K_{\Xi} - A)^{- 1} : L^2 & \rightarrow  \mathscr{D} (H) \label{shift1} \\	
	(K_{\Xi} - A_{\varepsilon})^{- 1} : L^2 & \rightarrow  \ssp^2 \label{shift2}
	\end{align}
	are bounded.
\end{proposition}

\begin{proof}
	We will prove the statement for $A$ using a generalization of Lax-Milgram, see~\cite{bab71}.  The proof for $A_\varepsilon$ follows the same lines with the same constant $K_\Xi$, in virtue of Remark \ref{rem:2deps}.

	{{Fix the constant $K_{\Xi} > C_{\Xi}  > 0$ ($C_{\Xi}$ as in \eqref{equ:2dconstant}) such that }}
	\[ \| u \|^2_{L^2} < \langle - (A - K_{\Xi}) u, u \rangle \ \forall u \in \mathscr{D}
	(A). \]
	{ which is possible by Proposition \ref{lem:2dh1bound}.}
	
	{{Define the bilinear map }}
	\begin{eqnarray*}
		B : \mathscr{D} (A) \times L^2 & \rightarrow & \mathbb{R}\\
		B (u, v) & \assign &  \langle - (A - K_{\Xi}) u, v \rangle,
	\end{eqnarray*}
	
	{{then $B$ is continuous, namely}}
	\[ | B (u, v) | \lesssim \| u \|_{\mathscr{D} (A)} \| v \|_{L^2}, \ \forall u
	\in \mathscr{D} (A), \  v \in L^2, \]
	{{and it is {{weakly coercive}} i.e.}}
	\[ \| u \|_{\mathscr{D} (A)} = \| - (A - K_{\Xi}) u \|_{L^2} = \underset{\| v
		\|_{L^2} = 1}{\sup}  \langle - (A - K_{\Xi}) u, v \rangle \ \tmop{for} \tmop{any} u \in
	\mathscr{D} (A) . \]
	{{The last property to check is that for any $0 \neq v \in L^2$,}}
	\[ \underset{\| u \|_{\mathscr{D} (A)} = 1}{\sup} | B (u, v) | > 0. \]
	{{Assume for the sake of contradiction that there is a $0 \neq v \in
			L^2$ s.t.}.}
	\[ | B (u, v) | = 0,  \quad \forall u \in \mathscr{D} (A), \]
	{{This means that}}
	\[ \langle u, v \rangle_{\mathscr{D} (A), \mathscr{D} (A)^{\ast}} = 0
	\ \tmop{for} \tmop{all} \ u \in \mathscr{D} (A), \]
	{{i.e. $v = 0 \ \tmop{in} \  \mathscr{D} (A)^{\ast} .$ But since
			$\mathscr{D} (A)$ is dense in $L^2$, this implies $v = 0$ in $L^2 $ which is a
			contradiction.}}
	{{Then the Babuska-Lax-Milgram Theorem says that for any $f \in
			(L^2)^{\ast} = L^2$ there exists a unique $u_f \in \mathscr{D} (A)$ with }}
	\[ B (u_f, v) = \langle f, v \rangle \ \tmop{for} \tmop{all} \  v \in L^2 \]
	{{with the bound $\| u_f \|_{\mathscr{D} (A)} \lesssim \| f \|_{L^2} .$
			In other words}}
	\[ (- A + K_{\Xi})^{- 1} : L^2 \rightarrow \mathscr{D}
	(A) \] is bounded.
\end{proof}

\begin{definition} \label{def:setconstant}
	We define the following shifted operators
	\begin{align*}
	H_\varepsilon & :=   A_\varepsilon - K_\Xi \\
	H & := A -  K_\Xi.  \\
	\end{align*}
	\end{definition}

\begin{remark}\upshape
We would like to point out that in the sequel the constant $K_\Xi$ can be updated to be larger, as needed, without notice.
\end{remark}

Also, we use the above  estimates to give a characterization of the domain and the form domain in terms of standard Sobolev norms of $u^\sharp$.  First, we define the form domain.

\begin{definition} \label{def:energySp2d}
	We set $u= \Gamma u^{\sharp}$, as in Lemma \eqref{lem:gamma}. The form domain of $H$, that we denote as $\ED$, is defined as the closure of the domain under the following norm
		\[ \| \Gamma u^{\sharp} \|_{\ED} \assign
	\sqrt{ \langle \Gamma u^{\sharp}, -{H} \Gamma u^{\sharp} \rangle} = \sqrt{ \langle u, -{H} u \rangle}
	. \]
\end{definition}

\begin{proposition}
	\label{lem:formdom}{\tmdummy}
	
	\begin{enumeratenumeric}
		\item $\Gamma u^{\sharp} \in \mathscr{D} ({H})
		\Leftrightarrow u^{\sharp} \in \ssp^2$,
		where $\Gamma$ is the map from Proposition \ref{lem:gamma}.
		More precisely, on $\mathscr{D}(H)$ we have the following norm equivalence
		\begin{equation} \label{equ:normequ1}
		\| u^{\sharp} \|_{\ssp^2} \lesssim \| {H} \Gamma u^{\sharp} \|_{L^2}
		\lesssim \| u^{\sharp} \|_{\ssp^2 .} \end{equation}
		\item $\Gamma u^{\sharp} \in \mathscr{D} ( \sqrt{-
			{H}}) \Leftrightarrow u^{\sharp} \in \ssp^1$,
		where the form domain of $- {H}$ is given by the closure of
		$\mathscr{D} ({H})$ under the norm
		\begin{equation} \label{equ:normequ2}
		\| \Gamma u^{\sharp} \|_{\ED} \assign
		\sqrt{ \langle \Gamma u^{\sharp}, -{H} \Gamma u^{\sharp} \rangle}
		. \end{equation}
		We will see in the following pages that the operator $-H$ is self-adjoint and positive, so this is in fact a norm.
		Then the precise statement is that on $\mathscr{D} ({H})$ the
		following norm equivalence holds
		\[ \| u^{\sharp} \|_{\ssp^1} \lesssim \| \Gamma u^{\sharp}
		\|_{\ED} \lesssim \| u^{\sharp} \|_{\ssp^1}, \]
		and hence the closures with respect to the two norms coincide.
	\end{enumeratenumeric}
\end{proposition}
\begin{proof}
	\begin{enumerate}
		\item  The first inequality in \eqref{equ:normequ1} follows directly from \eqref{equ:h2bound} and the second by first expanding using \eqref{equ:defandersonHam} and then estimating as in the proof of Theorem  \ref{lem:h2bound}.
	\item In \eqref{equ:normequ2}, the first inequality follow directly from the Proposition \ref{lem:2dh1bound}.  For the second term, one plugs in the definition \eqref{equ:defandersonHam} and then the only non-trivial term is $\langle u^\sharp \circ \xi, u^\sharp \rangle$.  For this term, we also have
	\[
	|\langle u^\sharp \circ \xi, u^\sharp \rangle| \leq C_{\Xi} ||u^\sharp||_{\mathscr{H}^{1}}^2
	\]
	by similar arguments as in proof of Proposition \ref{lem:2dh1bound}.  
\end{enumerate}
\end{proof}

	In order to show self-adjointness we would like to use the following result.

\begin{proposition}\cite[X.1]{reedsimon2} \label{prop:selfadj}
	A closed symmetric operator on a Hilbert space $\mathcal{H}$ is self-adjoint if it has at least one real value in its resolvent set.
\end{proposition}

Now, we can show self-adjointness.  

\begin{lemma}
	\label{thm:2dselfadj}The operators $H : \mathscr{D}(H)
	\rightarrow L^2$ and $H_\varepsilon : \ssp^2
	\rightarrow L^2$ as defined in Theorem \ref{lem:h2bound} are self-adjoint.
\end{lemma}

\begin{proof}
	This follows from Proposition \ref{prop:selfadj}. Observe that Proposition \ref{lem:laxmil} implies  $K_\Xi$ is in the resolvent of $A$ and $A_\varepsilon$. Therefore, the result.
\end{proof}

Next, we first show, in Theorem \ref{thm:resoventConv1},  the strong  revolvent convergence of $H_\varepsilon$ to $H$ which we will use in the PDE part.  Then in Theorem \ref{thm:normResolventMain} we prove the stronger result of norm resolvent convergence. In 2d case, this result was obtained in   \cite[Lemma 4.15]{allez_continuous_2015} but we give a proof  in our framework which can also be applied to the 3d case, namely that, generalizes the result in the cited article to the 3d case.

	\begin{theorem} \label{thm:resoventConv1}
		Recall the operators $H_\varepsilon $ and  $H$, as defined in Definition \eqref{def:setconstant}.  For any $f \in L^2$, we have
		\[
		|| H^{-1} f - H_\varepsilon^{-1} f||_{L^2} \rightarrow 0
		\]
		as $\varepsilon \rightarrow 0$. Namely, $H_\varepsilon $ converges to $H$ in the strong resolvent  sense.
	\end{theorem}
	
	\begin{proof}
		First we take an approximation as in Proposition \ref{prop:opclosed}.  Since $H_\varepsilon^{-1} : L^2 \rightarrow L^2$ is a bounded operator we have
		\[
		||u_\varepsilon  - H_\varepsilon^{-1} Hu||_{L^2} = ||H_\varepsilon^{-1} (Hu - H_\varepsilon u_\varepsilon)||_{L^2} \lesssim ||Hu - H_\varepsilon u_\varepsilon||_{L^2}
		\]
		
		Then, by Proposition \ref{prop:opApp}, we readily obtain
		\[
		||u_\varepsilon  - H_\varepsilon^{-1} Hu||_{L^2} \rightarrow 0
		\]
		as $\varepsilon \rightarrow 0$.
		Now, for $u \in \mathscr{D}(H)$ we estimate
		\[
		||u - H_\varepsilon^{-1} Hu ||_{L^2} \leq||u - u_\varepsilon  ||_{L^2} + ||u_\varepsilon  - H_\varepsilon^{-1} Hu ||_{L^2}
		\]
		and obtain that $H_\varepsilon^{-1} H : \mathscr{D}(H) \subset L^2  \rightarrow L^2$ tend to the identity operator over $\mathscr{D}(H)$.  Then, for any $f = Hu \in L^2$ we can write 
		\[
		|| H_\varepsilon^{-1} Hu  - u||_{L^2} = || H_\varepsilon^{-1} Hu  - H^{-1}Hu||_{L^2}  = || H_\varepsilon^{-1} f  - H^{-1}f||_{L^2} \rightarrow 0
		\]
		for all $f \in L^2$.  By the existence of $H^{-1}$, indeed for any $f \in L^2$ we can find such $u$.   Hence, the result.
	\end{proof}

In the next theorem, we show that in fact the above convergence can be improved to the norm resolvent convergence in the $\mathscr{H}^\gamma$-norm.

	\begin{theorem} \label{thm:normResolventMain}
We have
	\[
	|| H^{-1}  - H_\varepsilon^{-1} ||_{L^2 \rightarrow \mathscr{H}^\gamma} \rightarrow 0
	\]
	as $\varepsilon \rightarrow 0$. Namely, $H_\varepsilon $ converges to $H$  in  the  norm resolvent sense.
\end{theorem}

\begin{proof}
		Recall that $\Gamma : \mathscr{H}^2 \rightarrow \mathscr{D}(H)$ and $\Gamma_\varepsilon : \mathscr{H}^2 \rightarrow \mathscr{H}^2$  in which case we have $\Gamma^{-1} : \mathscr{D}(H) \rightarrow H^2$ and $\Gamma_\varepsilon : \mathscr{H}^2 \rightarrow \mathscr{H}^2$ .
Recall that in Proposition \ref{prop:opApp} we obtained 
	\[
	|| H_\varepsilon \Gamma_\varepsilon- H\Gamma ||_{\mathscr{H}^2 \rightarrow L^2} \rightarrow 0.
	\]
	 This implies the   norm resolvent convergence  
	\[
	|| \Gamma_\varepsilon^{-1} H_\varepsilon^{-1}  - \Gamma^{-1}  H^{-1} ||_{L^2 \rightarrow \mathscr{H}^2}  \rightarrow 0.
	\]
	
	To conclude, by using Proposition \ref{lem:gamma} we can write the estimate
	\[
	|| \Gamma \Gamma_\varepsilon^{-1} H_\varepsilon^{-1} - \Gamma \Gamma^{-1}  H^{-1} ||_{\mathscr{H}^\gamma} \leq 3|| \Gamma_\varepsilon^{-1} H_\varepsilon^{-1} - \Gamma^{-1}  H^{-1} ||_{\mathscr{H}^\gamma}
	\]
	
	where, as $\varepsilon \rightarrow 0$ by Lemma \ref{lem:identityGammaconv}, we get the convergence
	
	\[
	||  H_\varepsilon^{-1}  - H^{-1} ||_{L^2 \rightarrow \mathscr{H}^\gamma} \rightarrow 0.
	\]
	
Hence, the result.
\end{proof}


As a corollary of the norm resolvent convergence we note following observation.
	
		\begin{corollary}[cfr.~\cite{reedsimon1}, VIII.20]
			For any bounded continuous function $f:[-C_\Xi,\infty)\to\mathbb{C}$ we get  
			\[ f(H_{\varepsilon}) g \to f(H) g \text{ in }L^2  \]
			for any $g\in L^2$ i.e. strong operator convergence. 
		\end{corollary}

\subsubsection{Functional inequalities}

In this section, we obtain certain inequalities for the Anderson Hamiltonian which will be crucial when we study the PDEs.

The first one is an $L^p$-embedding result.

\begin{lemma}[$L^p$ estimates]\label{lem:2dlpbounds}
	For $u\in \ED$ and $p\in[1,\infty)$ we have
	\begin{align}
	||u||_{L^p}\lesssim_\Xi ||u||_{\ED}.
	\end{align}
	Moreover, for $v\in \mathscr{D}(\sqrt{ - H_\varepsilon})=\ssp^1,$ we have
	\begin{align}
	||v||_{L^p}\lesssim_\Xi ||u||_{\mathscr{D}(\sqrt{ - H_\varepsilon})}.
	\end{align}
\end{lemma}
\begin{proof}
	For $p<\infty$ and $\delta(p)>0$ small enough we have by Sobolev embedding and Propositions \ref{lem:gamma} and \ref{lem:formdom}
	\begin{align*}
	||u||_{L^p}\lesssim||u||_{\ssp^{1-\delta}}\lesssim||u^\sharp||_{\ssp^{1-\delta}}\lesssim||u^\sharp||_{\ssp^{1}}\lesssim_\Xi||(-H)^{1/2}u||_{L^2}\lesssim_\Xi ||u||_{\ED}.
	\end{align*}
	and by Remark \ref{rem:2deps}, the same computation works for the second inequality with constants independent of $\varepsilon.$
\end{proof}

In light of the Proposition \ref{lem:formdom}, the following result is an analogue of the embedding $\mathscr{H}^2 \subset L^\infty$ in 2d.
\begin{lemma} \label{lem:embedinfty}
For $u \in \mathscr{D}(H)$ we have
\[
||u||_{L^\infty} \lesssim ||Hu||_{L^2}
\]
\end{lemma}
\begin{proof}
By using $\mathscr{H}^2 \subset L^\infty$ and  Propositions \ref{lem:gamma} and \ref{lem:formdom} we have the following chain of inequalities:
\[
||u||_{L^\infty} \lesssim_\Xi ||u^\sharp||_{L^\infty}  \lesssim_\Xi  ||u^\sharp||_{\mathscr{H}^2} \lesssim_\Xi  ||Hu||_{L^2}.
\]
Hence, the result.
\end{proof}
In addition to the above result,  we can also prove an inequality that, in some sense, interpolates the  $L^\infty$-norm between the energy norm and the logarithm of the domain norm.  Namely, we  prove a version of Brezis-Gallouet inequality for the Anderson Hamiltonian.  We first recall below the original version of the inequality.

\begin{theorem} \cite{brga80}
	Let $\Omega$ be a domain in $\mathbb{R}^2$ with compact smooth boundary.  Then, for $v \in \mathscr{H}^2(\Omega)$ we have
	\begin{align*}
	||v||_{L^\infty}\lesssim C \left( 1+ \sqrt{1+\log||v||_{\ssp^2}}\right).
	\end{align*}
	for every $v$ that satisfies $||v||_{\mathscr{H}^1(\Omega)} \leq 1$.
\end{theorem}

Our version for the Anderson Hamiltonian is as follows.

\begin{theorem}\label{lem:brgaineq}

		For $v \in \mathscr{D}(H)$ we have
		\[ \| v \|_{L^{\infty}} \lesssim_\Xi||v||_{\ED}  \sqrt{(1 + \log (1+ \|v \|_{\mathscr{D}(H)}))}    . \]
	
	As a corollary, we obtain, for $v \in \mathscr{D}(H_\varepsilon)=\ssp^2$, 
		\[ \| v \|_{L^{\infty}} \lesssim_\Xi||(-H_\varepsilon)^{1/2}v||_{L^2}  \sqrt{(1 + \log (1+  \|H_\varepsilon v \|_{L^2}))}    , \]
		where the constant depends on the limiting noise $\Xi$ and can be chosen independently of $\varepsilon$.

\end{theorem}

\begin{proof}

		We first observe
		\[ \| v \|_{L^{\infty}}^2 \le 2 \| \Delta_{\le M} v
		\|_{L^{\infty}}^2 + 2 \| \Delta_{> M} v \|_{L^{\infty}}^2. \]
		By Bernstein inequalities, Lemma \ref{lem:bernstein} (in $d = 2$), we can write
		\[ \| \Delta_{\le M} v \|_{L^{\infty}} \le \sum_{i = - 1}^{2^M} \|
		\Delta_i v \|_{L^{\infty}} \lesssim \sum_{i = - 1}^{2^M} \| \Delta_i v
		\|_{L^{\infty}} \lesssim \sum_{i = - 1}^{2^M} 2^i \| \Delta_i v
		\|_{L^2} \lesssim 2^{M/2} \| v \|_{\ssp^1}. \]
		Moreover one can show that for any $\delta > 0$ we have
		\[ \| \Delta_{> M} v \|_{L^{\infty}}  \lesssim \| \Delta_{> M} v \|_{\ssp^{1 +
				\delta}} \lesssim 2^{2^M (\delta - 1)} \| v \|_{\ssp^2} \]
		so
		\[ \| v \|_{L^{\infty}}^2 \le  2^{M+1} \| v \|^2_{\ssp^1} + 2^{2 M
			(\delta - 1)} \| v \|^2_{\ssp^2} \leq 2 M \| v \|^2_{\ssp^1} + 2^{2 M
			(\delta - 1)} (1+\| v \|)^2_{\ssp^2}. \]
		Choosing $2^{2^M (\delta - 1)} (1+ \| v \|_{\ssp^2}) = 1^{}$ we reobtain the usual
		Brezis--Gallouet inequality, in the following form
			\begin{align*}
		||v||_{L^\infty}\lesssim ||v||_{\ssp^1} \sqrt{1+\log(1+||v||_{\ssp^2})}.
		\end{align*}
		By using this and  Propositions \ref{lem:2dh1bound}, and \ref{lem:gamma}, we obtain 
		\begin{align*}
		||v||_{L^\infty}&\lesssim_\Xi||v^\sharp||_{L^\infty}
		\lesssim_\Xi ||v^\sharp||_{\ssp^1}\sqrt{1+\log(1+||v^\sharp||_{\ssp^2})} \\
		&\lesssim_\Xi||v||_{\ED}\sqrt{1+\log(1+||v||_{\mathscr{D}(H)})}
		\end{align*}
		By Remark \ref{rem:2deps}, the same estimates as for $\mathscr{D}(H)$ are also true for $\mathscr{D}(H_\varepsilon)$, in particular the estimates in Proposition \ref{lem:2dh1bound} hold with constants independent of $\varepsilon$. 
	Hence, the result.
\end{proof}

\subsection{The three-dimensional case}\label{section:3dham}

In this section, we study the Anderson Hamiltonian in 3d. As in the 2d
case we will perform a paracontrolled analysis of the Anderson Hamiltonian.
This case is more technical since the noise term  has  much lower H\"older
regularity of $\CC^{- 3 / 2 -}$. So, the paracontrolled ansatz as in the 2-d case turns out to be insufficient.  We follow a two step procedure for the defining the operator. As a first step, similar to
{\cite{debussche2016schr}}, we perform an exponential transformation depending
on the noise and as a second step we make an Ansatz for the transformed
operator using paracontrolled distributions.

\subsubsection{Enhanced noise in 3d}\label{sec:3dren}

Recall that in the 2d case we needed to define the space of enhanced noise (see Def. \ref{def:2dencNoise}), namely $\mathscr{X}^\alpha,$ for the renormalization.  In this section, we first define this space in the 3d case. 

In the results below, we prove that $X = (- \Delta)^{-1} \xi $ can be lifted to an
element $\Xi$ in the space $\mathscr{X}^{\alpha}$ of enhanced distributions
such that all the stochastic terms, we will need for the ansatz in the next section,  exist with  correct regularities.  In the following sequence of results, we construct the enhanced white noise space in 3d and prove the related approximation results.  In particular, we  show that the lifts $\Xi_{\varepsilon} $ (of the regularized noise $\xi_{\varepsilon}$) converges to an element, that we denote by $\Xi$, in $\mathscr{X}^{\alpha}.$ 

\begin{definition}\label{def:3dnoise}
	For $0 < \alpha < \frac{1}{2},$ we define the space
	$\mathscr{X}^{\alpha}$ to be the closure of the set
	\[ \left\{ \left( \phi, \phi_a^{\zzone}, \phi^{\zztwo}, \phi^{\zzthree},
	\phi_b^{\zzfour}, \nabla \phi \circ \nabla \phi^{\zzthree} \right) : (a,
	b) \in \mathbbm{R}^2, \phi \in \CC^2 (\mathbbm{T}^3) \right\} \]
	with respect to the $\CC^{\alpha} (\mathbbm{T}^3) \times
	\CC^{2 \alpha} (\mathbbm{T}^3) \times \CC^{\alpha + 1}
	(\mathbbm{T}^3) \times \CC^{\alpha + 1} (\mathbbm{T}^3) \times
	\CC^{4 \alpha} (\mathbbm{T}^3) \times \CC^{2 \alpha - 1}
	(\mathbbm{T}^3)$ norm. Here, we defined
	\begin{eqnarray*}
		\phi^{\zzone}_a & : = & (1 - \Delta)^{- 1} (| \nabla \phi |^2 - a)\\
		\phi^{\zztwo} & : = & 2 (1 - \Delta)^{- 1} ( \nabla \phi \cdummy
		\nabla \phi_a^{\zzone} )\\
		\phi^{\zzthree} & : = & (1 - \Delta)^{- 1} ( \nabla \phi \cdummy
		\nabla \phi^{\zztwo} )\\
		\phi_b^{\zzfour} & : = & (1 - \Delta)^{- 1} ( | \nabla
		\phi_a^{\zzone} |^2 - b ) .
	\end{eqnarray*}
\end{definition}

\begin{theorem}
	\label{thm:3dren}For $\xi_{\varepsilon}$ given by (\ref{eqn:xieps}) we
	define
	\begin{eqnarray*}
		X_{\varepsilon} & = & (- \Delta)^{- 1} \xi_{\varepsilon}\\
		X^{\zzone}_{\varepsilon} & = & (1 - \Delta)^{- 1} (| \nabla
		X_{\varepsilon} |^2 - c^1_{\varepsilon})\\
		X^{\zztwo}_{\varepsilon} & = & 2 (1 - \Delta)^{- 1} ( \nabla
		X_{\varepsilon} \cdummy \nabla X_{\varepsilon}^{\zzone} )\\
		X^{\zzthree}_{\varepsilon} & = & (1 - \Delta)^{- 1} ( \nabla
		X_{\varepsilon} \cdot \nabla X^{\zztwo}_{\varepsilon} )\\
		X^{\zzfour}_{\varepsilon} & = & (1 - \Delta)^{- 1} ( | \nabla
		X_{\varepsilon}^{\zzone} |^2 - c^2_{\varepsilon} ),
	\end{eqnarray*}
	where the $c_{\varepsilon} $ are diverging constants which can be chosen as
	\begin{align*} c^1_{\varepsilon} &= \underset{k \in \mathbbm{Z}^3 \backslash \{ 0
		\}}{\sum} \frac{| m (\varepsilon k) |^2}{| k |^2} \sim
	\frac{1}{\varepsilon} \\ c^2_{\varepsilon} &= \underset{k_1, k_2
		\neq 0}{\sum} | m (\varepsilon k_1) |^2 | m (\varepsilon k_2) |^2 \frac{|
		k_1 \cdummy k_2 |}{| k_1 - k_2 |^2 | k_1 |^4 | k_2 |^2} \sim \left( \log
	\frac{1}{\varepsilon} \right)^2 . \end{align*}
	Then the sequence $\Xi^{\varepsilon} \in \mathscr{X}^{\alpha},$ given by
	\[ \Xi_{}^{\varepsilon} : = ( X_{\varepsilon},
	X^{\zzone}_{\varepsilon}, X^{\zztwo}_{\varepsilon},
	X_{\varepsilon}^{\zzthree}, X^{\zzfour}_{\varepsilon}, \nabla
	X_{\varepsilon} \circ \nabla X_{\varepsilon}^{\zzthree} ) \]
	converges a.s. to a unique limit $\Xi \in \mathscr{X}^{\alpha} $ , given by
	\begin{equation} \label{equ:enhancednoise3d}
	\Xi_{}^{} : = ( X, X^{\zzone}, X^{\zztwo}, X^{\zzthree},
	X^{\zzfour}, \nabla X \circ \nabla X^{\zzthree} ), 
	\end{equation}
	where
	\begin{eqnarray*}
		X_{} & = & (- \Delta)^{- 1} \xi\\
		X^{\zzone}_{} & = & (1 - \Delta)^{- 1} (: | \nabla X | :^2)\\
		X^{\zztwo}_{} & = & 2 (1 - \Delta)^{- 1} ( \nabla X_{} \cdummy \nabla
		X_{}^{\zzone} )\\
		X^{\zzthree}_{} & = & (1 - \Delta)^{- 1} ( \nabla X \cdummy \nabla
		X^{\zztwo}_{} )\\
		X^{\zzfour} & = & (1 - \Delta)^{- 1} ( : | \nabla X^{\zzone}
		| :^2 ) .
	\end{eqnarray*}
\end{theorem}
\begin{proof}
	We omit the parts of the proof which goes in a similar way to Theorem 7.11 in {\cite{cannizzaro2015}} (see also Chapter 9 of
	{\cite{gubinelli2017kpz}}).
	Note that their estimates are for the parabolic case, but by using the
	resolvent identity
	\[ \int_{0}^{\infty} e^{- t} e^{t \Delta}
	\mathd t = (1 - \Delta)^{- 1}, \]
	one can easily adapt their computations to our setting, essentially through
	multiplying by $e^{- t}$ and integrating over $t.$ This, in particular,
	implies that the diverging constants are the same.
	Note that the last term in our enhanced noise \eqref{equ:enhancednoise3d} is slightly different from
	the one in {\cite{cannizzaro2015}}.  However one can easily show that the
	most singular part of $\nabla X \circ \nabla X^{\zzthree}$ is given by
	$\nabla X \circ \nabla (1 - \Delta)^{- 1} \nabla X,$ which is the term
	from {\cite{cannizzaro2015}}. In fact, we have
	\begin{align*}
	\nabla X \circ \nabla X^{\zzthree}  = & \nabla X \circ \nabla (1 -
	\Delta)^{- 1} \left( \nabla X \prec \nabla X^{\zztwo} + \nabla
	X^{\zztwo} \prec \nabla X + \nabla X \circ \nabla X^{\zztwo} \right)\\
	= & \nabla X \circ (1 - \Delta)^{- 1} \left( \nabla \left( \nabla X
	\prec \nabla X^{\zztwo} + \nabla X \circ \nabla X^{\zztwo} \right)+ 
	\nabla^2 X^{\zztwo} \prec \nabla X \right) \\&+ \nabla X \circ (1 -
	\Delta)^{- 1} \left( \nabla X^{\zztwo} \prec \nabla^2 X \right),
	\end{align*}
	where first expression make sense assuming the correct regularity
	for the other stochastic terms. For the second term,  we apply the
	commutator  Lemma \ref{lem:comm2} (or more precisely the H\"older
	version) and Proposition \ref{prop:commu}.  We  compute
	\begin{align*}
	\nabla X \circ (1 - \Delta)^{- 1}  & \left(  \nabla X^{\zztwo} \prec
	\nabla^2 X \right)  =  \nabla X \circ \left( \nabla X^{\zztwo} \prec
	(1 - \Delta)^{- 1} \nabla^2 X + R \left( \nabla X^{\zztwo}, \nabla^2 X
	\right) \right)\\
	&  =  \nabla X^{\zztwo} (\nabla X \circ \nabla (1 - \Delta)^{- 1}
	\nabla X) + C \left( \nabla X^{\zztwo}, (1 - \Delta)^{- 1} \nabla^2 X,
	\nabla X \right)\\ &  + \nabla X \circ R \left( \nabla X^{\zztwo}, \nabla^2 X
	\right) .
	\end{align*}
	This proves that $\nabla X \circ \nabla (1 - \Delta)^{- 1} \nabla X \in
	\CC^{2 \alpha - 1}$ which in turn implies that $\nabla X \circ \nabla X^{\zzthree}
	\in \CC^{2 \alpha - 1} .$ Thus our result follows from Theorem
	7.11 in {\cite{cannizzaro2015}}.\\				
	See also Theorems 9.1 and 9.3 in {\cite{gubinelli2017kpz}} where a
	similar renormalization was performed with 1d space-time
	white noise with regularity that of the 3d spatial
	white noise.
\end{proof}
\begin{lemma}
	Let $\alpha, X, X^{\zzone}, X^{\zztwo}, X_{\varepsilon},
	X_{\varepsilon}^{\zzone}, X_{\varepsilon}^{\zztwo} $ be as above, then $e^X
	\in \CC^{\alpha}, e^{X^{\zzone}} \in \CC^{2 \alpha},
	e^{X^{\zztwo}} \in \CC^{\alpha + 1}$ and
	\begin{align*}
	e^{X_{\varepsilon}} & \rightarrow  e^X \text{ in }
	\CC^{\alpha}\\
	e^{X^{\zzone}_\varepsilon} & \rightarrow  e^{X^{\zzone}} \text{ in }
	\CC^{2 \alpha}\\
	e^{X_\varepsilon^{\zztwo}} & \rightarrow  e^{X^{\zztwo}} \text{ in }
	\CC^{\alpha + 1} .
	\end{align*}
	\begin{proof}
		{\tmem{We prove the result for $X,$ the others are proved in the 
				same way. Since $\alpha > 0$, we use the equivalent classical H{\"o}lder
				norms on $\CC^{\alpha} .$ One easily sees that the spaces
				$\CC^{\alpha}$ are Banach Algebras, so $e^X = \underset{n
					\ge 0}{\sum} \frac{1}{n!} X^n  \in  \CC^{\alpha}$ and since
				$X_{\varepsilon} \rightarrow X \ \tmop{in} \  \CC^{\alpha},$ we can
				estimate}}
		\begin{eqnarray*}
			\| e^X - e^{X_{\varepsilon}} \|_{\mathscr{C}^\alpha} & \le & \| e^X
			\|_{\mathscr{C}^\alpha} \| 1 - e^{X_{\varepsilon} - X} \|_{\mathscr{C}^\alpha}
			=  \| e^X \|_{\mathscr{C}^\alpha} \left\| \underset{n \ge 1}{\sum}
			\frac{1}{n!} (X_{\varepsilon} - X)^n \right\|_{\mathscr{C}^\alpha}\\
			& \le & \| e^X \|_{\mathscr{C}^\alpha} (e^{\| X_{\varepsilon} - X
				\|_{\mathscr{C}^\alpha}} - 1),
		\end{eqnarray*}
		{\tmem{and conclude that $e^{X_{\varepsilon}} \rightarrow e^X$ in
				$\CC^{\alpha} .$}}
	\end{proof}
\end{lemma}

\begin{lemma}\label{lem:expprod}
	For $\alpha, X, X^{\zztwo}$ as above, $W  :=  X + X^{\zzone} + X^{\zztwo}$ and
	\begin{eqnarray*}
Z & = & (1 - \Delta)^{- 1} \left( \left| \nabla X^{\zztwo} \right|^2 + 2
		\nabla X^{\zzone} \cdummy \nabla X^{\zztwo} - X^{\zzone} - X^{\zztwo}
		\right) + X^{\zzfour} + 2 X^{\zzthree}.
	\end{eqnarray*}
	 we have
	\[ \nabla e^X \cdummy \nabla e^{X^{\zztwo}} \in \CC^{\alpha - 1},
	\]
	which implies that $e^{2 W} (1 - \Delta) Z \in \CC^{\alpha - 1} .$
	
	\begin{proof}
		{\tmem{We use paralinearisatio{\tmem{}}n, see Lemma \ref{lem:paralin}, to
				rewrite}}
		\begin{eqnarray*}
			e^X & = & e^X \prec X + g^{\sharp}\\
			e^{X^{\zztwo}} & = & e^{X^{\zztwo}} \prec X^{\zztwo} + f^{\sharp},
		\end{eqnarray*}
		{\tmem{where $g^{\sharp} \in \CC^{2 \alpha} \ \tmop{and} \  f^{\sharp} \in \CC^{2 \alpha + 1}
				.$ Thus,}}
		\begin{eqnarray*}
			\nabla e^X & = & (\nabla e^X) \prec X + e^X \prec \nabla X + \nabla
			g^{\sharp}\\
			& \tmop{and} & \\
			\nabla e^{X^{\zztwo}} & = & \left( \nabla e^{X^{\zztwo}} \right) \prec
			X^{\zztwo} + e^{X^{\zztwo}} \prec \nabla X^{\zztwo} + \nabla f^{\sharp}
			.
		\end{eqnarray*}
		{\tmem{Note that the only problematic term in the product is }}
		
		\begin{equation} \label{equ:enhancednoise3d:equ1}
	(e^X \prec
		\nabla X) ( e^{X^{\zztwo}} \prec \nabla X^{\zztwo} ).	
		\end{equation}
		{\tmem{More precisely, we only have to make sense of the resonant product in \eqref{equ:enhancednoise3d:equ1}.
				Since the paraproducts are always defined. We compute}}
		\begin{align*}
		(e^X \prec \nabla X) \circ \left( e^{X^{\zztwo}} \prec \nabla X^{\zztwo}
		\right)  = & e^{X^{\zztwo}} \left( \nabla X^{\zztwo} \circ (e^X \prec
		\nabla X) \right) + C \left( e^{X^{\zztwo}}, \nabla X^{\zztwo}, (e^X
		\prec \nabla X) \right)\\
		= & e^{X^{\zztwo} + X} \left( \nabla X^{\zztwo} \circ \nabla X
		\right) + e^{X^{\zztwo}} C \left( e^X, \nabla X, \nabla X^{\zztwo}
		\right) \\&+ C \left( e^{X^{\zztwo}}, \nabla X^{\zztwo}, (e^X \prec \nabla
		X) \right) .
		\end{align*}
		{\tmem{Now, since $\nabla X^{\zztwo} \circ \nabla X$ is assumed to be in
				$\CC^{2 \alpha - 1},$ the above resonant product is also in
				$\CC^{2 \alpha - 1} .$ This finishes the proof that }}$\nabla e^X
		\cdummy \nabla e^{X^{\zztwo}} \in \CC^{\alpha - 1} .${\tmem{
				Moreover, by reinserting the definitions we obtain}}
		\begin{align*}
		& e^{2 W} (1 - \Delta) Z \\  &=  e^{2 X + 2 X^{\zzone} + 2 X^{\zztwo}}
		\left( : \left| \nabla X^{\zzone} \right|^2 : + \left| \nabla X^{\zztwo}
		\right|^2 + 2 \nabla X \cdummy \nabla X^{\zztwo} + 2 \nabla X^{\zzone}
		\cdummy \nabla X^{\zztwo} - X^{\zzone} - X^{\zztwo} \right)\\
		&=  e^{2 X + 2 X^{\zzone} + 2 X^{\zztwo}} \left( : \left| \nabla
		X^{\zzone} \right|^2 : + \left| \nabla X^{\zztwo} \right|^2 + 2 \nabla
		X^{\zzone} \cdummy \nabla X^{\zztwo} - X^{\zzone} - X^{\zztwo} \right) \\&+
		\frac{1}{2} e^{2 X^{\zzone}} \nabla \left(e^{2 X}\right) \nabla \left( e^{2
			X^{\zztwo}} \right)
		\end{align*}
		{\tmem{by noting the previous computations and the fact that all the terms in the
				first bracket have regularity at least $2 \alpha - 1$.  We can finally conclude
				that
				\[ \| e^{2 W} (1 - \Delta) Z \|_{\CC^{\alpha - 1}} \lesssim \|
				e^{2 W} \|_{\CC^{\alpha}} \| \Xi \|_{\mathscr{X}^{\alpha}} . \]}}
	\end{proof}
\end{lemma}

\subsubsection{The domain, the $\Gamma$-map and the definition of the 3-d Hamiltonian} \label{sec:domainDef3d}

In this section, building on our work in Section \ref{sec:3dren}, we perform the renormalization of the Anderson Hamiltonian in 3d.  Recall the  following quantities we introduced and justified in  Section \ref{sec:3dren}.
\begin{align*}
X &= (- \Delta)^{- 1} \xi (x)\in\CC^{1/2-} \\ X^{\zzone} &= (1 - \Delta)^{- 1} : | \nabla X
|^2 :\in\CC^{1-} && X^{\zztwo} = 2 (1 - \Delta)^{- 1} \left( \nabla X \cdot \nabla
X^{\zzone} \right)\in\CC^{3/2-}\\
X^{\zzthree} &= (1 - \Delta)^{- 1} \left( \nabla X \cdummy
\nabla X^{\zztwo} \right)\in\CC^{3/2-} && X^{\zzfour} = (1 - \Delta)^{- 1}_{} : \left|
\nabla X^{\zzone} \right|^2 :\in\CC^{2-}.
\end{align*}

In the following we first  motivate the ansatz, that we will use in 3d, through informal calculations and then conclude formally in Definition \ref{def:3dham}.   

Initially, we make the
following Ansatz for the domain of the Hamiltonian
\[ u = e^{X + X^{\zzone} + X^{\zztwo}} u^{\flat}, \]
where the form of $u^{\flat}$ will be specified later. We begin by computing
\begin{align*}
\Delta u + u \xi  = & e^{X + X^{\zzone} + X^{\zztwo}}\Big( \Delta \left(
X + X^{\zzone} + X^{\zztwo} \right) u^{\flat} + \left| \nabla \left( X +
X^{\zzone} + X^{\zztwo} \right)^{} \right|^2 u^{\flat} \\&+ \Delta u^{\flat} +
2 \nabla \left( X + X^{\zzone} + X^{\zztwo} \right) \nabla u^{\flat} +
u^{\flat} \xi \Big)\\
= & e^{X + X^{\zzone} + X^{\zztwo}} \Big( \Delta u^{\flat} + \left( |
\nabla X |^2 - : | \nabla X |^2 :^{} + \left| \nabla X^{\zzone} \right|^2 +
\left| \nabla X^{\zztwo} \right|^2 \right.  \\&+ \left. 2 \nabla X \cdummy \nabla X^{\zztwo} +
2 \nabla X^{\zzone} \cdummy \nabla X^{\zztwo} - X^{\zzone} - X^{\zztwo}
\right) u^{\flat} + 2 \nabla \left( X + X^{\zzone} + X^{\zztwo} \right)
\cdummy \nabla u^{\flat} \Big).
\end{align*}
Note that the regularity of $X^{\zzone}$ is too low for the term $\left|
\nabla X^{\zzone} \right|^2$ to be defined so we have to replace it by by its
Wick ordered version, also note the appearing difference $| \nabla X |^2 - : |
\nabla X |^2 : .$ Here one sees the two divergences that arise, since we formally have
\begin{eqnarray*}
	: | \nabla X |^2 : & = & | \nabla X |^2 - \infty,\qquad
	: \left| \nabla X^{\zzone} \right|^2 :  =  \left| \nabla X^{\zzone}
	\right|^2 - \infty.
\end{eqnarray*}
However this notation is misleading since the rate of divergence is different in both cases, as we calculated as constants $c_1^\varepsilon$ and $c_2^\varepsilon$ in Theorem \ref{thm:3dren}.
This again suggests that, as in 2d, the renormalized Hamiltonian
can be formally written in the suggestive form
\[ A = \Delta + \xi - \infty . \]
We set
\begin{equation}
A u = A (e^W u^{\flat}) = e^W (\Delta u^{\flat} + 2 (1 - \Delta) \widetilde{W}
\cdummy \nabla u^{\flat} + (1 - \Delta) Z u^{\flat}),
\label{eqn:3dHuflat}
\end{equation}
for functions $u^{\flat} $ that  this auto-consistently makes sense with the form of $u^{\flat} $.  For simplicity,  we put
\begin{eqnarray*}
	W & = & X + X^{\zzone} + X^{\zztwo}\\
	\widetilde{W} & = & (1 - \Delta)^{- 1} \nabla W\\
	Z & = & (1 - \Delta)^{- 1} \left( \left| \nabla X^{\zztwo} \right|^2 + 2
	\nabla X^{\zzone} \cdummy \nabla X^{\zztwo} - X^{\zzone} - X^{\zztwo}
	\right) + X^{\zzfour} + 2 X^{\zzthree}.
\end{eqnarray*}
As we have seen in section \ref{sec:3dren}, these stochastic terms have the
following regularities
\[ X, W \in \CC^{1 / 2 -}, X^{\zzone} \in \CC^{1 -},
X^{\zztwo}, X^{\zzthree}, \widetilde{W}, Z \in \CC^{3 / 2 -} \ \tmop{and} \ 
X^{\zzfour} \in \CC^{2 -} . \]
This suggests to make a paracontrolled ansatz for $u^{\flat}$ in terms of $Z$
and $\widetilde{W}$ since the products appearing are classically ill-defined. In
fact, we make the following ansatz
\begin{equation}\label{equ:ubetaansatz}
 u^{\flat} = u^{\flat} \prec Z + \nabla u^{\flat} \prec \widetilde{W} + B_{\Xi}
(u^{\flat}) + u^{\sharp}, 
\end{equation}
with $u^{\sharp} \in \ssp^2$ and for a correction term that we denote by $B_{\Xi} (u^{\flat})$. To the correction term, we will absorb the terms which has regularity not worse than $\ssp^{2 -}$. Similar to the 2d case, we will  introduce a frequency cut-off  that will allow us to write $u^{\flat}$ as a function of $u^{\sharp}$
but for notational brevity, we will omit this for the time being. 

For the remainder of this section, we define
\[ L \assign (1 - \Delta) \quad \text{and} \quad  L^{- 1} = (1 - \Delta)^{- 1} \text{} . \]
Note that the  ansatz \eqref{equ:ubetaansatz} directly implies $u^{\flat} \in \ssp^{3 / 2 -}$ by the
paraproduct estimates in Lemma \ref{lem:paraest}. 

We want to determine the form of the correction term $B_{\Xi}
(u^{\flat})$ in \eqref{equ:ubetaansatz}.  We first  compute
\begin{align*}
\Delta u^{\flat}  = & \Delta u^{\flat} \prec Z + 2 \nabla u^{\flat} \prec
\nabla Z + u^{\flat} \prec \Delta Z + \nabla \Delta u^{\flat} \prec
\widetilde{W} + 2 \nabla^2 u^{\flat} \prec \nabla \widetilde{W} \\&+ \nabla u^{\flat}
\prec \Delta \widetilde{W} + \Delta B_{\Xi} + \Delta u^{\sharp}\\
= & \Delta u^{\flat} \prec Z + 2 \nabla u^{\flat} \prec \nabla Z -
u^{\flat} \prec (  L Z-Z) + \nabla \Delta u^{\flat} \prec \widetilde{W} + 2
\nabla^2 u^{\flat} \prec \nabla \widetilde{W}\\& - \nabla u^{\flat} \prec
(  L \widetilde{W}-\widetilde{W}) - L B_{\Xi} (u^{\flat}) - B_{\Xi} (u^{\flat}) +
\Delta u^{\sharp}.
\end{align*}
By using the paraproduct decomposition we obtain
\begin{eqnarray}
\Delta u^{\flat} + 2 L \widetilde{W} \cdummy \nabla u^{\flat} + L
Zu^{\flat}  =  \Delta u^{\sharp} + \widetilde{G} (u^{\flat}) + 2 L
\widetilde{W} \circ \nabla u^{\flat} + L Z \circ u^{\flat}, \label{eqn:gtilde} 
\end{eqnarray}
where we have defined
\begin{align*}
\widetilde{G} (u^{\flat}) \assign& \Delta u^{\flat} \prec Z + 2 \nabla
u^{\flat} \prec \nabla Z + u^{\flat} \prec Z + \nabla \Delta u^{\flat}
\prec \widetilde{W} + 2 \nabla^2 u^{\flat} \prec \nabla \widetilde{W} + \nabla
u^{\flat} \prec \widetilde{W} \\&- L B_{\Xi} (u^{\flat}) - B_{\Xi} (u^{\flat}) + 2
L \widetilde{W} \prec \nabla u^{\flat} + L Z \prec u^{\flat}.
\end{align*}
These are the ``non-problematic'' terms that can also be absorbed into $B_{\Xi}$. We still have to take care of the resonant
product $L \widetilde{W} \circ \nabla u^{\flat}$, which is not a priori defined
and the other resonant product which is actually defined as is, but we shall
see at a later time that it is necessary to decompose it further. To be
precise, we insert the ansatz and use Proposition \ref{prop:commu}
\begin{align*}
L \widetilde{W} \circ \nabla u^{\flat}  = & L \widetilde{W} \circ (\nabla u^{\flat}
\prec Z + u^{\flat} \prec \nabla Z + \nabla^2 u^{\flat} \prec \widetilde{W} +
\nabla^{} u^{\flat} \prec \nabla \widetilde{W} + \nabla B_{\Xi} (u^{\flat}) +
\nabla u^{\sharp})\\
= & \nabla u^{\flat} (L \widetilde{W} \circ Z) + C (\nabla u^{\flat}, Z, L
\widetilde{W}) + u^{\flat} (L \widetilde{W} \circ \nabla Z) + C (u^{\flat}, \nabla
Z, L \widetilde{W})\\ &+ L \widetilde{W} \circ (\nabla^2 u^{\flat} \prec \widetilde{W}) +
\nabla u^{\flat} (L \widetilde{W} \circ \nabla \widetilde{W})\\ &+ C (\nabla u^{\flat},
\nabla \widetilde{W}, L \widetilde{W}) + L \widetilde{W} \circ (\nabla B_{\Xi}
(u^{\flat}) + \nabla u^{\sharp}) .
\end{align*}
In section \ref{sec:3dren}, we have seen that the following stochastic terms can be defined  and have regularity
\begin{eqnarray*}
	L \widetilde{W} \circ Z & \in & \CC^{1 -}\\
	L \widetilde{W} \circ \nabla Z & \in & \CC^{0 -}\\
	L \widetilde{W} \circ \nabla \widetilde{W} & \in & \CC^{0 -} .
\end{eqnarray*}
We furthermore expand the products appearing above as
\begin{align*}
& L \widetilde{W} \circ \nabla u^{\flat} \\  &=  \nabla u^{\flat} \prec (L \widetilde{W}
\circ Z) + \nabla u^{\flat} \succcurlyeq (L \widetilde{W} \circ Z) + C (\nabla
u^{\flat}, Z, L \widetilde{W}) + u^{\flat} \prec (L \widetilde{W} \circ \nabla Z)\\ &+
u^{\flat} \succcurlyeq (L \widetilde{W} \circ \nabla Z) + C (u^{\flat}, \nabla
Z, L \widetilde{W}) + L \widetilde{W} \circ (\nabla^2 u^{\flat} \prec \widetilde{W}) +
\nabla u^{\flat} \prec (L \widetilde{W} \circ \nabla \widetilde{W}) \\ &+ \nabla
u^{\flat} \succcurlyeq (L \widetilde{W} \circ \nabla \widetilde{W})+ C (\nabla
u^{\flat}, \nabla \widetilde{W}, L \widetilde{W}) + L \widetilde{W} \circ (\nabla
B_{\Xi} (u^{\flat}) + \nabla u^{\sharp}) .
\end{align*}
For the other resonant product in (\ref{eqn:gtilde}), we do the same and get
\begin{align*}
L Z \circ u^{\flat}  = & u^{\flat} \prec (L Z \circ Z) + u^{\flat}
\succcurlyeq (L Z \circ Z) + C (u^{\flat}, Z, L Z) + \nabla^{} u^{\flat}
\prec (L Z \circ \widetilde{W})\\& + \nabla u^{\flat} \succcurlyeq (L Z \circ \widetilde{W})
+ C (\nabla u^{\flat}, \widetilde{W}, L Z) + L Z \circ (B_{\Xi} (u^{\flat}) +
u^{\sharp}) .
\end{align*}
Now we are in a position to give the precise definition of the correction term, we put
\begin{equation}\label{eqn:3db}
\begin{aligned}
B_{\Xi} (u^{\flat}) & \assign L^{- 1} \left[ \Delta u^{\flat} \prec Z + 2\nabla
u^{\flat} \prec \nabla Z + u^{\flat} \prec Z + \nabla \Delta u^{\flat} \prec
\widetilde{W} + 2\nabla^2 u^{\flat} \prec \nabla \widetilde{W}\right.\\
 &- \nabla u^{\flat}
\prec \widetilde{W} + 2 L \widetilde{W} \prec \nabla u^{\flat} + L Z \prec u^{\flat}
\\ & + 2 \nabla u^{\flat} \prec (L \widetilde{W} \circ Z) + 2 \nabla u^{\flat} \succ
(L \widetilde{W} \circ Z) \\
 &+ 2 u^{\flat} \prec (L \widetilde{W} \circ \nabla Z) + 2
u^{\flat} \succ (L \widetilde{W} \circ \nabla Z) +  2 \nabla u^{\flat} \prec (L
\widetilde{W} \circ \nabla \widetilde{W}) \\
&+ 2 \nabla u^{\flat} \succ (L \widetilde{W}
\circ \nabla \widetilde{W}) + u^{\flat} \prec (L Z \circ Z)+u^{\flat} \succ (L
Z \circ Z)\\ 
&+ \left. \nabla^{} u^{\flat} \prec (L Z \circ \widetilde{W}) + \nabla
u^{\flat} \succ (L Z \circ \widetilde{W}) \right] . 
\end{aligned}
\end{equation}
Using again the paraproduct estimates from Lemma \ref{lem:paraest}, one sees
that the terms in the brackets are at least of regularity $\ssp^{0 -}$, which
implies $B_{\Xi} (u^{\flat}) \in \ssp^{2 -}.$ We make this precise in the following result.

\begin{lemma}
	Let $B_{\Xi}$ be defined as above, then we have the following bounds for
	$\sigma < 2$ and $\varepsilon > 0$
	\begin{enumeratenumeric}
		\item $\| B_{\Xi} (v) \|_{\ssp^{\sigma}} \le C_{\Xi} \| v \|_{\ssp^{\sigma
				- 1 / 2 + \varepsilon}} $
		
		\item $\| B_{\Xi} (v) \|_{\CC^{\sigma}} \le C_{\Xi} \| v
		\|_{\CC^{\sigma - 1 / 2}} $,
	\end{enumeratenumeric}
	where for the the constant we can choose $C_{\Xi} = C \| \Xi
	\|_{{\mathscr{X}}^{\sigma - 3 / 2}}$, see Definition \ref{def:3dnoise} for the precise definition of the norm, where $C > 0$ is an independent constant. 
\end{lemma}
\begin{proof}
	This follows from the paraproduct estimates, Lemma \ref{lem:paraest}, for
	the first case. The second case works precisely in the same way using the
	paraproduct estimates for Besov-H\"older spaces and Schauder estimates, see
	e.g. {\cite{gubinelli2015paracontrolled}}.
\end{proof}

Finally we collect everything in the following rigorous definition which describes the domain of the Anderson Hamiltonian.

\begin{definition}
	\label{def:3dham}Let $W, \widetilde{W}, Z$ be as above. Then, for $1 <
	\gamma < 3 / 2$, we define the space
	\[ \mathscr{W}_{\Xi}^{\gamma} \assign e^W \mathscr{U}_{\Xi}^{\gamma} \assign
	e^W \{ u^{\flat} \in \ssp^{\gamma} \ \text{s.t.} \  u^{\flat} =
	u^{\flat} \prec Z + \nabla u^{\flat} \prec \widetilde{W} + B_{\Xi}
	(u^{\flat}) + u^{\sharp}, ~\text{for}~ u^\sharp \in \ssp^2 \}, \]
	where $B_{\Xi} (u^{\flat}) $ is as in (\ref{eqn:3db}).We furthermore equip
	the space with the scalar product given by, for $f, g \in \mathscr{W}_{\Xi}^{\gamma}$,
	\[ \langle u,w
	\rangle_{\mathscr{W}_{\Xi}^{\gamma}} \assign \langle u^{\flat}, w^{\flat}
	\rangle_{\ssp^{\gamma}} + \langle u^{\sharp}, w^{\sharp} \rangle_{\ssp^2} . \]
	Given $u = e^W u^{\flat} \in \mathscr{W}_{\Xi}^{\gamma}$ we
	define the renormalized Anderson Hamiltonian acting on $u$ in the following
	way
	\begin{align}\label{equ:def3dAH}
	 A u = e^W (\Delta u^{\sharp} + L Z \circ u^{\sharp} + 2L \widetilde{W} \circ
	\nabla u^{\sharp} + G (u^{\flat})), 
	\end{align}
	where
	\begin{align*}
	 G (u^{\flat})   \assign &  B_{\Xi} (u^{\flat}) + 2\nabla u^{\flat} \circ (L
	\widetilde{W} \circ Z) + 2C (\nabla u^{\flat}, Z, L \widetilde{W}) + u^{\flat}
	\circ (L \widetilde{W} \circ \nabla Z)
	\\ & + C (u^{\flat}, \nabla Z, L \widetilde{W})
	+ 2L \widetilde{W} \circ (\nabla^2 u^{\flat} \prec \widetilde{W})
	\\& + 2\nabla
	u^{\flat} \circ (L \widetilde{W} \circ \nabla \widetilde{W}) +2 C (\nabla
	u^{\flat}, \nabla \widetilde{W}, L \widetilde{W}) + 2L \widetilde{W} \circ \nabla
	B_{\Xi} (u^{\flat}) 
	\end{align*}
	and $C$ denotes the commutator from Proposition \ref{prop:commu}. Note that
	this definition is equivalent to (\ref{eqn:3dHuflat}) by construction.
\end{definition}

After this definition, some remarks are in order.
\begin{remark} \label{def:3dregularizedA}
	In view of \eqref{eqn:3dHuflat}, for regularized white noise $\xi_{\varepsilon}$, we set
	\begin{eqnarray}\label{equ:HamiltonianApp3d}
	A_{\varepsilon} u & \assign & e^{W_{\varepsilon}} (\Delta u^{\flat} + 2 (1
	- \Delta) \widetilde{W_{\varepsilon}}_{} \cdummy \nabla u^{\flat} + (1 -
	\Delta) Z_{\varepsilon} u^{\flat}) \label{eqn:Hepsdef} \\
	& = & \Delta u + \xi_{\varepsilon} u - (c^1_{\varepsilon} +
	c^2_{\varepsilon}) u, \nonumber
	\end{eqnarray}
	where we have
	
	$\begin{array}{lll}
	W_{\varepsilon} & = & X_{\varepsilon} + X_{\varepsilon}^{\zzone} +
	X_{\varepsilon}^{\zztwo}\\
	X_{\varepsilon} & = & (- \Delta)^{- 1} \xi_{\varepsilon}\\
	X^{\zzone}_{\varepsilon} & = & (1 - \Delta)^{- 1} (| \nabla
	X_{\varepsilon} |^2 - c^1_{\varepsilon})\\
	X^{\zztwo}_{\varepsilon} & = & 2 (1 - \Delta)^{- 1} \left( \nabla
	X_{\varepsilon} \cdummy \nabla X_{\varepsilon}^{\zzone} \right)\\
	\widetilde{W_{\varepsilon}} & = & (1 - \Delta)^{- 1} \nabla
	W_{\varepsilon}\\
	Z_{\varepsilon} & = & (1 - \Delta)^{- 1} \left( | \nabla
	X_{\varepsilon}^{\zzone} |^2 - c^2_{\varepsilon} + | \nabla
	X_{\varepsilon}^{\zztwo} |^2 + 2 \nabla X_{\varepsilon} \cdummy
	\nabla X_{\varepsilon}^{\zztwo} + 2 \nabla X_{\varepsilon}^{\zzone}
	\cdummy \nabla X_{\varepsilon}^{\zztwo} - X_{\varepsilon}^{\zzone} -
	X_{\varepsilon}^{\zztwo} \right) \\
	& \tmop{and} & \\
	u^{\flat} & \assign & e^{- W_{\varepsilon}} u.
	\end{array}$
	Recall that the renormalization constants, from Section \ref{sec:3dren}, are
	\[ c^1_{\varepsilon} = O (\varepsilon^{- 1}) \ \tmop{and} \ c^2_{\varepsilon} =
	O (\log \varepsilon). \]
	Observe that, now this makes the constant $c_\varepsilon$ in \eqref{eqn:2dHepsdef} precise as $c_\varepsilon= c^1_{\varepsilon}+ c^2_{\varepsilon}$.
\end{remark}
\begin{remark}\upshape\upshape
	As in the 2d case, the space $\mathscr{W}^\gamma_\Xi$ is independent of $\gamma$ and we will denote it simply by $\mathscr{W}_\Xi.$ Moreover, one can introduce a Fourier cutoff $\Delta_{>N}$ at level $2^N$ and write
	\begin{equation}
	u^{\flat} =
	\Delta_{> N}(u^{\flat} \prec Z + \nabla u^{\flat} \prec \widetilde{W} + B_{\Xi}
	(u^{\flat}) )+ u^{\sharp}.
	\end{equation}
	In this case, the space can be shown to be the same in a similar way to 2d case, please see Remark \ref{rem:equivalentCuttofSp}.
\end{remark}

We set up the notation with the following remark.

\begin{remark}\upshape\label{rem:domainNotationset3d}
	Similar to Remark \ref{rem:domainNotationset2d}, we introduce the notation \[\mathscr{D}(A) := \mathscr{W}_{\Xi}.\]
\end{remark}

We can generalize our framework that uses the $\Gamma$-map to 3d. This time, we  define
the linear map  $\Gamma$ as
\begin{equation}\label{equ:ubetaCUTTansatz}
\Gamma f = \Delta_{> N} (\Gamma f \prec Z + \nabla (\Gamma f) \prec
\tilde{W} + B_{\Xi} (\Gamma f)) + f. 
\end{equation}
This  allows us to realize $u^\flat$ as a fixed point of this map, that is $u^\flat = \Gamma u^{\sharp}$. Similar to the 2d case, for $N$ large enough,  we
can show this map exists and has useful bounds, and obtain the following generalization of Proposition \ref{lem:gamma} to 3d.

\begin{proposition}
	\label{lem:gamma3}We can choose $N$ large enough depending only on $\Xi$ and
	$s$ so that
	\begin{equation}
	\| \Gamma f \|_{L^{\infty}} \le 2 \| f \|_{L^{\infty}},
	\label{eq:gamma3-bound-Linfty}
	\end{equation}
	\begin{equation}
	\| \Gamma f \|_{\ssp^s} \le 2 \| f \|_{\ssp^s},
	\label{eq:gamma3-bound-sobolev}
	\end{equation}
	for $s \in \left[ 0, \frac{3}{2} \right)$.
\end{proposition}

\begin{proof}
	With slight modifications, the proof is basically the same as in the 2d case, namely Propostion \ref{lem:gamma}. For
	(\ref{eq:gamma3-bound-Linfty}) choose again a small $\delta > 0,$ then
	\begin{eqnarray*}
		\| \Gamma f \|_{L^{\infty}} & \le & \| f \|_{L^{\infty}} + \|
		\Delta_{> N} (\Gamma f \prec Z + \nabla (\Gamma f) \prec \widetilde{W} +
		B_{\Xi} (\Gamma f)) \|_{\CC^{\delta}}\\
		& \le & \| f \|_{L^{\infty}} + 2^{- \delta N} \| \nobracket \Gamma
		f \prec Z + \nabla (\Gamma f) \prec \widetilde{W} + B_{\Xi} (\Gamma f
		\nobracket) \|_{\CC^{2 \delta}}\\
		& \tmop{and} & \\
		\| \Gamma f \prec Z \|_{\CC^{2 \delta}} & \lesssim & \| \Gamma f
		\|_{\CC^{- \delta}} \| Z \|_{\CC^{3 \delta}} \lesssim
		C_{\Xi} \| \Gamma f \|_{L^{\infty}}\\
		\| \nabla \Gamma f \prec \widetilde{W} \|_{\CC^{2 \delta}} & \lesssim
		& \| \nabla (\Gamma f) \|_{\CC^{- 1 - \delta}} \| \widetilde{W}
		\|_{\CC^{1 + 3 \delta}} \lesssim C_{\Xi} \| \Gamma f
		\|_{L^{\infty}}\\
		\| B_{\Xi} (\Gamma f) \|_{\CC^{2 \delta}} & \lesssim & C_{\Xi} \| \Gamma f
		\|_{\CC^{2 \delta - 1 / 2}} \lesssim C_{\Xi} \| \Gamma f
		\|_{L^{\infty},}
	\end{eqnarray*}
	which allows us to conclude by choosing $N$ large enough depending on the norm of the 
	enhanced noise $\Xi$.
	The proof of the Sobolev case is similar.
\end{proof}

\begin{remark}\upshape\label{rem:smoothNotation2}
Also in this (3d) case, similar to the remark \ref{rem:smoothNotation}, we define  $\Gamma_\varepsilon$, using the approximations in Theorem \ref{thm:3dren}.  
\end{remark}

By using the map $\Gamma$ we can obtain an analysis very similar to the 2d case.  For starters, we state the following analogous result and conclude this section.

\begin{lemma} \label{lem:identityGammaconv3d}
	Let $\gamma$ be as  in Definition \ref{def:3dham}. We have that $|| \text{id} - \Gamma  \Gamma_\varepsilon^{-1}||_{\mathscr{H}^\gamma \rightarrow \mathscr{H}^\gamma} \rightarrow 0$.
\end{lemma}

\begin{proof}
The proof is very similar to that of Lemma \ref{lem:identityGammaconv}.
\end{proof}

\subsubsection{Density, symmetry and self-adjointness}

First, we prove the density of the domain of $A$, as stated in Definition \ref{def:3dham}.

\begin{proposition} \label{prop:3ddensity}
	Let $\beta<1/2,$ then the space $\mathscr{W}_\Xi$, as introduced in Definition \ref{def:3dham} is dense in $\mathscr{H}^\beta$.  Therefore, dense in $L^2$.
\end{proposition}

\begin{proof}
By using the multiplication estimates in Proposition \ref{lem:paraest}, for $f \in \mathscr{H}^\gamma$, we can write
\begin{align*}
|| e^{W} f  - e^{W} \Gamma \Gamma_\varepsilon^{-1} f||_{\mathscr{H}^\beta} \lesssim ||e^W||_{\mathscr{C}^{\beta}} ||f  -  \Gamma \Gamma_\varepsilon^{-1} f||_{\mathscr{H}^\gamma} 
\end{align*}
taking the limit as $\varepsilon \rightarrow 0$ shows that any element in the space $e^{W} H^{\gamma} \subset H^{\beta}$ can be approximated by elements in $\mathscr{W}_\Xi$.  For an arbitrary $f \in H^{\beta}$ one can further approximate it by elements $e^{W}(e^{-W_\varepsilon}f_\varepsilon)$ where we took $f_\varepsilon \in H^{\gamma}$ and   $||f - f_\varepsilon||_{H^{\beta}} \rightarrow 0$.
\end{proof}

The following is an analogue of Theorem \ref{lem:h2bound} for the 3d
Hamiltonian.
\begin{theorem} \label{thm:3dusharpbound}
	The renormalized Anderson Hamiltonian 
	$A : \mathscr{D}(A) \rightarrow L^2$
	is a bounded operator and we get the following $\ssp^2$ bound for $u^{\sharp}$
	\begin{equation} \label{equ:usharpBound}
	\| u^{\sharp} \|_{\ssp^2} \lesssim \| e^{- W} A u \|_{L^2} + C_{\Xi} \|
	u^{\flat} \|_{L^2} . 
	\end{equation}
\end{theorem}
\begin{proof}
	By the definition of $A$ we have
	\[ e^{- W} A u = \Delta u^{\sharp} + L Z \circ u^{\sharp} + 2L \widetilde{W}
	\circ \nabla u^{\sharp} + G (u^{\flat}), \]
	then we estimate\quad
	\begin{eqnarray*}
		\| L Z \circ u^{\sharp} \|_{L^2} & \lesssim & \| Z \|_{\CC^{3 / 2
				- \delta}} \| u^{\sharp} \|_{\ssp^{1 / 2 + 2 \delta}} 
		 \le  C_{\varepsilon, \delta} C_{\Xi} \| u^{\flat} \|_{L^2} +
		\varepsilon \| u^{\sharp} \|_{\ssp^2}\\
		& \text{and} & \\
		\| L \widetilde{W} \circ \nabla u^{\sharp} \|_{L^2} & \lesssim & \| \widetilde{W}
		\|_{\CC^{3 / 2 - \delta}} \| u^{\sharp} \|_{\ssp^{3 / 2 + 2 \delta}}
		 \le  C_{\varepsilon, \delta} C_{\Xi} \| u^{\flat} \|_{L^2} +
		\varepsilon \| u^{\sharp} \|_{\ssp^2}
	\end{eqnarray*}
	for any $\varepsilon > 0$ using Young's inequality, Sobolev interpolation,
	and the straightforward bound $\| u^{\sharp} \|_{L^2} \le C_{\Xi} \|
	u^{\flat} \|_{L^2 .}$ Moreover we bound $G (u^{\flat})$
	\begin{eqnarray*}
		\| G (u^{\flat}) \|_{L^2} & \le & C_{\Xi} \| u^{\flat} \|_{\ssp^{1 +
				\delta}}
		 \le  C_{\varepsilon, \delta} C_{\Xi} \| u^{\flat} \|_{L^2} +
		\varepsilon \| u^{\sharp} \|_{\ssp^2},
	\end{eqnarray*}
	where the first estimate follows from the paraproduct estimates, Proposition
	\ref{lem:paraest}, and the commutator bounds  (Proposition
	\ref{prop:commu}). This allows us to conclude
	\begin{equation} \label{equ:usharpbound2}
		\| A u \|_{L^2}  =  \| e^W e^{- W} A u \|_{L^2}
		 \le  \| e^W \|_{L^{\infty}} \| e^{- W} A u \|_{L^2}
		\le  C_{\Xi} (\| u^{\sharp} \|_{\ssp^2} + \| u^{\flat} \|_{L^2}),
	\end{equation}
	and, in a similar manner,
	\begin{eqnarray*}
		\| u^{\sharp} \|_{\ssp^2} & \le & \| e^{- W} A u \|_{L^2} + \| L Z
		\circ u^{\sharp} + 2L \widetilde{W} \circ \nabla u^{\sharp} \|_{L^2}\\
		& \le & \| e^{- W} A u \|_{L^2} + C_{\Xi} \| u^{\flat} \|_{L^2} +
		\frac{1}{2} \| u^{\sharp} \|_{\ssp^2},
	\end{eqnarray*}
	using the above bounds.
\end{proof}

\begin{proposition}
	We have that $A$ is a closed operator over its dense domain $\mathscr{D}(A)$.
\end{proposition}

\begin{proof}
For  $u_n \in \mathscr{D}(A)$, suppose that
\begin{align*}
u_n & \rightarrow u \\
A u_n & \rightarrow g.
\end{align*} 
Then, by \eqref{equ:usharpBound}, we have that $u_n^\sharp$ is a Cauchy sequence and $||w  - u_n^\sharp||_{\mathscr{H}^2} \rightarrow 0$ for some $w$.  We observe that then $u = e^{W}\Gamma(w)$, that is $u \in \mathscr{D}(A)$.  After that, writing the same estimate in the end of the proof of Proposition \ref{prop:opclosed} concludes the proof, this time utilizing \eqref{equ:usharpbound2} instead.
\end{proof}

For the domain what we know is $\mathscr{D}(A) \subset e^{W} \mathscr{H}^\gamma$.  But in the sequel we will need a precise approximation by smooth elements in $\mathscr{H}^2$.  The following Proposition establishes that.

\begin{proposition}
	For every $u \in \mathscr{D}(A)$ there exists $u_\varepsilon \in \mathscr{H}^2$ such that 
	\[
	|| u^\flat - u_\varepsilon^\flat  ||_{\mathscr{H}^\gamma} + || u^\sharp - u_\varepsilon^\sharp ||_{\mathscr{H}^2} \rightarrow 0
	\]
	as $\varepsilon \rightarrow 0$. For $u, v \in \mathscr{D}(A)$, with this approximation, we obtain
	\[
	\langle A_\varepsilon u_\varepsilon, v_\varepsilon \rangle \rightarrow \langle A u, v \rangle.
	\]
	Consequently, $A$ is a closed symmetric operator.
\end{proposition}

\begin{proof}
	The proof is similar to that of 2d case, this time using Proposition \ref{lem:gamma3}. In this case,  for $u^\sharp \in \mathscr{H}^2$, we take $u^\flat  = \Gamma_\varepsilon(u^\sharp)$ and $u_\varepsilon = e^{W_\varepsilon}\Gamma_\varepsilon(u^\sharp)$ for the approximations.  We omit the details.
\end{proof}

Before we introduce the resolvent and also the form domain we need the following result.

\begin{proposition}
	\label{lem:3dh1}Let $W$ be as above, then there exists a constant $C_{\Xi} >
	0$ such that
	\[ \| \nabla u^{\flat} \|^2_{L^2} \lesssim \| e^{- 2 W} \|_{L^{\infty}} (-
	\langle u, A u \rangle + C_{\Xi} \| u \|_{L^2}), \]
	where $u = e^W u^{\flat} \in \mathscr{D}(A) .$
\end{proposition}

\begin{proof}
	Using (\ref{eqn:3dHuflat}), we write
	\begin{eqnarray*}
		\langle u, A u \rangle & = & \langle e^{2 W} u^{\flat}, \Delta u^{\flat} +
		2 \nabla u^{\flat} \nabla W + L Z u^{\flat} \rangle\\
		& = & - \langle e^{2 W} \nabla u^{\flat}, \nabla u^{\flat} \rangle +
		\langle e^{2 W} u^{\flat}, L Z u^{\flat} \rangle,
	\end{eqnarray*}
	where the gradient term disappeared because we integrated by parts. Thus
	\begin{eqnarray*}
		\| \nabla u^{\flat} \|^2_{L^2} & \le & \| e^{- 2 W} \|_{L^{\infty}}
		\| e^W \nabla u^{\flat} \|_{L^2}\\
		& = & \| e^{- 2 W} \|_{L^{\infty}} (\langle e^{2 W} u^{\flat}, L Z u^{\flat} \rangle - \langle u, A u \rangle)\\
		& \le & \| e^{- 2 W} \|_{L^{\infty}} (\| u^{\flat} \|_{\ssp^{1 / 2 +
				\varepsilon}} \| e^{2 W} L Z u^{\flat} \|_{\ssp^{- 1 / 2 -
				\varepsilon}} - \langle u, A u \rangle)\\
		& \le & \| e^{- 2 W} \|_{L^{\infty}} (\| u^{\flat} \|^2_{\ssp^{1 / 2 +
				\varepsilon}} \| e^{2 W} L Z \|_{\CC^{- 1 / 2 -
				\varepsilon}} - \langle u, A u \rangle)\\
		& \le & \| e^{- 2 W} \|_{L^{\infty}} (C_{\Xi} \| e^{2 W}
		\|_{\CC^{1 / 2 - \varepsilon}} \| u^{\flat} \|^2_{\ssp^{1 / 2 +
				\varepsilon}} - \langle u, A u \rangle),
	\end{eqnarray*}
	where we have used Lemma \ref{lem:expprod}. Using again Sobolev
	interpolation and Young's inequality we can conclude by choosing
	$\varepsilon > 0$ small enough and pick a proper constant $C_\Xi >0$ for the conclusion.
\end{proof}

After this, we are ready to conclude the self-adjointness of the operator.
\begin{theorem} \label{thm:3dselfadj}
	The operator $H$ with domain $\mathscr{D}
	(H)$ is self-adjoint.
\end{theorem}
\begin{proof}
	Choosing $C_{\Xi} > 0$ (using Proposition \ref{lem:3dh1})
	large enough, we again want to prove that
	\[ (C_{\Xi} - A)^{- 1} : \mathscr{D} (A) \rightarrow L^2 \ \tmop{is} \
	\tmop{bounded} . \]
	This can be done in precisely the same way as the 2d case, similar to the proof of Proposition \ref{lem:laxmil},  by applying again the Babuska-Lax-Milgram theorem to the the bilinear map
	\begin{eqnarray*}
		B : \mathscr{D} (A) \times L^2 & \rightarrow & \mathbbm{R}\\
		B (u, v) & : = & \langle (C_{\Xi} - A) u, v \rangle .
	\end{eqnarray*}
	Afterwards, one concludes self-adjointness by using Proposition \ref{prop:selfadj}.
\end{proof}

Observe that the Proposition \ref{lem:3dh1} implies the positivity of the form for $
C_{\Xi} - A.$ Accordingly, we introduce the shifted operators.

\begin{definition} \label{def:setconstant3d}
	For a constant $K_\Xi > C_{\Xi}$, where $C_{\Xi}$ is as in the proof of Theorem \ref{thm:3dselfadj}, we define the following shifted operators
	\begin{align*}
	H_\varepsilon & :=   A_\varepsilon - K_\Xi \\
	H & := A -  K_\Xi  \\
	\end{align*}
	where, in the future, the constant $K_\Xi$ will be updated to be larger, as needed.
\end{definition}

Now, we  also define the form domain.

\begin{definition}\label{def:energySp3d}
	From Proposition \eqref{lem:gamma3} recall that $u= \Gamma u^{\sharp}$. We define the form domain of $H$, denoted by $\ED$, as the closure of the domain under the following norm
	\[ \| \Gamma u^{\sharp} \|_{\ED} \assign
	\sqrt{ \langle \Gamma u^{\sharp}, -{H} \Gamma u^{\sharp} \rangle}
	. \]
\end{definition}

We furthermore have the following classification for the domain and the form
domain of $H$, this is the 3d version of Proposition
\ref{lem:formdom}.
\begin{proposition} \label{lem:3dnormequiv}
	We have the following characterizations for the domain and the form domain:
	\begin{enumeratenumeric}
		\item $\Gamma u^{\sharp} \in e^{- W} \mathscr{D} ({H})
		\Leftrightarrow u^{\sharp} \in \ssp^2$.
		More precisely, on $\mathscr{D} ({H})  e^W \mathscr{U}_{\Xi}$ we have the
		following norm equivalence
		\[ \| u^{\sharp} \|_{\ssp^2} \lesssim_\Xi \| {H} \Gamma u^{\sharp} \|_{L^2}
		\lesssim_\Xi \| u^{\sharp} \|_{\ssp^2 .} \]
		\item $u \in \mathscr{D} ( \sqrt{- {H}} ) \Leftrightarrow
		e^{- W} u \in \ssp^1$.
		Then the precise statement is that on $\mathscr{D} ({H})$ the
		following norm equivalence holds
		\[ \| e^{- W} u \|_{\ssp^1} \lesssim_{\Xi} \| u \|_{\sqrt{- {H}}}
		\lesssim_{\Xi} \| e^{- W} u \|_{\ssp^1}, \]
		and hence the closures with respect to the two norms coincide.
	\end{enumeratenumeric}
\end{proposition}
\begin{proof}
	This follows from Theorem \ref{thm:3dusharpbound} and Proposition \ref{lem:3dh1}; similar to the proof of Proposition \ref{lem:formdom}.
\end{proof}

\subsubsection{Norm resolvent convergence}

In this section, we address the resolvent convergence results for the regularized operators as introduced in remark \ref{def:3dregularizedA} and definition \ref{def:setconstant3d}.  We first address the norm convergence of approximating Hamiltonians composed with the $\Gamma$-maps.  

\begin{proposition} \label{prop:opApp3d}
	Let  $u^\sharp \in \mathscr{H}^2$, $u  =e^{W} \Gamma(u^\sharp )$,  $u_\varepsilon^\flat  = \Gamma_\varepsilon(u^\sharp )$ and $u_\varepsilon  =e^{W_\varepsilon}u_\varepsilon^\flat$. We have that
	\begin{equation}\label{equ:normOpConv23d}
	|| Hu - H_\varepsilon u_\varepsilon ||_{L^2} \lesssim_{\Xi}   || \Xi_\varepsilon - \Xi||_{\mathscr{X}^\alpha} ||u^\sharp||_{\mathscr{H}^2}.
	\end{equation}
	Consequently, this implies that
	\begin{equation} \label{equ:normOpConv33d}
	|| He^W \Gamma - H_\varepsilon e^{W_\varepsilon} \Gamma_\varepsilon   ||_{\mathscr{H}^2 \rightarrow L^2}  \rightarrow 0.
	\end{equation}
	That is to say, $He^W \Gamma \rightarrow H_\varepsilon e^{W_\varepsilon} \Gamma_\varepsilon$ in norm.
\end{proposition}

\begin{proof}
The proof is similar to that of Proposition \ref{equ:normOpConv2}.  This time one uses the formula \eqref{equ:def3dAH}.  Then, proceeds in the same way  by using  Lemma \ref{lem:identityGammaconv3d} instead. Hence, the result.
\end{proof}

In the following results, using the techniques we have used in the 2d part, we address the notions of strong  resolvent  and norm resolvent convergence.

	\begin{theorem} \label{thm:resoventConv13d}
	For any $f \in L^2$, we have
	\[
	|| H^{-1} f - H_\varepsilon^{-1} f||_{L^2} \rightarrow 0
	\]
	as $\varepsilon \rightarrow 0$. Namely, $H_\varepsilon $ converges to $H$ in the strong  resolvent sense.
\end{theorem}
\begin{proof}
The proof follows the lines of the 2d case, that of Theorem \ref{thm:resoventConv1}. This time we take $u_\varepsilon = e^{W_\varepsilon}\Gamma_\varepsilon (u^\sharp)$ and use the Proposition \ref{prop:opApp3d}.
\end{proof}

	\begin{theorem} \label{thm:normResolventMain3d}
	Let $\beta$ be as defined in Proposition \ref{prop:3ddensity}.  Then, we have
	\[
	|| H^{-1} - H_\varepsilon^{-1} ||_{L^2 \rightarrow \mathscr{H}^\beta} \rightarrow 0
	\]
	as $\varepsilon \rightarrow 0$. Namely, $H_\varepsilon $ converges to $H$ in the norm resolvent sense.
\end{theorem}

\begin{proof}
	This proof is similar to that of Theorem \ref{thm:normResolventMain}. We only mention the different points.  
	
	By Proposition \ref{prop:opApp3d} we have that
	\[
	||H_\varepsilon e^{W_\varepsilon} \Gamma_\varepsilon u^\sharp - H e^{W} \Gamma u^\sharp||_{\mathscr{H}^2  \rightarrow L^2} \rightarrow  0
	\]

	This implies 
	\[
	||\Gamma_\varepsilon^{-1} e^{-W_\varepsilon} H_\varepsilon^{-1}   - \Gamma^{-1} e^{-W} H^{-1}|||_{L^2 \rightarrow \mathscr{H}^2 } \rightarrow 0.
	\]

	By using the same tricks as in the proof of Theorem \ref{thm:normResolventMain}, this time using Proposition \ref{lem:gamma3} and Lemma \ref{lem:identityGammaconv3d}, one obtains
	\[
	|| e^{-W_\varepsilon} H_\varepsilon^{-1}   - e^{-W} H^{-1}|||_{L^2 \rightarrow H^\gamma} \rightarrow 0.
	\]
	
	Observe also that $|| e^{(W-W_\varepsilon)} ||_{H^\gamma \rightarrow H^\beta} \rightarrow \text{id}$.  We can write the estimate
	\begin{align*}
	|| e^{(W-W_\varepsilon)} H_\varepsilon^{-1}   -  H^{-1}||_{L^2 \rightarrow H^\beta}&= || e^{W}(e^{-W_\varepsilon} H_\varepsilon^{-1}   - e^{-W} H^{-1})||_{L^2 \rightarrow H^{\beta}}\\ &\leq || e^{W}||_{ H^\beta \rightarrow H^\beta} ||e^{-W_\varepsilon} H_\varepsilon^{-1}   - e^{-W} H^{-1}||_{L^2 \rightarrow H^\beta}.
	\end{align*}
	which gives the result.
\end{proof}


Lastly, we give a version of Agmon's inequality which can be seen as a 3d analogue of Theorem \ref{lem:brgaineq}.
\begin{lemma}\label{lem:3dagmon}
	For $u\in \mathscr{D}(H)$ and $\mathscr{D}(H_\varepsilon)$ respectively,   we have the following $L^\infty$ bounds
	\begin{align*}
	\|u\|_{L^\infty}\lesssim_\Xi \|{H} u\|^{1/2}_{L^2}\|\sqrt{-{H}}u\|^{1/2}_{L^2}
	\end{align*}
		\begin{align*}
	\|u\|_{L^\infty}\lesssim_\Xi \|{H_\varepsilon} u\|^{1/2}_{L^2}\|\sqrt{-{H_\varepsilon}}u\|^{1/2}_{L^2}.
	\end{align*}
\end{lemma}
\begin{proof}
	The classical version of Agmon's inequality~\cite{Ag65} gives the bound\[
	\|v\|_{L^\infty}\lesssim\|v\|^{1/2}_{\ssp^1}\|v\|^{1/2}_{\ssp^2}.
	\]
	Now we compute
	\begin{align*}
	\|u\|_{L^\infty}&\le \|e^W\|_{L^\infty}\|\Gamma(u^\sharp)\|_{L^\infty}
	\lesssim_\Xi \|u^\sharp\|_{L^\infty}
	\lesssim_\Xi \|u^\sharp\|^{1/2}_{\ssp^1}\|u^\sharp\|^{1/2}_{\ssp^2}\\
	&\lesssim_\Xi \|{H} u\|^{1/2}_{L^2}\|(-{H})^{1/2}u\|^{1/2}_{L^2},
	\end{align*}
	where we have used Propositions \ref{lem:gamma3} and \ref{lem:3dnormequiv} in addition to Agmon's inequality and the straightforward bound $\|u^\sharp\|_{\ssp^1}\lesssim_\Xi\|u^\flat\|_{\ssp^1}.$  The second inequality follows the same argument with the note that the inequality constant is independent of $\varepsilon$.
\end{proof}

\section{Semilinear evolution equations } \label{sec:section3}
To recap, in the previous section we have introduced the operators $H$ and $H_\varepsilon$ (Definitions \ref{def:setconstant} and \ref{def:setconstant3d}) along with their domains $\mathscr{D}(H), \mathscr{D}(H_\varepsilon) = \mathscr{H}^2 $ (Remarks \ref{rem:domainNotationset2d} and \ref{rem:domainNotationset3d}) and energy domains $\ED, \mathscr{D}(\sqrt{H_\varepsilon}) = \mathscr{H}^1$ (Definitions \ref{def:energySp2d} and \ref{def:energySp3d}).  We have also studied their resolvents and norm resolvent convergence of regularized operators (Theorems \ref{thm:normResolventMain} and \ref{thm:normResolventMain3d}).  At the end, we have  obtained certain functional inequalities to be used in the present section.

In this part, we utilize this preceding analysis in the study of some semilinear PDEs, more precisely nonlinear Schr\"odinger and wave-type equations with the linear part given by the 2-d and 3-d Anderson Hamiltonian. As preliminaries, we derive and record some simple results for the corresponding linear equations as well as for certain PDEs with sufficiently nice nonlinearities. 
\subsection{Linear equations and bounded nonlinearities}

In this section, as a transition to the more sophisticated nonlinear cases; we first demonstrate the  solutions to the linear evolutions and PDEs with bounded type nonlinearities. We also obtain the convergence of the solutions to the regularized equations.
		\subsubsection{Abstract Cauchy theory for the linear and bounded nonlinear
			equations}\label{sec:abslinear}
		We want to apply Theorem 3.3.1 from Cazenave~\cite{cazenave2003semilinear}. This proves global well-posedness of
		\begin{eqnarray*}
			i \partial_t u & = & Q u + g (u)\\
			u (0) & = & u_0 \in \mathscr{D} (Q)
		\end{eqnarray*}
		in the strong sense, meaning $u \in C (\mathbb{R}; \mathscr{D} (Q)) \cap C^1
		(\mathbbm{R}; X),$ for a sufficiently nice nonlinearity $g$ and self-adjoint
		$Q.$
		\begin{theorem}
			\label{thm:abscauchy}Consider the abstract Cauchy problem
			\begin{eqnarray}
			\left\{ \begin{array}{l}
			i \partial_t u = Q u + g (u)\\
			u (0) = u_0
			\end{array} \right. &  &  \label{eqn:abscauchyA}
			\end{eqnarray}
			where $Q$ is a self-adjoint operator on a Hilbert space $X.$ Then we have the
			following two results for Schr\"odinger and wave equations  respectively.
			\begin{enumeratenumeric}
				\item  Assume $(Q u, u) \le 0$ for $u \in \mathscr{D} (Q)$ and $g :
				X \rightarrow X$ is Lipschitz on bounded sets as well as $(g (x), i x)_X =
				0$ for all $x \in X$ and $g = G'$ where $G \in C^1 ( \mathscr{D}
				( \sqrt{- Q} ) )$. For the case $Q = H$, $X = L^2 (\mathbbm{T}^d) \ d = 2, 3$ $u_0
				\in \mathscr{D} (H)$and $g (u) \assign K_{\Xi} u + u \varphi' (| u |^2)$
				where $\varphi \in C_b^2$ we get a unique global strong solution of
				(\ref{eqn:abscauchyA})
				\[ u \in C ([0, \infty) ; \mathscr{D} (H)) \cap C^1 ([0, \infty) ; L^2) .
				\]
				We can also relax this slightly if we ask for $u_0 \in \mathscr{D} (
				\sqrt{- H} )$.  We get a unique global energy solution
				\[ u \in C ( [0, \infty) ; \mathscr{D} ( \sqrt{- H} )
				) \cap C^1 ( [0, \infty) ; \mathscr{D}^{\ast} ( \sqrt{-
					H} ) ) . \]
				In both cases conservation of mass and energy holds for all times.
				
				\item For the wave equation, with $d=2,3,$ we set 
				\[ Q = i \left(\begin{array}{cc}
				0 & \mathbbm{I}\\
				{H} & 0
				\end{array}\right), \quad \mathscr{D} (Q) =\mathscr{D} ({H}) \oplus
				\mathscr{D} ( \sqrt{- {H}} )\] \[X = (L^2
				(\mathbbm{T}^d))^2,\quad  g (u) = \left(\begin{array}{c}
				0\\
				- K_{\Xi} u
				\end{array}\right) . \]
				Then the abstract linear wave equation
				\begin{eqnarray*}
					i \frac{\mathd}{\tmop{dt}} \left(\begin{array}{c}
						u\\
						\partial_t u
					\end{array}\right) & = & i \left(\begin{array}{cc}
					0 & \mathbbm{I}\\
					- {H} & 0
				\end{array}\right) \left(\begin{array}{c}
				u\\
				\partial_t u
			\end{array}\right) + \left(\begin{array}{c}
			0\\
			- K_{\Xi} u
		\end{array}\right)\\
		\left(\begin{array}{c}
			u\\
			\partial_t u
		\end{array}\right)_{t = 0} & = & \left(\begin{array}{c}
		u_0\\
		u_1
	\end{array}\right)
\end{eqnarray*}
has a unique global strong solution $(u, \partial_t u) \in C ([0, \infty)
; \mathscr{D} (A)) \times C^1 ([0, \infty) ; L^2)$ i.e. $u \in C ([0, \infty)
; \mathscr{D} (H)) \cap C^1 ( [0, \infty) ; \mathscr{D} (
\sqrt{- {H}} ) ) \cap \CC^2 ([0, \infty) ; L^2)$ and
energy conservation holds. 
\end{enumeratenumeric}
\end{theorem}

\begin{proof}
	\begin{enumeratenumeric}
		\item The properties of the Hamiltonian have already been verified, it
		remains to check that the nonlinearity $g$ satisfies all the conditions. We claim that
		\[ g (v) = G' (v)\  \tmop{with} \  G (v) = \int \frac{K_{\Xi}}{2} | v |^2 +
		\frac{1}{2} \int \varphi (| v |^2) \in C^1 ( \mathscr{D} (
		\sqrt{{H}} ) ; \mathbbm{R} ) . \]
		Moreover $g : L^2 \rightarrow L^2$ is locally Lipschitz and $\langle g (u), i u\rangle
		= 0$ for $u \in L^2 .$
		
		By construction we have
		\[ \langle g (u), i u\rangle = \tmop{Re} i \int K_{\Xi} | u |^2 + | u |^2 \varphi' (|
		u |^2) = 0. \]
		Next we show the differentiability of $G.$ Let $u, v \in \mathscr{D}( \sqrt{{H}}),$ then
		\begin{align*}
			& G (u) - G (v) - G' (v) (u - v) \\ & =  \int \frac{K_{\Xi}}{2} | u |^2 +
			\frac{1}{2} \int \varphi (| u |^2) - \int \frac{K_{\Xi}}{2} | v |^2 -
			\frac{1}{2} \int \varphi (| v |^2) - (g (v), u - v)\\
			& =  \int \frac{K_{\Xi}}{2} | u - v |^2 + \frac{1}{2} \int f (u) - f
			(v) - f' (v) (\overline{u - v})\\
			| \ldots | & \le  (K_{\Xi} + \| \varphi \|_{C_b^2}) \| u - v
			\|^2_{L^2}\\
			& \le  (K_{\Xi} + \| \varphi \|_{C_b^2}) \| u - v
			\|^2_{\mathscr{D} ( \sqrt{{H}} )},
		\end{align*}
		with $f (u) \assign \varphi (| u |^2) .$ This proves the
		differentiability. Lastly we prove the $L^2 $ local Lipschitz property of
		$g.$ Fix $v \in L^2$ and $u \in B_M (v),$ for some $M > 0.$ Then
		\begin{eqnarray*}
			\| g (u) - g (v) \|_{L^2} & \le & K_{\Xi} \| u - v \|_{L^2} + \| u
			\varphi' (| u |^2) - v \varphi' (| v |^2) \|_{L^2}\\
			& \le & K_{\Xi} \| u - v \|_{L^2} + \| \varphi' \|_{\infty} \| u
			- v \|_{L^2} + \| v \|_{L^2} \| \varphi' (| u |^2) - \varphi' (| v |^2)
			\|_{L^2}\\
			& \le & K_{\Xi} \| u - v \|_{L^2} + \| \varphi' \|_{\infty} \| u
			- v \|_{L^2} + \| v \|_{L^2} \| \varphi'' \|_{\infty} \| u - v \|_{L^2}
		\end{eqnarray*}
		hence $g$ is locally Lipschitz as a map from $L^2$ to $L^2 .$
		
		\item  See  \cite[Chapter X.13]{reedsimon2}.
	\end{enumeratenumeric}
\end{proof}
\subsubsection{The linear multiplicative
	Schr\"odinger equation}\label{sec:linschr}
In this part, we demonstrate the solution to the linear Schr\"odinger equation  
\begin{equation}
i \partial_t u = H u \quad \tmop{on}  \ \mathbbm{T}^d,
\end{equation}
with domain initial data.
A simple but  important observation is that the Schr{\"o}dinger
equation conserves the $L^2$ norm. Also observe that $\partial_t u$ formally
satisfies
\[ i \partial_t \partial_t u = H (\partial_t u), \]
so it solves the same equation and in particular we
have that $\| \partial_t u (t) \|_{L^2}$ is conserved and that
\[ \| H u \|_{L^2} = \| \partial_t u (t) \|_{L^2} = \| \partial_t u (0)
\|_{L^2} = \| H u (0) \|_{L^2}, \]
which we will assume to be finite. Naturally, this gives the condition that
the  initial data have to satisfy.   Therefore, we will assume $u_0 \in \mathscr{D}(H)$  which implies  $\| H u_0
\|_{L^2} < \infty$,  by Theorem \ref{lem:h2bound} . 
To make this precise, we write
\begin{align*}
u (t)  &=  e^{{-i t H}} u_0\\
\frac{d}{d t} u (t)  &= -i e^{{-i t H}} H u_0\\
\left\| \frac{d}{d t} u (t) \right\|_{L^{^2}}  &=  \| H u_0 \|_{L^{^2}}
=  \| H u (t) \|_{L^{^2}} .
\end{align*}
So $\| \partial_t u (t) \|_{L^2}$ is conserved for for the \tmtextit{abstract}
solution $u$ as defined in Section \ref{sec:abslinear}. For the regularized
equation, the unique solution is given by
\[ u_{\varepsilon} (t) = e^{-i t H_{\varepsilon}} u^{\varepsilon}_0 \in \ssp^2, \]
where $u^{\varepsilon}_0 \in \ssp^2$ is the regularized initial datum. If we
choose the regularization
\begin{align*}
u^{\varepsilon}_0  \assign &  H_{\varepsilon}^{- 1}  H u_0 \in
\ssp^2, \end{align*}
then we readily have,
$
H_{\varepsilon} u^{\varepsilon}_0  =  H u_0
$
and we get $u^{\varepsilon}_0 \rightarrow u_0$ in $L^2$,  by norm resolvent convergence, namely Theorem \ref{thm:normResolventMain}.
By  \cite[Theorem VIII.21]{reedsimon1}, $e^{-i t H_{\varepsilon}} \rightarrow
e^{-i t H}$ strongly for any time $t$, which implies
\begin{align*}
e^{-i t H_{\varepsilon}} u^{\varepsilon}_0  \rightarrow & e^{-i t H} u_0,\\
\tmop{and} & \\
e^{-i t H_{\varepsilon}} H_{\varepsilon} u^{\varepsilon}_0  \rightarrow &
e^{-i t H} H u,
\end{align*}
in $L^2$ for any $t \in \mathbbm{R}.$

We summarize these results in the following theorem

\begin{theorem}
	Let $T > 0,$ $u_0 \in \mathscr{D} (H)$. Then there exists a unique solution
	$u \in \nobracket C ([0, T] ; \mathscr{D}(H)) \cap C^1 ([0,
	T] ; L^2) $ to the equation
	\[ \left\{ \begin{array}{l}
	i \partial_t u = H u \quad\\
	u (0, \cdummy) = u_0
	\end{array} \right. \ \tmop{on} \ [0, T] \times \mathbbm{T}^d . \]
	Moreover, this agrees with the $L^2 -$limit of the solutions
	$u_{\varepsilon} \in C ([0, T] ; \ssp^2) \cap C^1 ([0, T] ; L^2) $ to
	\[ \left\{ \begin{array}{l}
	i \partial_t u_{\varepsilon} = H_{\varepsilon} u_{\varepsilon} \quad\\
	u (0, \cdummy) = u^{\varepsilon}_0
	\end{array} \right. \ \tmop{on} \ [0, T] \times \mathbbm{T}^d, \]
	with the regularized data given as
	\[ u^{\varepsilon}_0 \assign  H_{\varepsilon}^{- 1}  H u_0 \in
	\ssp^2 . \]
	One also obtains the convergence of $\partial_t u_\varepsilon$ and $H_\varepsilon u_\varepsilon$ to $\partial_t u$ and $H_\varepsilon u$ in $L^2$.
\end{theorem}
\begin{remark}\upshape\upshape
	One could even get global wellposedness for the equation with initial data in $\mathscr{D}(\sqrt{-H})$ or in $L^2.$ Moreover one could have a bounded nonlinearity as in section \ref{sec:abslinear}.
\end{remark}
\subsubsection{The linear multiplicative wave equation}

Similarly to the Schr{\"o}dinger case, now we consider  the linear wave equation
\[ \partial^2_t u = H u \]
with initial data $(u_0, u_1) \in \mathscr{D} (H) \times
\mathscr{D} ( \sqrt{- H} ) .$
For the regularized equation
\begin{align*}
\partial^2_t u_{\varepsilon} & =  H_{\varepsilon} u_{\varepsilon}\\
(u_{\varepsilon}, \partial_t u_{\varepsilon}) |_{t = 0} & = 
(u^{\varepsilon}_0, u^{\varepsilon}_1) \in \ssp^2 \times \ssp^1
\end{align*}
a solution is given by
\[ \left(\begin{array}{c}
u_{\varepsilon}\\
\partial_t u_{\varepsilon}
\end{array}\right) = e^{i t Q_{\varepsilon}} \left(\begin{array}{c}
u^{\varepsilon}_0\\
u_1^{\varepsilon}
\end{array}\right),\qquad\text{where}\qquad Q_{\varepsilon} = i \left(\begin{array}{cc}
0 & \mathbbm{I}\\
- H_{\varepsilon} & 0
\end{array}\right).
 \] We again choose the same approximation for $u_0$ as in
the Schr{\"o}dinger case
\begin{align*}
u^{\varepsilon}_0 & \assign (- H_{\varepsilon})^{- 1}(-H) u_0 \in
\ssp^2 \\
u^{\varepsilon}_1 & \assign (\sqrt{- H_{\varepsilon}})^{- 1}\sqrt{-H} u_1 \in
\ssp^1
\end{align*}
for some initial datum $u_0 \in \mathscr{D} (H) .$ Then we again have
\begin{eqnarray*}
	u^{\varepsilon}_0 & \rightarrow & u_0 \ \tmop{in} \ L^2\\
	H_{\varepsilon} u^{\varepsilon}_0 & \rightarrow & H u_0 \ \tmop{in} \ L^2.
\end{eqnarray*}
For the initial velocity, we also have 
\[ u^{\varepsilon}_1 \rightarrow u_1  \ \tmop{in} \  L^2 . \]
Then for any time $t$ we get as in the Schr{\"o}dinger case
\[ \left(\begin{array}{c}
u_{\varepsilon} (t)\\
\partial_t u_{\varepsilon} (t)
\end{array}\right) = e^{i \tmop{tQ}_{\varepsilon}} \left(\begin{array}{c}
u^{\varepsilon}_0\\
u^{\varepsilon}_1
\end{array}\right) \rightarrow e^{i t Q} \left(\begin{array}{c}
u_0\\
u_1
\end{array}\right) \ \tmop{in} \  L^2 \]
and
\[ \frac{\mathd}{\tmop{dt}} \left(\begin{array}{c}
u_{\varepsilon} (t)\\
\partial_t u_{\varepsilon} (t)
\end{array}\right) = i t e^{i \tmop{tQ}_{\varepsilon}} Q_{\varepsilon}
\left(\begin{array}{c}
u^{\varepsilon}_0\\
u^{\varepsilon}_1
\end{array}\right) \rightarrow i t e^{i t Q} Q \left(\begin{array}{c}
u_0\\
u_1
\end{array}\right) \  \tmop{in} \  L^2 . \]
Moreover, we have that the convergence of the energies, namely
\begin{align*} E_{\varepsilon} (t) & \assign \left\langle \left(\begin{array}{c}
u_{\varepsilon} (t)\\
\partial_t u_{\varepsilon} (t)
\end{array}\right), \left( \begin{array}{cc}
- H_{\varepsilon} & 0\\
0 & \mathbbm{I}
\end{array} \right) \left(\begin{array}{c}
u_{\varepsilon} (t)\\
\partial_t u_{\varepsilon} (t)
\end{array}\right) \right\rangle \\
&  \rightarrow \left\langle
\left(\begin{array}{c}
u (t)\\
\partial_t u (t)
\end{array}\right), \left( \begin{array}{cc}
- H & 0\\
0 & \mathbbm{I}
\end{array} \right) \left(\begin{array}{c}
u (t)\\
\partial_t u_{} (t)
\end{array}\right) \right\rangle = E (t) \end{align*}
for any time $t$ and thus in particular the energy conservation passes to the
limit. We record these observations in the following theorem.

\begin{theorem}
	Let $T > 0$ and  $(u_0, u_1) \in \mathscr{D} (H) \times \mathscr{D} (
	\sqrt{- {H}} )$. Then there exists a unique solution $(u,
	\partial_t u) \in C ( [0, T] ; \mathscr{D} (H) \times \mathscr{D}
	( \sqrt{- {H}} ) ) \cap C^1 ( [0, T] ;
	\mathscr{D} ( \sqrt{-{H}} ) \times L^2 )$ to the
	equation
	\begin{eqnarray*}
		\partial^2_t u & = & H u \ \text{in} \  (0, T) \times \mathbbm{T}^d\\
		(u, \partial_t u) |_{t = 0} & = & (u_0, u_1)
	\end{eqnarray*}
	moreover it is equal to the $L^{\infty} ((0, T) ; L^2 (\mathbbm{T}^d) \times
	L^2 (\mathbbm{T}^d))$ limit of the approximate solutions $(u_{\varepsilon},
	\partial_t u_{\varepsilon}) $ to \
	\begin{eqnarray*}
		\partial^2_t u_{\varepsilon} & = & H_{\varepsilon} u_{\varepsilon}
		\ \tmop{in} \  (0, T) \times \mathbbm{T}^d\\
		(u_{\varepsilon}, \partial_t u_{\varepsilon}) |_{t = 0} & = &
		(u^{\varepsilon}_0, u^{\varepsilon}_1),
	\end{eqnarray*}
and  moreover we have the following convergences		
\begin{align*}
u_{\varepsilon}(t)&\to u(t) \text{ in }L^2\\
H_\varepsilon u_{\varepsilon}(t) &\to H u(t) \text{ in }L^2\\
(-H_\varepsilon)^{1/2}\partial_tu_{\varepsilon}(t)&\to (-H)^{1/2}\partial_tu(t) \text{ in }L^2\\
\partial_t^2u_{\varepsilon}(t)&\to \partial^2_tu(t) \text{ in }L^2
\end{align*} 
	with $(u^{\varepsilon}_0, u^{\varepsilon}_1)$ as above. Also, the
	energies converge and are conserved in time.
\end{theorem}
\begin{proof}
	The computations above prove that the $L^2 $ limit of the solutions we
	obtain is equal to the solution of the abstract Cauchy problem in Theorem
	\ref{thm:abscauchy} for all times. Hence, the two are equal.{\tmem{}}
\end{proof}

\subsection{Nonlinear Schr\"odinger equations in two dimensions}
In this section,  we are interested in solving the following defocussing cubic Schr\"odinger-type equation
\begin{align}
i \partial_t u=H u-u|u|^2\label{eqn:2dnls},
\end{align}
with domain and energy space data.

  Recall that, for the operator $H$ we have
\[
\langle u,-H u\rangle\ge 0 \text{ for all }u\in\mathscr{D}(H).
\]
We consider the mild formulation of \eqref{eqn:2dnls}
\begin{align}
u(t)=e^{-i t H}u_0+i \int_{0}^{t}e^{i(s-t)H}u(s)|u(s)|^2\mathrm{d}s\label{eqn:mildnls}
\end{align}
Furthermore, we introduce the energy for $u$ as
\begin{align}
E(u)(t):=-\frac{1}{2}\langle u,H u\rangle+\frac{1}{4}\int|u|^4.
\end{align} 
Using the equation one sees that the energy is formally conserved in time.
\subsubsection{Solution with initial conditions in $\mathscr{D}(H)$} \label{subsection:5.4}
In this section, we  assume  $u_0 \in \mathscr{D} (H)$. This is similar in spirit to the global strong  well-posedness of the classical cubic NLS with initial data in $\ssp^2,$ which was solved in \cite{brga80}. We obtain global in time strong solutions in our setting, which is the best one can hope for in view of the classical result.
We regularize the initial data in the following way \[u^{\varepsilon}_0 =
(-H_{\varepsilon})^{- 1} (-H) u_0 \in  \mathscr{D}  (H_{\varepsilon})\]
 so that by
the norm resolvent convergence of $H_{\varepsilon}$ to $H$
(by Theorem \ref{thm:normResolventMain}) we have
\begin{align*}
\lim\limits_{\varepsilon\to 0}  u_0^{\varepsilon} = u_0  \in L^2 \\ H_{\varepsilon} u_0^{\varepsilon} = H u_0\in L^2.
\end{align*} 
Note that $ \mathscr{D}  (H_{\varepsilon}) = \ssp^2$ so there exists global solutions
$u_{\varepsilon} \in C ([0, T],  \mathscr{D}  (H_{\varepsilon})) \cap C^1 ([0, T], L^2).$
This is an immediate consequence of the following result, which shows that for the operators $H_\varepsilon$ and $H$, we obtain global in time strong solutions of the associated cubic NLS on $\mathbb{T}^2$. 

\begin{theorem}\label{thm:2dnlsfix}
	For an arbitrary time $T>0,$ there exist unique solutions $u_\varepsilon\in C([0,T];\ssp^2)\cap C^1([0,T];L^2)$ and $u\in C([0,T];\mathscr{D}(H))\cap C^1([0,T];L^2)$ to
	\begin{align}
	u_\varepsilon(t) = e^{- i t H_\varepsilon} u^\varepsilon_0 + i\int_{0}^{t} e^{- i s H_\varepsilon} u_\varepsilon | u_\varepsilon
	|^2 (t - s) d s, \end{align}
	and
	\begin{align}
	u (t) = e^{- i t H} u_0 + i\int_{0}^{t} e^{- i s H} u | u
	|^2 (t - s) d s \end{align}
	respectively, with  initial data  $u^\varepsilon_0\in \ssp^2$ and $u_0\in \mathscr{D}(H).$
\end{theorem}

Before we prove the theorem, we need the following technical lemmas which will be used throughout the proof. The first one is a logarithmic Gronwall lemma.

\begin{lemma} \label{lem:loggronwall}
	Let $C_2,  \log C_1 \geq 1$ and $\theta(t)\ge1$ satisfy 
	\[
	\theta (t) \le C_1  + C_2 \int_0^t \theta (s) \log^{}
	(1+\theta (s)) \mathd s = h (t) .
	\]
	Then we have
	\[
	h (t) \le \exp (\log h (0) e^{C_2 t})-1 .
	\]
\end{lemma}

\begin{proof}
	We have that $h$ is a subsolution of the equation
	\[ \partial_t h (t) = C_2 \theta (t) \log^{} (1 + \theta (t)) \le C_2 (h(t)+1) \log^{} (h (t)+1) . \]
	So taking $\rho (t)$ being a solution of \ $\partial_t \rho (t) = C_2 (\rho (t)+1)
	\log^{} (\rho (t)+1), \quad \rho (0) = h (0)$, we have $\rho (t) \ge h
	(t)$. Indeed $\rho (0) = h (0)$ and whenever we have $\rho (t) = h (t)$ then
	\[ \partial_t (\rho (t) - h (t)) \ge C_2( \rho (t)+1 )\log^{} (\rho (t)+1) -
	C_2 (h (t)+1 )\log^{} (h (t)+1) = 0. \]
	Observe moreover that
	\[ \partial_t \log (\rho (t) +1) = C_2 \log^{} (\rho (t) +1) \Rightarrow \log (\rho (t) +1)
	= (\log h (0)) e^{C_2 t} . \]
\end{proof}

\begin{lemma} \label{fuderivative} For $v  \in C ([0, T] ;
	\ssp^2) \cap C^1 ([0, T] ; L^2)$,  $f(v)(t) = |v(t)|^2 v(t)$ is $C^1$ as a map from $[0,T]$ to $L^2$. The same is true for $v\in C([0,T];\mathscr{D}(H))\cap C^1 ([0, T] ; L^2).$
\end{lemma}
\begin{proof}
	We write
	
	\[
	\frac{|v(t+h)|^2v(t+h) - |v(t)|^2v(t) }{h} 
	\]
	and add and subtract the term $v^2(t+h)\bar v(t)$ which yields
	\[
	\frac{\bar{v}(t+h) (v(t+h))v(t+h) - v^2(t+h)\bar v(t) + v^2(t+h)\bar{v}(t) - \bar{v}(t) v(t)v(t) }{h}.
	\]
	This can  be rearranged as 
	\[
	{v}^2(t+h) \frac{\bar v(t+h) - \bar{v}(t)}{h} + \left(\bar{v}(t) (v(t+h) + v(t)) \right) \frac{v(t+h) - v(t))}{h} 
	\]
	where all the terms converge individually in $L^2$ as $h \rightarrow 0$. Indeed, one can easily check that the multiplication map $(f,g) \rightarrow f\cdot g$ defines a continuous map $\ssp^2 \times L^2 \rightarrow L^2$.  This follows from the embedding $\ssp^2\hookrightarrow L^\infty$ in 2d. Since, by Lemma \ref{lem:embedinfty}, we also have the embedding $\mathscr{D}(H)\hookrightarrow L^\infty,$ the same holds in this case. 
\end{proof}
\begin{proof}[Proof of Theorem \ref{thm:2dnlsfix}]
	This is a fixed point argument, which is essentially the same in both cases. For fixed $u_0 \in \mathscr{D} (H),$ we define the operator
	\[ \Phi (u) (t) \assign e^{- i t H} u_0 + i \int_{0}^{t} e^{-
		i s H} u | u |^2 (t - s) d s \]
	and claim that is in fact a contraction on $X = C ([0, T_E] ; \mathscr{D} (H))
	\cap C^1 ([0, T_E] ; L^2),$ where the time $T_E > 0$ depends on the initial data and  the energy,
	which is conserved. This will allow us to obtain a global in time solution.
	
	We bound, for $\| u \|_X \le M$ with $M$ chosen below, using  Theorem \ref{lem:brgaineq},
	\begin{align*}
	\| \partial_t \Phi (u) \|_{L^2} (t)\le& \| H u_0 \|_{L^2} +
	\int_{0}^{t} \|\partial_t(u|u|^2)(s)\|_{L^2} \tmop{ds} + C\|
	u_0 \|^3_{L^6} 
	\\
	\le & \| H u_0 \|_{L^2} +
	\int_{0}^{t} C M E (u) (s) (1 + \log (M+1)) \tmop{ds} + C\|
	u_0 \|^3_{L^6}
	\le \frac{M}{2}
	\end{align*}
	for $t \le T_E$ small enough such that \[\int_{0}^{T_E} C M E (u) (s) (1 +  \log (M+1)) \tmop{ds}\le \frac{M}{2}-(\| H u_0 \|_{L^2}+C\|u_0 \|^3_{L^6})\] and $M$ such that $\frac{M}{2}-(\| H u_0 \|_{L^2}+C\|u_0 \|^3_{L^6})>0.$
	
	Analogously, we compute
	\begin{align*}
	\| H \Phi (u) \|_{L^2} (t)  \le & \| H u_0 \|_{L^2} +
	\int_{0}^{t} C M E (u) (s) (1 +  \log (M+1)) \tmop{ds} + C\|
	u_0 \|^3_{L^6}
	\le  \frac{M}{2},
	\end{align*}
	since we can integrate by parts in the integral. Furthermore, Stone's theorem \cite[Theorem VIII.7 ]{reedsimon1} implies the time regularity of $\Phi(u).$
	For the contraction property, we need to estimate
	\begin{align*}
	\partial_t \Phi (u) (t) - \partial_t \Phi (v) (t)  = 
	\int_{0}^{t} e^{- i s H} \partial_t (u | u |^2 - v | v
	|^2) (t - s) \tmop{ds}.
	\end{align*}
	We obtain, by using Theorem \ref{lem:brgaineq} and Lemma \ref{lem:embedinfty},
	\begin{align*}
	\|& \partial_t \Phi (u) (t) - \partial_t \Phi (v) (t) \|_{L^2}\le  \\\le &
	3 \int_{0}^{t} \| u (s) \|^2_{L^{\infty}} \| \partial_t u
	- \partial_t v \|_{L^{\infty} L^2} + \| \partial_t v \|_{L^{\infty} L^2} \|
	H u - H v \|_{L^{\infty} L^2} (\| u (s) \|_{L^{\infty}} + \| v (s)
	\|_{L^{\infty}}) \tmop{ds}\\
	\le & 3 \| u - v \|_X  \int_{0}^{t} E (u) (s) (1
	+ \log (1+M))+ M \sqrt{(1 + \tmop{ \log (M+1)})}  (E^{1 / 2} (u) (s) + E^{1 / 2} (v)
	(s)) \tmop{ds}\\
	< & \| u - v \|_X
	\end{align*}
	for $t \le T_E$ by possibly making $T_E$ smaller depending on $E (u).$
	This gives us short time wellposedness, but since the time span depends only
	on the energy and the initial data, this can be iterated to yield a global
	strong solution. In fact, the only thing that is left to show is that $\|H u(T_E)\|_{L^2}$ can be bounded by $\|H u_0\|_{L^2}$, i.e. a priori bounds. This allows us to choose a global $M$ and then also a fixed time span $T_E$ which immediately implies a global solution.
	For the solution that exists up to time $T_E$, we have the estimate
	\begin{align*}
	\|H u(T_E)\|_{L^2}&\lesssim \|H u_0\|_{L^2}+\|u_0\|^3_{L^6}+\int_{0}^{T_E}\|\partial_t u(|u|^2)(s)\|_{L^2} ds\\
	&\lesssim \|H u_0\|_{L^2}+\|u_0\|^3_{L^6}+\int_{0}^{T_E}\|\partial_t u(s)\|_{L^2}\|u(s)\|^2_{L^\infty} ds\\
	&\lesssim \|H u_0\|_{L^2}+\|u_0\|^3_{L^6}+\int_{0}^{T_E}(\|H u(s)\|_{L^2} \\ & + E^{3/2}(u_0))E(u_0)(1+\log (1+\|H u(s)\|_{L^2}))  ds,
	\end{align*} 
	where we have used again Theorem \ref{lem:brgaineq} and the fact that one can estimate $||\partial_t u||_{L^2}$ by $||H u||_{L^2}$ using the equation. Now, we can conclude by using Lemma \ref{lem:loggronwall}. This gives us a bound, by possibly taking larger constants,  of the form
	\begin{align}\label{equ:domtimeEstimate}
	||H u(T_E)||_{L^2}\lesssim C_\Xi {E}^{3/2}(u_0)+ \exp (e^{ c E(u_0)T}
	\log [C_\Xi {E}^{3/2}(u_0) + \| H u_0 \|_{L^2}  )-1,
	\end{align}
	where $T$ is the maximum time of existence. Hence $M$, and therefore $T_E$, can be chosen globally which means that we can solve the cubic NLS on the whole interval $[0,T]$ by iterating.
	The proof for the regularized Hamiltonian follows the same lines with the crucial note that the inequality constant in Theorem \ref{lem:brgaineq} does not blow up, namely the constant does not depend on $\varepsilon$ but only on  $\Xi.$  
\end{proof}
\begin{remark}\upshape\upshape\label{rem:nlsdiffpowerfocsmalldata}
	One sees from the proof that the same remains true for NLS with general power nonlinearity, i.e.
	\begin{align*}
	i\partial_tu=H u-u|u|^{p-1},
	\end{align*}
	with $p\in(1,\infty),$ since all $L^p-$norms can be controlled by the energy. The result will also remain true in the focusing case under some suitable smallness conditions on $u_0.$
\end{remark}
We will moreover prove that the approximate solutions $u_\varepsilon,$ which are strong solutions of 
\begin{align} \label{eqn:nlsdomapp}
i\partial_t u_{\varepsilon} &=  H_{\varepsilon} u_{\varepsilon}-u_\varepsilon|u_\varepsilon|^2,  \\ \nonumber
u_\varepsilon(0)&=u^\varepsilon_0\in \mathscr{D}({H_{\varepsilon}} )= \ssp^2. \end{align}
converges to the solution $u$ of the limiting problem.
We prove the following result.
\begin{theorem} \label{thm:2ddomnls}
	Let $u_0 \in \mathscr{D} (H)$ and $T>0$ be  an arbitrary time.  Solutions to the regularized equations with initial data $u^\varepsilon_0:=(-H_\varepsilon)^{-1}(-H)u_0\in \ssp^2,$ (the unique global strong solutions $u_\varepsilon$ of \eqref{eqn:nlsdomapp}) converges to the unique global strong solutions  $u \in C([0, T] ; \mathscr{D}(H))\cap C^{1} ([0,
	T] ; L^2)$ of
	\begin{align*}
	 i \partial_t u &= H u -|u|^2 u,  \\
	 u(0)&=u_0,
	\end{align*}
	which is obtained in Theorem \ref{thm:2dnlsfix}. In fact, we get the following convergence results
		\begin{align*}
		u_{\varepsilon}(t)&\to u(t) \text{ in }L^2\\
		H_{\varepsilon} u_{\varepsilon}(t)&\to H u(t) \text{ in }L^2\\
		\partial_t u_{\varepsilon}(t)&\to \partial_t u(t) \text{ in }L^2
		\end{align*}
		for all $t\in[0,T].$
\end{theorem}

\begin{proof}

We know that the $u_\varepsilon$ satisfy the mild formulation
	\[ u_{\varepsilon} (t) = e^{- i H_{\varepsilon} t} u^{\varepsilon}_0 +i
	\int_0^t e^{- i  (t - s)H_{\varepsilon}}  u_{\varepsilon} (s)|u_{\varepsilon} (s)|^2 \mathd s
	\]
and $u$ satisfies
	\[ u (t) = e^{- i H_{\varepsilon} t} u_0 +
i\int_0^t e^{- i  (t - s)H}  u (s)|u (s)|^2 \mathd s.
\]
We compute
\begin{align*}
& H u(t)-H_\varepsilon u_\varepsilon(t)= \\=&(e^{-itH}-e^{-itH_\varepsilon})Hu_0+\int_{0}^{t}e^{-i(t-s)H}\partial_s(u|u|^2(s)) \mathd s-\int_{0}^{t}e^{-i(t-s)H_\varepsilon}\partial_s(u_\varepsilon|u_\varepsilon|^2(s)) \mathd s\\&+u|u|^2(t)-u_\varepsilon|u_\varepsilon|^2(t)\\
=&(e^{-itH}-e^{-itH_\varepsilon})Hu_0+\int_{0}^{t}(e^{-i(t-s)H}-e^{-i(t-s)H_\varepsilon})\partial_s(u|u|^2(s)) \mathd s\\&-\int_{0}^{t}e^{-i(t-s)H_\varepsilon}(\partial_s(u_\varepsilon|u_\varepsilon|^2(s))-\partial_s(u|u|^2(s))) \mathd s\\&+\int_{0}^{t}\partial_s(u|u|^2(s))-\partial_s(u_\varepsilon|u_\varepsilon|^2(s))\mathd s+u_0|u_0|^2-u^\varepsilon_0|u^\varepsilon_0|^2.
\end{align*}
Therefore, we have
\begin{align*}
\|H u(t)-H_\varepsilon u_\varepsilon(t)\|_{L^2}\lesssim& \|(e^{-itH}-e^{-itH_\varepsilon})Hu_0\|_{L^2}+\|u_0|u_0|^2-u^\varepsilon_0|u^\varepsilon_0|^2\|_{L^2}\\&+\int_{0}^{t}\|(e^{-i(t-s)H}-e^{-i(t-s)H_\varepsilon})\partial_s(u|u|^2(s))\|_{L^2} \mathd s\\
&+\int_{0}^{t}\|\partial_s(u|u|^2(s))-\partial_s(u_\varepsilon|u_\varepsilon|^2(s))\|_{L^2}\mathd s,
\end{align*}
where the first three terms converge to zero by norm resolvent convergence and Theorem VIII.21 in {\cite{reedsimon1}}. For the last term we can bound similarly to the proof of Theorem \ref{thm:2dnlsfix},
\begin{align*}
&\int_{0}^{t}\|\partial_s(u|u|^2(s))-\partial_s(u_\varepsilon|u_\varepsilon|^2(s))\|_{L^2}\mathd s\\ \lesssim&\int_{0}^{t}\|\partial_su(s)-\partial_su_\varepsilon(s)\|_{L^2}(\|u\|^2_{L^\infty L^\infty}+\|u_\varepsilon\|^2_{L^\infty L^\infty})\\&+\|u(s)-u_\varepsilon(s)\|_{L^\infty}\|\partial_t u\|_{L^\infty L^2}(\|u\|_{L^\infty L^\infty}+\|u_\varepsilon\|_{L^\infty L^\infty})\mathd s
\\\lesssim& C(T,u_0,\Xi)\int_{0}^{t}\|\partial_su(s)-\partial_su_\varepsilon(s)\|_{L^2}+\|Hu(s)-H_\varepsilon u_\varepsilon(s)\|_{L^2}\\&+\|u(s)- u_\varepsilon(s)\|_{L^2}\mathd s+C(T,u_0,\Xi)\|\Xi-\Xi_{\varepsilon}\|_{\mathscr{X}^\alpha},
\end{align*}
where we have used the a priori bounds obtained in the proof of Theorem \ref{thm:2dnlsfix} and the bound 
\begin{align*}
\|u(s)-u_\varepsilon(s)\|_{L^\infty} & \lesssim_\Xi \|u^\sharp(s)-u^\sharp_\varepsilon(s)\|_{\ssp^2} \\ &\lesssim_\Xi\|H u(s)-H_\varepsilon u_\varepsilon(s)\|_{L^2}+\|u(s)-u_\varepsilon(s)\|_{L^2}+C(T,u_0)\|\Xi-\Xi_{\varepsilon}\|_{\mathscr{X}^\alpha}
\end{align*}
which can be proved in the same way as Theorem \ref{lem:h2bound}, using Proposition \ref{lem:gamma} and the embedding $\ssp^2\hookrightarrow L^\infty.$\\
Similarly we can bound (by $O(\varepsilon)$ we denote terms that converge to zero as $\varepsilon\to0$)
\begin{align*}
& \|\partial_t u(t)-\partial_t u_\varepsilon(t)\|_{L^2}\\
\lesssim& \|(e^{-itH}-e^{-itH_\varepsilon})Hu_0\|_{L^2}+\|e^{-itH}u_0|u_0|^2-e^{-itH_\varepsilon}u^\varepsilon_0|u^\varepsilon_0|^2\|_{L^2}\\
&+\left| \left| \int_{0}^{t}e^{-isH}\partial_t(u|u|^2(t-s)) \mathd s-\int_{0}^{t}e^{-isH_\varepsilon}\partial_t(u_\varepsilon|u_\varepsilon|^2(t-s)) \mathd s\right| \right| _{L^2}\\
\lesssim& O(\varepsilon) +C(T,u_0,\Xi)\int_{0}^{t}\|\partial_su(s)-\partial_su_\varepsilon(s)\|_{L^2}\\ &+\|Hu(s)-H_\varepsilon u_\varepsilon(s)\|_{L^2}+\|u(s)- u_\varepsilon(s)\|_{L^2}\mathd s
\end{align*}
and we have
\begin{align*}
& \|u(t)-u_\varepsilon(t)\|_{L^2} \\
&\lesssim \|e^{-itH}u_0-e^{-itH_\varepsilon}u^\varepsilon_0\|_{L^2}+\left|\left|\int_{0}^{t}e^{-i(t-s)H}(u|u|^2(s)) \mathd s-\int_{0}^{t}e^{-i(t-s)H_\varepsilon}(u_\varepsilon|u_\varepsilon|^2(s)) \mathd s\right|\right|_{L^2}\\
&\lesssim O(\varepsilon) +C(T,u_0,\Xi)\int_{0}^{t}\|u(s)- u_\varepsilon(s)\|_{L^2}\mathd s.
\end{align*}
Thus, for $\phi_\varepsilon(t):= \|u(t)-u_\varepsilon(t)\|_{L^2}+\|\partial_tu(t)-\partial_tu_\varepsilon(t)\|_{L^2}+\|Hu(t)-H_\varepsilon u_\varepsilon(t)\|_{L^2}$ we have
\begin{align*}
\phi_\varepsilon(t)\lesssim O(\varepsilon)+\int_{0}^{t}\phi_\varepsilon(s)\mathd s
\end{align*}
and by Gronwall we can conclude that $\phi_\varepsilon(t)\to 0$ for all $t$ as $\varepsilon\to0.$

This finishes the proof.
\end{proof}

\begin{remark}\upshape\upshape\label{rem:nls3dagmon}
	Observe that the above also works in three dimensions.  The only difference is that one uses Lemma \ref{lem:3dagmon} instead of Theorem \ref{lem:brgaineq}.  But note that, this gives only local in time strong solutions and as in Theorem \ref{thm:2ddomnls}, we also obtain the convergence of solutions to the approximated PDEs.  This is due to the fact that, unlike the 2d case, one uses a polynomial type Gronwall \cite{Dra03}, as opposed to a logarithmic Gronwall, which leads to an estimate that blows up in finite time.  In fact, this can be formulated as a blow up alternative (with respect to the $L^\infty$-norm) similarly to the classical case of $\ssp^2$-solutions \cite{cazenave2003semilinear}.
\end{remark}
\subsubsection{Energy solutions}

In this section we solve \eqref{eqn:2dnls} in the energy space. For the global well-posedness of (standard) cubic NLS on the 2d torus see \cite{bourgain1993fourier}. Note that the result we get here is somewhat weaker, since we obtain  only existence and partial regularity in time.  But the Strichartz estimates in the case of Anderson Hamiltonian is not known and our result is still as good as what one would get in the classical case without the use of Strichartz estimates, please see \cite{BGT04} and references therein for further information.  In this section, we denote the dual of $\ED$ by $\ED^{\ast}$.  So,  we naturally have $\ED \subseteq L^2 \subseteq \ED^{\ast}$. 

\begin{theorem}\label{thm:2dnlsenergy} 
	For $u_0 \in \ED$, the equation \eqref{eqn:2dnls} has a solution $u$ such that $u \in C^{1 / 2} ([0, T]; L^2) \cap C([0,T]; \ED)$.
\end{theorem}
For the initial datum $u_0$, we construct the following approximation
\begin{align*}
u^\varepsilon_0:=(1+\varepsilon \sqrt{-H})^{-1}u_0\in \mathscr{D}(H).
\end{align*}
Note that by continuous functional calculus the operator $(- H)^{-1/2}:L^2\to \ED$ is bounded and we have $u^\varepsilon_0\to u_0$ in $\ED.$

\begin{lemma}\label{lem:energyconv}
	For $u_0,u^\varepsilon_0$ as above, we have the following convergence of energies.
	\begin{align*}
	E_\varepsilon(u^\varepsilon_0):=-\frac{1}{2}\langle u^\varepsilon_0,H u^\varepsilon_0\rangle+\frac{1}{4}\int|u_0^\varepsilon|^4\to-\frac{1}{2}\langle u_0,H u_0\rangle+\frac{1}{4}\int|u_0|^4=E(u_0)
	\end{align*}
\end{lemma}
\begin{proof}
	By the above observation the first terms converge.  For the $L^4$ terms, we can conclude using Lemma \ref{lem:2dlpbounds} and the $\ED$ convergence.

\end{proof}
Consider the nonlinear Schr{\"o}dinger equation
\begin{align} \label{eqn:nlsueps}
i\partial_t u_{\varepsilon} &=  H u_{\varepsilon}-u_\varepsilon|u_\varepsilon|^2  \\ \nonumber
u_\varepsilon(0)&=u^\varepsilon_0\in \mathscr{D}(H). \end{align}

As we have seen in section \ref{subsection:5.4}, there exists a unique solution $u_{\varepsilon}$ to this
equation in $C ([0, T] ; \mathscr{D}(H)) \cap C^1 ([0, T] ; L^2)$ which conserves the
energy
\[ E (u_{\varepsilon} (0))=E (u_{\varepsilon} (t)) := -\frac{1}{2} \langle u_{\varepsilon} (t), H u_{\varepsilon}
(t) \rangle + \frac{1}{4}\| u_{\varepsilon} (t) \|_{L^4}^4 . \]

\begin{lemma}[A priori bounds]\label{lem:energyapriori}
	For solutions $u_\varepsilon$ to \eqref{eqn:nlsueps}, we have the following uniform bounds. 
	\begin{align*}
	||u_\varepsilon||_{L^\infty L^2}&\lesssim||u_0||_{L^2}\\
	||(-H)^{1/2}u_\varepsilon||_{L^\infty L^2}&\lesssim E^{1/2}(u_0)\\
	||(-H)^{-1/2}\partial_t u_\varepsilon||_{L^\infty L^2}&\lesssim_\Xi E^{1/2}(u_0)+E^{3/2}(u_0)
	\end{align*}
\end{lemma}
\begin{proof}
	Since we have conservation of mass and energy, the first and second follow directly, using also Lemma \ref{lem:energyconv} and the positivity of the energy. 
	For the third bound, we use the equation and the fact that \[
	||u_\varepsilon||^3_{L^6}\lesssim_\Xi
	E^{3/2}(u_0),
	\]
	which follows from Lemma \ref{lem:2dlpbounds}. 

\end{proof}

\begin{lemma}[Compactness]\label{lem:energycompact}
	Given $u_\varepsilon$ as above, we can extract a subsequence $u_{\varepsilon_n}$ and obtain a limit $u\in L^\infty([0,T];\ED)$ s.t.
	\begin{align}
	u_{\varepsilon_n}(t)&\to u(t) \textrm{ in } L^2\\
	(-H)^{1/2} u_{\varepsilon_n}(t)&\to (-H)^{1/2} u(t) \text{ in } L^2
	\end{align}
	for all 
	times $t\in[0,T].$
\end{lemma}
\begin{proof}
By weak compactness in the Hilbert space $\ED$ we obtain a subsequence $u_{\varepsilon_n}$ and a limit $u$ s.t.
\begin{align*}
u_{\varepsilon_n}(t)&\to u(t) \text{ in } L^2,\\
 u_{\varepsilon_n}(t)&\rightharpoonup u(t) \text{ in } \ED,
\end{align*}	
	for a dense set of times and using the third a priori bound from Lemma \ref{lem:energyapriori} we can extend this to all times $t\in[0,T]$. In particular get the $L^\infty$ bound in time. Lastly, we can use the convergence of energies to deduce the convergence of the $\ED$ norms of $u_\varepsilon$ and thus conclude that in fact strong convergence holds.

\end{proof}

Now we can conclude this section by proving Theorem \ref{thm:2dnlsenergy}.

\begin{proof}[Proof of Theorem \ref{thm:2dnlsenergy}]
	We prove that the limit we obtain in the previous lemma solves the mild formulation of \eqref{eqn:2dnls}. We have by construction that the $u_\varepsilon$ solves
	\begin{align*}
	u_\varepsilon(t)=e^{-itH}u^\varepsilon_0+i\int_{0}^{t}e^{i(s-t)H}u_\varepsilon(s)|u_\varepsilon(s)|^2\mathrm{d}s
	\end{align*}
	for all $t\in[0,T].$ Now we can prove that this converges in $L^2$ as $\varepsilon\to0$ for all times. The first term converges precisely as in the linear case from section \ref{sec:linschr}. For the nonlinear term the convergence follows from the fact that $u_\varepsilon(t)\to u(t)$ strongly in $L^6$ for all times. This is due to the fact that the embedding $\ED\hookrightarrow \ssp^{1-\delta}$ is continuous and the embedding $\ssp^{1-\delta}\hookrightarrow L^6$ is compact (in fact this is true for any $L^p$ with $p<\infty$).

	For continuity in $\mathscr{D}(\sqrt{-H})$, we simply observe 
	\begin{align*}
	||\sqrt{-H} u(t) - \sqrt{-H}u(s)||_{L^2} &\leq ||\sqrt{-H} u(t) - \sqrt{-H}u_{\varepsilon_n}(t)||_{L^2} + ||\sqrt{-H} u_{\varepsilon_n}(t) - \sqrt{-H}u_{\varepsilon_n}(s)||_{L^2} \\
	&+||\sqrt{-H} u_{\varepsilon_n}(s) - \sqrt{-H}u(s)||_{L^2}.
	\end{align*}
 By using Lemma \ref{lem:energycompact}, for a given $\delta >0$ we can  choose $N$ large such that
 \[
 \sup_{\tau} ||\sqrt{-H} u(\tau) - \sqrt{-H}u_{\varepsilon_N}(\tau)||_{L^2} < \delta/3
 \]
 for this chosen $N$ we can choose $\kappa >0$ such that; $|t-s|< \kappa$ implies
 \[
 ||\sqrt{-H} u_{\varepsilon_N}(t) - \sqrt{-H}u_{\varepsilon_N}(s)||_{L^2} < \delta/3.
 \]
 That is, we have found a $\kappa>0$ for arbitrary $\delta>0$.  Hence, the continuity. 
	
	Next, we prove the time regularity.
By using Lemma \ref{lem:energyapriori} we can write
\begin{align*}
||u(t) - u(s)||_2^2 &\leq ||H^{1/2}(u(t) - u(s)) ||_2 ||H^{-1/2}(u(t) - u(s)) ||_2\\ &\lesssim \left| \left|\int_s^t H^{-1/2} \partial_t u(\tau) d\tau \right|\right|_2 \lesssim |t-s|.
\end{align*} So, we can conclude that $$u \in C^{1 / 2} ([0, T], L^2)~\cap~ C([0,T]; \ED).$$

\end{proof}

In the following corollary, we show that a solution can be obtained by solving the approximating PDEs.
\begin{corollary}\label{corr:energyregconv}
	Consider the following PDE
	\[
	i \partial_tu_\varepsilon = H_\varepsilon u_\varepsilon - u_\varepsilon|u_\varepsilon|^2
	\]
	with initial data $u_0^\varepsilon = H_\varepsilon^{-1} H (1- \varepsilon \sqrt {- H})^{-1} u_0$, where $u_0 \in \mathscr{D} (\sqrt {- H})$ and $0 < \varepsilon <1$.  There exists a subsequence $\varepsilon_n$ such that $u_{\varepsilon_n} \to u$ and  $\sqrt{-H_{\varepsilon_n}}  u_{\varepsilon_n } \to \sqrt{-H}u$  in $L^2$.  In addition, $u$ solves \ref{eqn:2dnls}.
	\end{corollary}

\begin{proof}
Consider the initial data $u_0^{\varepsilon, \delta} = H_\varepsilon^{-1} H (1- \delta \sqrt {- H})^{-1} u_0$.  Then, by Theorem \ref{thm:2ddomnls},  taking $\epsilon \to 0$  we obtain $u_{\delta} \in \mathscr{D} (H)$ which solves  the equation
\[
i \partial_tu_{\delta} = H u_{\delta} - u_{\delta}|u_{\delta}|^2
\]
with initial data $u_0^{\delta} = (1- \delta \sqrt {- H})^{-1} u_0 \in \mathscr{D}(H)$.  For this solution, we also  have $\sqrt{-H_{\varepsilon_n}}  u_{\varepsilon_n, \delta } \to \sqrt{-H}u_\delta $  in $L^2$ and  in particular $u_{\varepsilon_n, \delta } \to u_\delta $.  Now, as in Theorem \ref{thm:2dnlsenergy}, we take $\delta \to 0$ and obtain an energy solution to   \eqref{eqn:2dnls}.  Taking a diagonal sequence yields the stated result.
\end{proof}

In the following remarks, we compare those results with the ones in domain case.
\begin{remark}\upshape\upshape
	Note that the solution we obtain is not necessarily unique, as opposed to the solution with initial data in $\mathscr{D}(H)$.
\end{remark}
\begin{remark}\upshape\upshape
Observe that this result holds for the NLS treated here, with any power nonlinearity as in the domain case.
\end{remark}

\subsection{Two and three dimensional cubic wave equations} \label{subsec:twothreewave}

In this section, we consider the cubic wave equations
\begin{align}
	\partial^2_t u  = & H u - u^3 \ \text{on} \  \mathbbm{T}^d \label{eqn:Hnlw}\\
	(u, \partial_t u) |_{t = 0}  = & (u_0, u_1),\nonumber
\end{align}
in two and three dimensions simultaneously.

We are interested in the case
\[ (u_0, u_1) \in \mathscr{D} (H) \times \mathscr{D} ( \sqrt{- H}). \]

However as we shall see, in a similar way, we can also consider the case
\[ (u_0, u_1) \in \mathscr{D} ( \sqrt{- H}
) \times  L^2 . \]

We refer to  \cite{tao2006nonlinear} and \cite{Evans10} for classical results about well-posedness of semilinear wave equations. We obtain global strong well-posedness for a range of exponents including the standard case $p=3,$ which we will consider in detail for simplicity. In also 3d,  the range of exponents which are covered by our methods is as good as what one can achieve in the classical case with similar methods.  

We fix an approximating sequence $(u^{\varepsilon}_0, u^{\varepsilon}_1) \in \ssp^2
\times \ssp^1$ such that
\begin{align*}
H_{\varepsilon} u^{\varepsilon}_0 &\rightarrow H u_0 \ \tmop{in} \ L^2,\\
(u^{\varepsilon}_1, H_{\varepsilon} u^{\varepsilon}_1)
&\rightarrow (u_1, H u_1) . 
\end{align*}
To be precise, we  choose
\begin{align*}
u_0^\varepsilon&:=(-H_\varepsilon)^{-1}(-H) u_0\\
u_1^\varepsilon&:=(-H_\varepsilon)^{-1/2}(-H)^{1/2} u_1.
\end{align*}
We will, as in the NLS case, prove that the solution to \eqref{eqn:Hnlw} is the limit of the solutions of the regularized equations (for $d=2,3$)
\begin{align}
\partial^2_t u_{\varepsilon}  = & H_{\varepsilon} u_{\varepsilon} -
u_{\varepsilon}^3 \ \tmop{on} \  \mathbbm{T}^d \label{eqn:Hepsnlw}\\
(u_{\varepsilon}, \partial_t u_{\varepsilon}) |_{t = 0}  = &
(u^{\varepsilon}_0, u^{\varepsilon}_1),\nonumber
\end{align}
in an appropriate sense.

We begin by proving global strong wellposedness of \eqref{eqn:Hnlw} and \eqref{eqn:Hepsnlw} by a  fixed point argument as in section \ref{subsection:5.4}.

\begin{theorem}\label{thm:wavefix}
	For $(u_0,u_1)\in\mathscr{D}(H)\times\mathscr{D}(\sqrt{-H})$ and $(u^\varepsilon_0,u^\varepsilon_1)\in \ssp^2\times \ssp^1$, there exist unique global in time solutions $u\in C([0,T];\mathscr{D}(H))\cap C^1([0,T];\mathscr{D}(\sqrt{-H}))\cap C^2([0,T];L^2)$ and $u_\varepsilon\in C([0,T];\ssp^2)\cap C^1([0,T];\ssp^1)\cap C^2([0,T];L^2)$ satisfying
	\begin{align*}
	u(t) = \cos \left( t \sqrt{- H} \right) u_0 + \frac{\sin \left( t
		\sqrt{- H} \right)}{\sqrt{- H}} u_1 +\int_{0}^{t}
	\frac{\sin \left( (t-s) \sqrt{- H} \right)}{\sqrt{- H}} u^3 (s) \tmop{ds}
	\end{align*}
	and
	\begin{align*}
	u_\varepsilon(t) = \cos \left( t \sqrt{- H_\varepsilon} \right) u^\varepsilon_0 + \frac{\sin \left( t
		\sqrt{- H_\varepsilon} \right)}{\sqrt{- H_\varepsilon}} u^\varepsilon_1 + \int_{0}^{t}
	\frac{\sin \left( (t-s) \sqrt{- H_\varepsilon} \right)}{\sqrt{- H_\varepsilon}} u_\varepsilon^3 ( s) \tmop{ds}
	\end{align*}
	respectively.
\end{theorem}
Before we come to the proof, we prove some auxiliary lemmas. 
We define the conserved energies  for  \eqref{eqn:Hnlw} and \eqref{eqn:Hepsnlw} respectively as
\[ E (u) \assign \frac{1}{2} \langle \partial_t u,
\partial_t u \rangle - \frac{1}{2} \langle u,
H u \rangle + \frac{1}{4} \int |
u |^4, \]
and
\[ E (u_{\varepsilon}) \assign \frac{1}{2} \langle \partial_t u_{\varepsilon},
\partial_t u_{\varepsilon} \rangle - \frac{1}{2} \langle u_{\varepsilon},
H_{\varepsilon} u_{\varepsilon} \rangle + \frac{1}{4} \int |
u_{\varepsilon} |^4. \]

Also, we introduce the {\tmem{almost conserved}} energies for the time derivatives
\[ \tilde{E} (\partial_t u) = \frac{1}{2} \langle \partial^2_t
u, \partial^2_t u \rangle - \frac{1}{2} \langle
\partial_t u, H \partial_t u
\rangle + \frac{3}{2} \int | u |^2 | \partial_t
u |^2, \]
and
\[ \tilde{E} (\partial_t u_{\varepsilon}) = \frac{1}{2} \langle \partial^2_t
u_{\varepsilon}, \partial^2_t u_{\varepsilon} \rangle - \frac{1}{2} \langle
\partial_t u_{\varepsilon}, H_{\varepsilon} \partial_t u_{\varepsilon}
\rangle + \frac{3}{2} \int | u_{\varepsilon} |^2 | \partial_t
u_{\varepsilon} |^2. \]
We clarify what we mean by almost conserved in the following lemma.

\begin{lemma}
	Let $u\in C([0,T];\mathscr{D}(H))\cap C^1([0,T];\mathscr{D}(\sqrt{-H}))\cap C^2([0,T];L^2)$ and $u_\varepsilon\in C([0,T];\ssp^2)\cap C^1([0,T];\ssp^1)\cap C^2([0,T];L^2)$ be solutions of \eqref{eqn:Hnlw} and \eqref{eqn:Hepsnlw} respectively. Then the energies $\tilde{E} (\partial_t u)$ and $\tilde{E} (\partial_t u_{\varepsilon})$ satisfy the following
	bounds
	\begin{align*}
	\tilde{E} (\partial_t u) (t) & \lesssim \exp (tC
	\tilde{E} (u_1)) E (u_0), \\
	\tilde{E} (\partial_t u_{\varepsilon}) (t) &\lesssim \exp (tC
	\tilde{E} (u^{\varepsilon}_1)) E (u^{\varepsilon}_0) ,
	\end{align*}
for some universal constant $C > 0.$

\end{lemma}

\begin{proof}
	We give the proof only for the regularized case. The other case can be done analogously by replacing $\ssp^2$ by $\mathscr{D}(H)$ and  $\ssp^1$ by ${\sqrt{-H}}$. \\
	First note that $\partial_t u_{\varepsilon}$ solves the equation
	\[ \begin{array}{lll}
	\partial^2_t \partial_t u_{\varepsilon} & = & H_{\varepsilon}
	\partial_t u_{\varepsilon} - 3 \partial_t u_{\varepsilon}
	u_{\varepsilon}^2 \ \tmop{in} \  C([0,T];\ssp^{- 1}).
	\end{array} \]
	Then one can formally compute
	\begin{eqnarray*}
		\frac{\mathd}{\tmop{dt}} \tilde{E} (\partial_t u_{\varepsilon}) (t) & = &
		\langle \partial^2_t u, \partial^3_t u_{\varepsilon} - H_{\varepsilon}
		\partial_t u_{\varepsilon} + 3 \partial_t u_{\varepsilon}
		u_{\varepsilon}^2  \rangle + 3 \int u_{\varepsilon} \partial_t
		u_{\varepsilon} | \partial_t u_{\varepsilon} |^2\\
		& = & 3 \int u_{\varepsilon} \partial_t u_{\varepsilon} | \partial_t
		u_{\varepsilon} |^2 .
	\end{eqnarray*}
	and conclude by Gronwall. However, since $\tilde{E} (\partial_t u_{\varepsilon})$ is not $C^1$ in time, this is
	not justified.  But one can argue that this
	computation is true in the integrated version. 
	We claim that we get the following weak differentiability, for any $\phi \in C_{c} ([0,
	\infty))$
	\begin{equation}
	\underset{\mathbbm{R}}{\int} \phi' (t) \tilde{E} (\partial_t
	u_{\varepsilon}) (t) \tmop{dt} = - 3 \underset{\mathbbm{R}}{\int} \phi (t)
	\int u_{\varepsilon} \partial_t u_{\varepsilon} | \partial_t
	u_{\varepsilon} |^2 (t) \tmop{dt} +\tilde{E}(\partial_tu_\varepsilon(0))\phi(0). \label{eqn:3ddtenergy}
	\end{equation}
	Moreover, this also holds in the integrated form
	\begin{equation}
	\tilde{E} (\partial_t u_{\varepsilon}) (t) = \tilde{E} (\partial_t
	u_{\varepsilon}) (0) + 3  \int_{0}^{t}
	u_{\varepsilon} \partial_t u_{\varepsilon} | \partial_t u_{\varepsilon}
	|^2 (s) \tmop{ds},  \label{eqn:3ddtenint}
	\end{equation}
	for any $t\in[0,T].$
	We prove this by a spectral approximation. For, consider $(e_n)_{n \in
		\mathbbm{Z}^3} \in \ssp^2$ an orthonormal eigenbasis of $H_{\varepsilon}$ with
	eigenvalues $\{ \lambda_n\} $ and set
	\[ u^N_{\varepsilon} (t, x) \assign \underset{| n | \le N}{\sum}
	(u_{\varepsilon} (t, \cdummy), e_n) e_n (x) . \]
	Then one has
	\begin{align*}
	\partial_t^{k}u^N_{\varepsilon}  \rightarrow  \partial^k_tu_{\varepsilon} \ \tmop{in} \ C([0,T];\ssp^{2-k})\\
\end{align*}
	for $0\le k\le 3,$
	which in turn implies that
	\[ E (u^N_{\varepsilon}) \rightarrow E (u_{\varepsilon}) \ \tmop{and}\ 
	\tilde{E} (\partial_t u^N_{\varepsilon}) \rightarrow \tilde{E}
	(\partial_t u_{\varepsilon}) . \]
	One also directly deduces
	\begin{align*}
	\partial^3_t u^N_{\varepsilon}  = H_{\varepsilon} u^N_{\varepsilon} - 3 \underset{| n | \le
		N}{\sum} (\partial_t u_{\varepsilon} u^2_{\varepsilon} (t), e_n) e_n (x) .
	\end{align*}
	Thus, we have 
	\begin{align}
	\underset{\mathbb{R}}{\int} \phi' (t) \tilde{E} (\partial_t
	u^N_{\varepsilon}) (t) \tmop{dt}  =&  - \underset{\mathbb{R}}{\int} \phi
	(t) \frac{\mathd}{\tmop{dt}} \tilde{E} (\partial_t u^N_{\varepsilon}) (t)
	\tmop{dt} +\tilde{E}(\partial_tu_\varepsilon(0))\phi(0)\nonumber\\
	=&  - \underset{\mathbb{R}}{\int} \phi (t) \left( (\partial^2_t
	u^N_{\varepsilon}, \partial^3_t u^N_{\varepsilon}) (t) - (\partial^2_t 
	u^N_{\varepsilon}, H_{\varepsilon} \partial_t u^N_{\varepsilon}) (t)\right.\nonumber\\
	& \left.  + 3
	\left( \partial^2_t u^N_{\varepsilon}, \partial_t u^N_{\varepsilon}
	(u^N_{\varepsilon})^2 \right) (t) + 3 (\partial_t u^N_{\varepsilon},
	(\partial_t u^N_{\varepsilon}) u^N_{\varepsilon}) (t) \right) \tmop{dt} \\
	&+\tilde{E}(\partial_tu_\varepsilon(0))\phi(0)\nonumber\\
	=&  - \underset{\mathbb{R}}{\int} \phi (t) [ 3 (
	\partial^2_t u^N_{\varepsilon}, \partial_t u^N_{\varepsilon}
	(u^N_{\varepsilon})^2 - \underset{| n | \le N}{\sum} (\partial_t
	u_{\varepsilon} u^2_{\varepsilon}, e_n) e_n (x) ) (t) \nonumber\\  &+  3
	(\partial_t u^N_{\varepsilon}, (\partial_t u^N_{\varepsilon})^2
	u^N_{\varepsilon}) (t) ] \tmop{dt} +\tilde{E}(\partial_tu_\varepsilon(0))\phi(0). \hspace{3em}
	\label{eqn:dtenergyN}
	\end{align}
	Now,  we can write
	\begin{eqnarray*}
		\partial_t u^N_{\varepsilon} (u^N_{\varepsilon})^2 & \rightarrow &
		\partial_t u_{\varepsilon} (u_{\varepsilon})^2\  \ \tmop{in} \  L^2\\
		& \tmop{and} & \\
		\underset{| n | \le N}{\sum} (\partial_t u_{\varepsilon}
		u^2_{\varepsilon}, e_n) e_n (x) & \rightarrow & \partial_t u_{\varepsilon}
		(u_{\varepsilon})^2
		 \ \tmop{in} \ L^2.
	\end{eqnarray*}
	Therefore,  we see that for $N \rightarrow \infty$ (\ref{eqn:dtenergyN}) converges
	to (\ref{eqn:3ddtenergy}).
	To prove (\ref{eqn:3ddtenint}), it suffices to take a sequence $\phi_n$ in
	(\ref{eqn:3ddtenergy}) that converges to the characteristic function
	$\chi_{[0, t]}$ monotonically.	
	We can thus compute
	\begin{align*}
	\tilde{E} (\partial_t u_{\varepsilon}) (t)  \le & \tilde{E}
	(u_1^{\varepsilon}) + 3
	\int_{0}^{t} \int |
		u_{\varepsilon} (s) | | \partial_t u_{\varepsilon} (s) | | \partial_t
		u_{\varepsilon} |^2 (s) \tmop{ds}\\
	\le & \tilde{E} (u_1^{\varepsilon}) + 3
	\int_{0}^{t} \| \partial_t u_{\varepsilon} \|_{L^2} \|
	u_{\varepsilon} \|_{L^6} \| \partial_t u_{\varepsilon} \|^2_{L^6} (s)
	\tmop{ds}\\
	\le & \tilde{E} (u_1^{\varepsilon}) + 3 CE
	(u_0^{\varepsilon}) \int_{0}^{t} \tilde{E}
		(\partial_t u_{\varepsilon}) (s) \tmop{ds},
	\end{align*}
	where we have used the bounds
	\begin{equation*}
	\| \partial_t u_{\varepsilon} \|^2_{L^2} \le E (u_{\varepsilon}), \qquad
	\| u_{\varepsilon} \|^2_{L^6} 
	 \lesssim_\Xi E (u_{\varepsilon}), \qquad
	 \| \partial_t u_{\varepsilon}
	\|^2_{L^6} \lesssim_\Xi
	\tilde{E} (\partial_t u_{\varepsilon}) . 
	\end{equation*}
	From this, we conclude by using Gronwall. 
\end{proof}

\begin{proof}[Proof of Theorem \ref{thm:wavefix}]
	This is similar to the NLS case (Section \ref{subsection:5.4}), except that the time $T_E$ is going to depend
	on the conserved energy $E (u)$ and the almost conserved energy $\tilde{E}
	(\partial_t u) .$ We again give the proof only for the $\mathscr{D}(H)$ case,  as the $\ssp^2$ case can be proved in a similar way.\\
	We claim that for $(u_0, u_1) \in \mathscr{D} (H) \times \mathscr{D}
	\left( \sqrt{- H} \right)$ there exists a unique fixed point of
	\begin{align*}
	\Phi (u)(t)= \cos \left( t \sqrt{- H} \right) u_0 + \frac{\sin \left( t
		\sqrt{- H} \right)}{\sqrt{- H}} u_1 + \int_{0}^{t}
	\frac{\sin \left( (t-s) \sqrt{- H} \right)}{\sqrt{- H}} u^3 ( s) \tmop{ds}
	\end{align*}
	in $X = C ([0, T] ; \mathscr{D} (H)) \cap C^1 \left( [0, T] ; \mathscr{D}
	\left( \sqrt{- H} \right) \right) \cap C^2 ([0, T] ; L^2)$.\\
	For the contraction property, we compute the following,  for $\|u\|_X\le M$ with $M>0$ fixed later, 
	\begin{align*}
	& \| H \Phi (u) (t) - H \Phi (v) (t) \|_{L^2} = \\
	=&\| \int_{0}^{t} \sqrt{-H} \sin((t-s)\sqrt{-H})(u^3(s)-v^3(s)) \tmop{ds} \|_{L^2}
	\\ 
	=&\| \int_{0}^{t} \partial_s( \cos((t-s)\sqrt{-H}))(u^3(s)-v^3(s)) \tmop{ds} \|_{L^2}\\
	=&\| \int_{0}^{t}  \cos((t-s)\sqrt{-H})\partial_s(u^3(s)-v^3(s)) \tmop{ds}+v^3(t)-u^3(t) \|_{L^2}
	\\ \le &2
	\int_{0}^{t} \| \partial_t (u^3 - v^3) (s) \|_{L^2}\\
	\le & 6 \int_{0}^{t} \| \partial_t u -
	\partial_t v \|_{L^{\infty} L^6} \| u \|^2_{L^6} (s) + \| \partial_t v
	\|_{L^{\infty} L^4} \| u - v \|_{L^{\infty} L^{\infty}} (\| u \|_{L^4} (s) +
	\| v \|_{L^4} (s)) \tmop{ds}\\
	\le & C \| u - v \|_X \int_{0}^{t} (E (u) (s) +
	M E^{1 / 2} (u) (s) + M E^{1 / 2} (v) (s)) \tmop{ds}\\
	< & \frac{1}{3} \| u - v \|_X
	\end{align*}
	for small enough time depending on the energy and $M$. Here we have used the bounds
	$\| \partial_t u \|_{L^6} \lesssim \left\| \sqrt{- H} \partial_t u
	\right\|_{L^2}$ and $\| u \|_{L^4} \lesssim E^{1/2} (u)$.
	For the other terms, we similarly compute 
	\begin{align*}
	\left\| \sqrt{- H} \partial_t \Phi (u) (t) - \sqrt{- H} \partial_t \Phi (v)
	(t) \right\|_{L^2}  \le &2\int_{0}^{t} \|
	\partial_t (u^3 - v^3) (s) \|_{L^2}\\
	< & \frac{1}{3} \| u - v \|_X
	\end{align*}
	and
	\begin{align*}
\| \partial^2_t \Phi (u) (t) - \partial_t^2 \Phi (v) (t) \|_{L^2}  < & \frac{1}{3} \| u - v \|_X .
	\end{align*}
	Lastly, we argue that $\Phi$ maps a ball to itself. Let $\| u \|_X \le M$
	for $M$ specified below, then we have
	\begin{align*}
	\| H \Phi (u) (t) \|_{L^2}  \lesssim & \| H u_0 \|_{L^2} + \| (- H)^{1 / 2}
	u_1 \|_{L^2} +\int_{0}^{t} \| \partial_t u \|_{L^6} (s)
	\| u \|^2_{L^6} (s) \tmop{ds}\\
	\lesssim & \| H u_0 \|_{L^2} + \| (- H)^{1 / 2} u_1 \|_{L^2} +
	\int_{0}^{t} \tilde{E}^{1 / 2} (\partial_t u) (s) E (u)
	(s) \tmop{ds}\\
	\le & \frac{M}{3}
	\end{align*}
	for large $M$ depending on the data and $t \le T_E$,  small depending on $E
	(u)$ and $\tilde{E} (\partial_t u)$.
	Analogously, we also have
	\[ \| \partial^2_t \Phi (u) \|_{L^{\infty} L^2} \le \frac{M}{3} \quad
	\tmop{and} \quad \left\| \sqrt{- H} \partial_t \Phi (u)
	\right\|_{L^{\infty} L^2} \le \frac{M}{3} . \]
	Moreover, the time regularity is again a consequence of Stone's Theorem. Thus there exists a unique strong solution up to the time $T_E$ that depends on (almost)
	conserved quantities and we can conclude that this yields a strong solution up to
	any time.  More precisely, we get a priori estimates that allow us to choose a globally valid $M>0$ and then iterate the solution map to obtain a solution up to any given time $T>0.$\\
	Assuming we have a solution on the interval $[0,T_E],$ then we can estimate similarly to above as,
	\begin{align*}
	\|H u(T_E)\|_{L^2}\lesssim & \| H u_0 \|_{L^2} + \| (- H)^{1 / 2} u_1 \|_{L^2} +
	\int_{0}^{T_E} \tilde{E}^{1 / 2} (\partial_t u) (s) E (u)
	(s) \tmop{ds}\\
	\lesssim& \| H u_0 \|_{L^2} + \| (- H)^{1 / 2} u_1 \|_{L^2} +
	T \exp(CT\tilde{E}(u_1))  E^{3/2} (u_0)
	\end{align*}
	and also  similarly for $\|\sqrt{-H} \partial_t u(T_E)\|_{L^2}.$ Thus we can choose $M$ globally and solve on the interval $[T_E,2T_E]$  and so on.
	
\end{proof}

From the above considerations, we obtain a priori bounds for the quantities
\[ \| u_{\varepsilon} \|_{L^\infty L^2}, \| H_{\varepsilon} u_{\varepsilon} \|_{L^\infty L^2}
\ \tmop{and} \ \sup\limits_{t\in[0,T]}(\partial_t u_{\varepsilon}, H_\varepsilon \partial_t u_{\varepsilon}), \]
independently of $\varepsilon$. By the same arguments, as in the
previous sections, we can also prove convergence of the approximate solutions.
\begin{theorem} \label{thm:waveapp}
	Assume we are in the above setting, i.e. we have unique global strong solutions to \eqref{eqn:Hnlw} and \eqref{eqn:Hepsnlw} and the initial data are given by $(u_0,u_1)\in\mathscr{D}(H)\times\mathscr{D}(\sqrt{-H})$ and 
	\begin{align*}
	u_0^\varepsilon&:=(-H_\varepsilon)^{-1}(-H) u_0\\
	u_1^\varepsilon&:=(-H_\varepsilon)^{-1/2}(-H)^{1/2} u_1.
	\end{align*}
	Then the solutions $u_\varepsilon$ converge to $u$ in the following way
		\begin{align*}
		u_{\varepsilon}(t)&\to u(t) \text{ in }L^2\\
		H_\varepsilon u_{\varepsilon}(t) &\to H u(t) \text{ in }L^2\\
		(-H_\varepsilon)^{1/2}\partial_tu_{\varepsilon}(t)&\to (-H)^{1/2}\partial_tu(t) \text{ in }L^2\\
		\partial_t^2u_{\varepsilon}(t)&\to \partial^2_tu(t) \text{ in }L^2
		\end{align*} 
		for all $t\in[0,T].$
\end{theorem}

\begin{proof}
The proof is similar to that of Theorem \ref{thm:2ddomnls}. Since we have strong convergence for the initial data, together with the fact that $\frac{\sin(t\sqrt{-H_\varepsilon})}{\sqrt{-H_\varepsilon}}\to\frac{\sin(t\sqrt{-H})}{\sqrt{-H}}$ strongly, we can bound for any fixed time $t\in[0,T]$ using the mild formulation for $u$ and $u_\varepsilon$.  We have
\begin{align*}
& \|Hu(t)-H_\varepsilon u_\varepsilon(t)\|_{L^2} \\ \lesssim& O(\varepsilon)+\int_{0}^{t}\|\partial_s(u^3)(s)-\partial_s(u^3_\varepsilon)(s)\|_{L^2}\mathd s\\
\lesssim& O(\varepsilon)+\int_{0}^{t}\|\partial_su(s)-\partial_su_\varepsilon(s)\|_{L^6}\|u\|^2_{L^\infty L^6}\\&+\|u(s)-u_\varepsilon(s)\|_{L^6}(\|u\|_{L^\infty L^6}+\|u_\varepsilon\|_{L^\infty L^6})\|\partial_t u_\varepsilon\|_{L^\infty L^6}\mathd s\\
\lesssim& O(\varepsilon)+C(T,u_0,u_1)\int_{0}^{t}\|\sqrt{-H}\partial_su(s)-\sqrt{-H_\varepsilon}\partial_su_\varepsilon(s)\|_{L^2}+\|Hu(s)-H_\varepsilon u_\varepsilon(s)\|_{L^2}\mathd s.
\end{align*}
Here,  we have used the a priori bounds obtained in Theorem \ref{thm:wavefix} and the estimate
\begin{align*}
\|\partial_t u(s)-\partial_t u_\varepsilon(s)\|_{L^6} & \lesssim_\Xi\|\partial_t u^\sharp(s)-\partial_t u^\sharp_\varepsilon(s)\|_{\ssp^1} \\
&\lesssim_\Xi \|\sqrt{-H}\partial_t u(s)-\sqrt{-H_\varepsilon}\partial_t u_\varepsilon(s)\|_{L^2}+C(u_0,u_1,T)\|\Xi-\Xi_\varepsilon\|_{\mathscr{X}^\alpha},
\end{align*}
where the first estimate follows by Sobolev embedding and Proposition \ref{lem:gamma} and \ref{lem:gamma3}. The second one can be proved analogously to Proposition \ref{lem:2dh1bound} and \ref{lem:3dh1} for 2 and 3d respectively. In a similar manner, we have the bound
\begin{align*}
\| u(s)-u_\varepsilon(s)\|_{L^6}\lesssim_\Xi \|Hu(s)-H_\varepsilon u_\varepsilon(s)\|_{L^2}+C(u_0,u_1,T)\|\Xi-\Xi_\varepsilon\|_{\mathscr{X}^\alpha}.
\end{align*}
Analogously we can also write
\begin{align*}
\|\sqrt{-H}\partial_tu(t)-\sqrt{-H_\varepsilon}\partial_tu_\varepsilon(t)\|_{L^2}\lesssim& O(\varepsilon)+C(T,u_0,u_1)\int_{0}^{t}\|\sqrt{-H}\partial_su(s)-\sqrt{-H_\varepsilon}\partial_su_\varepsilon(s)\|_{L^2}\\&+\|Hu(s)-H_\varepsilon u_\varepsilon(s)\|_{L^2}\mathd s,\\
\|\partial^2_tu(t)-\partial^2_tu_\varepsilon(t)\|_{L^2}\lesssim& O(\varepsilon)+C(T,u_0,u_1)\int_{0}^{t}\|\sqrt{-H}\partial_su(s)-\sqrt{-H_\varepsilon}\partial_su_\varepsilon(s)\|_{L^2}\\&+\|Hu(s)-H_\varepsilon u_\varepsilon(s)\|_{L^2}\mathd s\\
\|u(t)-u_\varepsilon(t)\|_{L^2}\lesssim& O(\varepsilon)+C(T,u_0,u_1)\int_{0}^{t}\|Hu(s)-H_\varepsilon u_\varepsilon(s)\|_{L^2}\mathd s.
\end{align*}
Thus, by defining 
\begin{align*}
 \phi_\varepsilon(t) &:= \|Hu(t)-H_\varepsilon u_\varepsilon(t)\|_{L^2}+\|\sqrt{-H}\partial_tu(t)-\sqrt{-H_\varepsilon}\partial_tu_\varepsilon(t)\|_{L^2}\\ &+\|\partial^2_tu(t)-\partial^2_tu_\varepsilon(t)\|_{L^2}+\|u(t)-u_\varepsilon(t)\|_{L^2},
\end{align*}
we can rewrite the above estimates as
\[
\phi_\varepsilon(t)\le O(\varepsilon)+C(T,u_0,u_1)\int_{0}^{t}\phi_\varepsilon(s)\mathd s
\]
and conclude by Gronwall that $\phi_\varepsilon(t)\to0$ as $\varepsilon\to 0$ for all $t\in[0,T].$\\
Hence, the result.
\end{proof}
Lastly, we state the analogous result for the energy space, i.e. with data $(u_0,u_1)\in\mathscr{D}(\sqrt{-H})\times L^2.$ In a nutshell, one can repeat the above arguments. For global well-posedness, one can use a fixed point argument in the space $C([0,T];\mathscr{D}(\sqrt{-H}))\cap C^1([0,T];L^2)\cap C^2([0,T];\mathscr{D}(\sqrt{-H})^\ast)$ together with energy conservation and  convergence can also be proved as above.  We omit the proofs.
\begin{theorem} \label{thm:waveenergy}
Let $(u_0,u_1)\in\mathscr{D}(\sqrt{-H})\times L^2$ and $T>0$, then \eqref{eqn:Hnlw} has a unique solution $u\in C([0,T];\mathscr{D}(\sqrt{-H}))\cap C^1([0,T];L^2)\cap C^2([0,T];\mathscr{D}(\sqrt{-H})^\ast)$. Moreover, \eqref{eqn:Hepsnlw} has a unique solution $u_\varepsilon\in C([0,T];\ssp^1)\cap C^1([0,T];L^2)\cap C^2([0,T];\ssp^{-1})$ with initial data $(u^\varepsilon_0,u_1)\in \ssp^1\times L^2$, where $u^\varepsilon_0:=(-H_\varepsilon)^{-1/2}(-H)^{1/2}u_0$ and the following convergence holds
\begin{align*}
u_\varepsilon(t)&\to u(t) \text{ in }L^2\\
 \sqrt{-H_\varepsilon}u_\varepsilon(t)&\to \sqrt{-H}u(t) \text{ in }L^2\\
 \partial_tu_\varepsilon(t)&\to \partial_tu(t) \text{ in }L^2
\end{align*}
for all $t\in[0,T].$
\end{theorem} 

\begin{remark}\upshape\label{rem:wave3ddiffpower}\upshape
The same result is also true in 2d for any power $p\in(1,\infty)$ both for the domain and energy space case. In 3d, our proof for global wellposedness also works for powers up to 5 in the domain case using an analogue of Agmon's inequality, which we included for completeness as Lemma \ref{lem:3dagmon}.
\end{remark}

\appendix
\section{Paracontrolled distributions and function spaces}
We recall the definitions of Bony paraproducts, Besov and Sobolev spaces 
and collect some results about products of distributions.
We work on the d-dimensional torus $\mathbb{T}^d:= \mathbb{R}^d/\mathbb{Z}^d$ for $d=2,3$.   
For any $f$ in the space $\mathscr {S}' ( \mathbb T^d,\mathbb R)$ of tempered distributions on $\mathbb{T}^d$, the Fourier transform of $f$ will be 
denoted by
$\hat{f} : \mathbb{Z}^d\to \mathbb{C}$ (or sometimes $\mathscr{F}f$) and is defined for $k\in \mathbb{Z}^d$ by
\begin{align*}
\hat{f}(k) := \langle f,\exp(2\pi i \langle k, \cdot\rangle) = \int_{\mathbb{T}^d} f(x) \exp(- 2\pi i\langle k,x\rangle) dx. 
\end{align*}
Recall that for any $f\in L^2(\mathbb{T}^d,\mathbb{R})$ and a.e. $x\in \mathbb{T}^d$, 
we have 
\begin{align}\label{decomposition-fourier}
f(x)=\sum_{k \in \mathbb{Z}^d} \hat{f}(k) \exp( 2\pi i\langle k,x\rangle). 
\end{align}

The Sobolev space $\ssp^\alpha(\mathbb T^d)$ with index $\alpha\in \mathbb{R}$ is defined as  
\begin{align*}
\ssp^\alpha(\mathbb {T}^d) := \{f \in \mathscr {S}' ( \mathbb {T}^d;\mathbb R): 
\sum_{k\in \mathbb{Z}^d} (1+|k|^2)^\alpha \, |\hat{f}(k)|^2 < +\infty\}\,. 
\end{align*}   

Before introducing the Besov spaces, we recall the definition of
Littlewood-Paley blocks.
We denote by $\chi$ and $\rho$ two nonnegative smooth and compactly supported 
radial functions $\mathbb{R}^d\to \mathbb{R}$ such that 
\begin{enumerate}
	\item The support of $\chi$ is contained in a ball $\{x\in \mathbb{R}^d: |x| \le R\}$ 
	and the support of $\rho$ is contained in an annulus $\{x\in \mathbb{R}^d: a\le |x| \le b\}$;
	\item For all $\xi \in \mathbb{R}^d$, $\chi(\xi)+\sum_{j\ge0}\rho(2^{-j}\xi)=1$;
	\item For $j\ge 1$, $\chi  \rho(2^{-j}\cdot)\equiv 0$ and $\rho(2^{-i}\cdot) \rho(2^{-j}\cdot)\equiv 0$
	for $|i-j|\ge 1$. 
\end{enumerate}
The Littlewood-Paley blocks $(\Delta_j)_{j\geq-1}$ acting on $f\in\mathscr S'(\mathbb {T}^d)$ 
are defined by 
\begin{align*}
\mathscr F(\Delta_{-1} f) = \chi  \hat{f}\ \mbox{ and\ for }\ j \ge 0, \quad \mathscr F(\Delta_j f) = \rho(2^{-j}.) \hat{f}.
\end{align*}
Note that, for $f\in\mathscr S'(\mathbb {T}^d)$, the Littlewood-Paley blocks $(\Delta_j f)_{j\ge -1}$ define 
smooth functions, as their Fourier transforms have compact supports. We also set, for $f\in  \mathscr S'$ and 
$j\ge 0$, 
\begin{align*}
S_j f := \sum_{i=-1}^{j-1} \Delta_i f
\end{align*} 
and note that $S_j f$ converges in the sense of distributions to $f$
as $j\to \infty$. \\
The Besov space with parameters $p,q \in [1,\infty),\alpha \in \mathbb{R}$ can now be defined as  
\begin{equation}\label{eq:Besov}
B_{p,q}^{\alpha}(\mathbb{T}^d, \mathbb{R}):=\left\{u\in \mathscr S'(\mathbb{T}^d); \quad ||u||_{ B_{p,q}^{\alpha}}=
\left(\sum_{j\geq-1}2^{jq\alpha}||\Delta_ju||^q_{L^p} \right)^{1/q} <+\infty\right\}.
\end{equation}
We also define the Besov-H\"older spaces
\begin{align*}
\mathscr {C}^{\alpha}:= B_{\infty,\infty}^{\alpha}
\end{align*}
which are naturally equipped with the norm 
$||f||_{\CC^\alpha}:=||f||_{B_{\infty,\infty}^{\alpha}}=\sup_{j\ge -1} 2^{j \alpha} 
||\Delta_j f||_{L^\infty}$. For $\alpha\in(0,1)$ these spaces coincide with the classical H\"older spaces.\\
We can formally decompose the product $fg$ of two distributions $f$ and $g$ as  
\begin{align*}
fg=f\prec g+f\circ g+f\succ g
\end{align*}
where
\begin{align*}
f\prec g :=\sum_{j\ge -1} \sum_{i=-1} ^{j-2} \Delta_i f \Delta_j g 
&&\text{and}&& f\succ g:= \sum_{j \ge -1} \sum_{i=-1}^{j-2} \Delta_i g \Delta_j f
\end{align*}
are usually referred to as the \textit{paraproducts} whereas 
\begin{align} \label{equ:resonantpro}
f\circ g:=\sum_{j\geq-1}\sum_{|i-j|\leq 1}\Delta_i f \Delta_j g 
\end{align} 
is called the \textit{resonant product}. \\
Moreover, we define the notations $f\preccurlyeq g:=f\prec g+f\circ g$ and $f\succcurlyeq g:=f\succ g+f\circ g$.\\
The paraproduct terms are always well defined irrespective of regularities.
The resonant product is a priori only well defined if the sum of regularities is strictly greater than zero. This is reminiscent of the well known 
fact that one can not multiply distributions in general.
The following result makes those comments precise and gives simple but extremely vital estimates for paraproducts.
\begin{proposition}[Bony estimates,\cite{allez_continuous_2015}]
	\label{lem:paraest} Let $\alpha, \beta \in \mathbb{R}$. We
	have the following bounds:
	\begin{enumerate}
		\item If $f \in L^2$ and $g \in \CC^{\beta}$, then
		
		$|| f \prec g ||_{\ssp^{\beta - \delta}} \leq C_{\delta, \beta} || f
		||_{L^2} || g ||_{\CC^{\beta}} 
		$ 
		for all $\delta > 0$.
		
		\item if $f \in \ssp^{\alpha}$ and $g \in L^{\infty}$ then
		
		$
		|| f \succ g ||_{\ssp^{\alpha}} \leq C_{\alpha, \beta} || f ||_{\ssp^{\alpha}}
		|| g ||_{\CC^{\beta}} \hspace{0.17em} .
		$
		
		\item If $\alpha < 0$, $f \in \ssp^{\alpha}$ and $g \in \CC^{\beta}$,
		then
		
		$
		|| f \prec g ||_{\ssp^{\alpha + \beta}} \leq C_{\alpha, \beta} || f
		||_{\ssp^{\alpha}} || g ||_{\CC^{\beta}} \hspace{0.17em} .
		$
		
		\item If $g \in \mathscr{C}^{\beta}$ and $f \in \ssp^{\alpha}$ for $\beta <
		0$ then
		
		$
		\|f \succ g\|_{\ssp^{\alpha + \beta}} \leq C_{\alpha, \beta} || f
		||_{\ssp^{\alpha}} || g ||_{\CC^{\beta}}
		$
		
		\item If $\alpha + \beta > 0$ and $f \in \ssp^{\alpha}$ and $g \in
		\CC^{\beta}$, then
		
		$
		| | f \circ g ||_{\ssp^{\alpha + \beta}} \leq
		C_{\alpha, \beta} || f ||_{\ssp^{\alpha}} || g ||_{\CC^{\beta}}
		\hspace{0.17em} .
		$
	\end{enumerate}
	where $C_{\alpha, \beta}$ is a finite positive constant.
\end{proposition}
\begin{proposition}
	\label{prop:commu} Given $\alpha \in (0, 1)$, $\beta, \gamma \in \mathbb{R}$
	such that $\beta + \gamma < 0$ and $\alpha + \beta + \gamma > 0$, 
	there exists a  trilinear operator $C$ with the following bound
	\begin{align*}
	|| C (f, g, h) ||_{\ssp^{\alpha + \beta + \gamma}} \lesssim || f
	||_{\ssp^{\alpha}} || g ||_{\CC^{\beta}} || h
	||_{\CC^{\gamma}}
	\end{align*}
	for all $f \in \ssp^{\alpha}$, $g \in \CC^{\beta}$ and $h \in
	\CC^{\gamma}$.
	
	The restriction of $C$  to the smooth functions satisfies 
	\begin{align*}
	C (f, g, h) = (f \prec g) \circ h - f (g \circ h).
	\end{align*}
	
\end{proposition}
\begin{proof}
This is a restatement of the result (commutator Lemma) in \cite{allez_continuous_2015}, and the proof follows the same lines, with slight modifications. 
\end{proof}

We also prove the following modified version of the above Proposition, which suits our framework.

\begin{proposition}
	\label{prop:commu2} Let $\alpha \in (0, 1)$, $\beta, \gamma \in \mathbb{R}$
	such that $\beta + \gamma < 0$ and $\alpha + \beta + \gamma > 0$. 
	Then, there exists a  trilinear operator $C_N$ with the following bound
	\begin{align*}
	|| C_N (f, g, h) ||_{\ssp^{\alpha + \beta + \gamma}} \lesssim || f
	||_{\ssp^{\alpha}} || g ||_{\CC^{\beta}} || h
	||_{\CC^{\gamma}}
	\end{align*}
	for all $f \in \ssp^{\alpha}$, $g \in \CC^{\beta}$ and $h \in
	\CC^{\gamma}$.
	
	The restriction of $C_N$  to the smooth functions satisfies 
	\begin{align*}
	C_N (f, g, h) &:= \left(\Delta_{> N} (f\prec g)\right))\circ h- f(g \circ h)
	\end{align*}
	
\end{proposition}
\begin{proof}
	Observe that we have
	\begin{align*}
	C (f, g, h) - C_N (f, g, h) &= \left(\Delta_{\leq N} (f\prec g)\right))\circ h.
	\end{align*}
	So, we only need to show
	\begin{align*}
	|| \left(\Delta_{\leq N} (f\prec g)\right))\circ h ||_{\ssp^{\alpha + \beta + \gamma}} \lesssim || f
	||_{\ssp^{\alpha}} || g ||_{\CC^{\beta}} || h
	||_{\CC^{\gamma}}.
	\end{align*}
	By product estimates, we obtain right away
	\begin{align*}
	|| \left(\Delta_{\leq N} (f\prec g)\right))\circ h ||_{\ssp^{\alpha + \beta + \gamma}} \lesssim || \left(\Delta_{\leq N} (f\prec g)\right))
	||_{\ssp^{\alpha+ \beta}} || h
	||_{\CC^{\gamma}}.
	\end{align*}
	We  need to show
	\[
	|| \left(\Delta_{\leq N} (f\prec g)\right))
	||_{\ssp^{\alpha+ \beta}} \lesssim || f
	||_{\ssp^{\alpha}} || g ||_{\CC^{\beta}}.
	\]
	We can write
	\begin{align*}
	|| \left(\Delta_{\leq N} (f\prec g)\right))
	||_{\ssp^{\alpha+ \beta}}^2 = \sum_{k=-1}^{\infty} 2^{2k(\alpha+\beta)} || \Delta_k \left(\Delta_{\leq N} (f\prec g)\right))||_{L^2}^2.
	\end{align*}
	
	By the support of Fourier transforms we have that $ \Delta_k \left(\Delta_{\leq N} (f\prec g)\right)) = 0$ for $k> N+1$ so we obtain
	\begin{align*}
	|| \left(\Delta_{\leq N} (f\prec g)\right))
	||_{\ssp^{\alpha+ \beta}}^2 = \sum_{k=-1}^{N+1} 2^{2k(\alpha+\beta)} || \Delta_k \left(\Delta_{\leq N} (f\prec g)\right))||_{L^2}^2.
	\end{align*}
	
	By using the convention $\Delta_{<k} f:= \sum_{i=-1}^{k-2} \Delta_k f$ we  rewrite
	
	\begin{align*}
	|| \left(\Delta_{\leq N} (f\prec g)\right))
	||_{\ssp^{\alpha+ \beta}}^2 = \sum_{k=-1}^{N+1} 2^{2k(\alpha+\beta)} || \Delta_k \left(\Delta_{\leq N} (\sum_{i=-1}^\infty \Delta_{<i} f \Delta_i g)\right))||_{L^2}^2.
	\end{align*}
	
	Again by support arguments this boils down to
	
	\begin{align*}
	|| \left(\Delta_{\leq N} (f\prec g)\right))
	||_{\ssp^{\alpha+ \beta}}^2 = \sum_{k=-1}^{N+1} 2^{2k(\alpha+\beta)} || \Delta_k \left(\Delta_{\leq N} (\sum_{i=-1}^{N+1} \Delta_{<i} f \Delta_i g)\right))||_{L^2}^2.
	\end{align*}
	
	Applying two successive Young's we obtain
	
	\begin{align*}
	|| \left(\Delta_{\leq N} (f\prec g)\right))
	||_{\ssp^{\alpha+ \beta}}^2 \leq  \sum_{k=-1}^{N+1} 2^{2k(\alpha+\beta)} ||\phi_k||_{L^1} || \phi_{\leq N}||_{L^1}    \sum_{i=-1}^{N+1} ||\Delta_{<i} f \Delta_i g||_{L^2}^2
	\end{align*}
	
	where on the right hand side we can write
	
	\begin{align*}
	& || \left(\Delta_{\leq N} (f\prec g)\right))
	||_{\ssp^{\alpha+ \beta}}^2 \leq  \sum_{k=-1}^{N+1} 2^{2k(\alpha+\beta)} ||\phi_k||_{L^1}^2 || \phi_{\leq N}||_{L^1}^2    \sum_{i=-1}^{N+1} ||\Delta_{<i} f \Delta_i g||_{L^2}^2\\
	&\leq \sum_{k=-1}^{N+1} 2^{(2k-2i)\beta} 2^{(2k-2i)\alpha}||\phi_k||_{L^1}^2 || \phi_{\leq N}||_{L^1}^2    \left( \sup_{1\leq i \leq N+1}2^{2 i \beta}||\Delta_i g||_{L^\infty}^2 \right) \sum_{i=-1}^{N+1} 2^{2i\alpha} ||\Delta_{<i} f||_{L^2}^2.
	\end{align*}
	
	At this point, it is clear that for a constant depending on $N$ we readily have
	
	\begin{align*}
	|| \left(\Delta_{\leq N} (f\prec g)\right))
	||_{\ssp^{\alpha+ \beta}}^2 \lesssim || f
	||_{\ssp^{\alpha}}^2 || g ||_{\CC^{\beta}}^2
	\end{align*}
	
and the result follows.

	\end{proof}

\begin{lemma}[Bernstein's inequality, \cite{gubinelli2015paracontrolled}]
	\label{lem:bernstein}Let $\mathscr{A}$ be an annulus and $\mathscr{B}$ be a
	ball. For any $k \in \mathbbm{N}, \lambda > 0,$and $1 \le p \le
	q \le \infty$ we have
	\begin{enumerate}
		\item if $u \in L^p (\mathbbm{R}^d) $ is such that $\tmop{supp}
		(\mathscr{F}u) \subset \lambda \mathscr{B}$ then
		\[ \underset{\mu \in \mathbbm{N}^d : | \mu | = k}{\max} \| \partial^{\mu}
		u \|_{L^q} \lesssim_k \lambda^{k + d \left( \frac{1}{p} - \frac{1}{q}
			\right)} \| u \|_{L^p} \]
		\item if $u \in L^p (\mathbbm{R}^d) $is such that $\tmop{supp}
		(\mathscr{F}u) \subset \lambda \mathscr{A}$ then
		\[ \lambda^k \| u \|_{L^p} \lesssim_k \underset{\mu \in \mathbbm{N}^d : |
			\mu | = k}{\max} \| \partial^{\mu} u \|_{L^p} . \]
	\end{enumerate}
\end{lemma}

\begin{proposition}[Paralinearisation,\cite{gubinelli2017kpz}]
	\label{lem:paralin} Let $\alpha \in (0, 1) $ and $F \in C^2 .$ Then there
	exists a locally bounded map $R_F : \CC^{\alpha} \rightarrow
	\CC^{2 \alpha}$ such that
	\[ F (f) = F' (f) \prec f + R_F (f) \ \tmop{for}   \tmop{all} \  f \in
	\CC^{\alpha} . \]
\end{proposition}

\begin{lemma}\label{lem:circadj}
	
		Let $\alpha, \beta, \gamma \in \mathbbm{R}$ with $\alpha + \beta <0$, $\alpha + \beta + \gamma \geq
	0$, and $f \in \ssp^{\alpha}, g \in \CC^{\beta}, h \in \ssp^{\gamma},$
	then there exists a map $D (f, g, h)$ with the following bound
	\begin{equation} \label{equ:commLemma}
	| D (f, g, h) | \lesssim \| g \|_{\CC^{\beta}} \| f
	\|_{\ssp^{\alpha}} \| h \|_{\ssp^{\gamma}} . \end{equation}

	Moreover the restriction of $D (f, g, h)$  to the smooth functions $f, g, h$ is as follows:
	\begin{align*}
	D (f, g, h) = \langle f, h \circ g \rangle - \langle f \prec g, h
	\rangle.
	\end{align*}

\end{lemma}

\begin{proof}
	We define
	\begin{align*}	
	D (f, g, h) := \left( \sum_{i \ge k - 1, | j - k | \le L} - \sum_{i \sim
		k, 1 < | j - k | \le L} \right) \langle \Delta_i f, \Delta_j h
	\Delta_k g \rangle.
	\end{align*}
	So we get, for some $\delta>0$,
	\begin{align*}
	| D (f, g, h) | &\lesssim \sum_{i \gtrsim k, j \sim k} | \langle \Delta_i
	f, \Delta_j h \Delta_k g \rangle |\\& \le \sum_{i \gtrsim k, j \sim k}
	\| \Delta_i f \|_{L^2} \| \Delta_j h \|_{L^2} \| \Delta_k g
	\|_{L^{\infty}} 
	\\&\le \| g \|_{\CC^{- 1 - \delta}} \sum_{i \gtrsim k, j \sim k} 2^{k (1
		+ \delta)} \| \Delta_i f \|_{L^2} \| \Delta_j h \|_{L^2} \\&\le \| g
	\|_{\CC^{- 1 - \delta}} \| f \|_{\ssp^{(1 + \delta) / 2}} \| h
	\|_{\ssp^{(1 + \delta) / 2}}
	\end{align*} 
	and this argument can be adapted to show \eqref{equ:commLemma} by simply observing $1 \leq 2^{k(\beta+ \alpha+ \gamma)}  = 2^{k\beta}2^{k ( \alpha+ \gamma)} $, since $\beta+ \alpha+ \gamma \geq 0$  .
	Moreover, for smooth functions $f,g,h $; we can compute
	\begin{align*}	
	\langle f, h \circ g \rangle - \langle f \prec g, h
	\rangle &= \sum_{i, | j - k | \le 1} \langle \Delta_i f,
	\Delta_j h \Delta_k g \rangle - \sum_{i < k - 1, j} \langle \Delta_i f
	\Delta_k g, \Delta_j h \rangle 
	\\ &= \left( \sum_{i, | j - k | \le 1} - \sum_{i < k - 1, | j - k |
		\le L} \right) \langle \Delta_i f \Delta_k g, \Delta_j h \rangle 
	\\&= \left( \sum_{i, | j - k | \le L} - \sum_{i < k - 1, | j - k |
		\le L} - \sum_{i, 1 < | j - k | \le L} \right) \langle
	\Delta_i f, \Delta_j h \Delta_k g \rangle 
	\\&= \left( \sum_{i \ge k - 1, | j - k | \le L} - \sum_{i, 1 < |
		j - k | \le L} \right) \langle \Delta_i f, \Delta_j h \Delta_k g
	\rangle 
	\\&= \left( \sum_{i \ge k - 1, | j - k | \le L} - \sum_{i \sim
		k, 1 < | j - k | \le L} \right) \langle \Delta_i f, \Delta_j h
	\Delta_k g \rangle = D (f, g, h).
	\end{align*}
	
	Hence the result.
\end{proof}

\begin{remark}\upshape\upshape
	Proposition \ref{lem:circadj} says that  the paraproduct is {\tmem{almost
		}}the adjoint of the resonant product, meaning up to a more regular
		remainder term as is often the case in paradifferential calculus.
	\end{remark}
	
	\begin{lemma}
		\label{lem:comm2}Let $f \in \ssp^{\alpha}, g \in \CC^{\beta},$ with
		$\alpha \in (0, 1), \beta \in \mathbbm{R},$  there exists a bilinear map $R (f, g)$ that satisfies the following bound
		\[ \| R (f, g) \|_{\ssp^{\alpha + \beta + 2}} \lesssim \| f
		\|_{\ssp^{\alpha}} \| g \|_{\CC^{\beta}}, \]
		and restricts to smooth functions as
		\[ R (f, g) = (1 - \Delta)^{- 1} (f \prec g) - f \prec (1 -
		\Delta)^{- 1} g. \]

		\begin{proof}
			The proof is basically a straightforward modification of the proof of {\cite[Proposition A.2]{allez_continuous_2015}}, which has almost the same statement.
		\end{proof}
	\end{lemma}

\bibliography{paracontrolled-wave.bib}{}

\begin{thebibliography}{10}

\bibitem{Ag65}
Shmuel Agmon.
\newblock {\em Lectures on elliptic boundary value problems}.
\newblock Prepared for publication by B. Frank Jones, Jr. with the assistance
  of George W. Batten, Jr. Van Nostrand Mathematical Studies, No. 2. D. Van
  Nostrand Co., Inc., Princeton, N.J.-Toronto-London, 1965.

\bibitem{allez_continuous_2015}
Romain Allez and Khalil Chouk.
\newblock The continuous {Anderson} {Hamiltonian} in dimension two.
\newblock {\em arXiv:1511.02718 [math]}, November 2015.
\newblock arXiv: 1511.02718.

\bibitem{bab71}
Ivo Babu{\v{s}}ka.
\newblock Error-bounds for finite element method.
\newblock {\em Numerische Mathematik}, 16(4):322--333, 1971.

\bibitem{bourgain1993fourier}
Jean Bourgain.
\newblock Fourier transform restriction phenomena for certain lattice subsets
  and applications to nonlinear evolution equations.
\newblock {\em Geometric \& Functional Analysis GAFA}, 3(2):209--262, 1993.

\bibitem{brga80}
Haim Brezis and Thierry Gallouet.
\newblock Nonlinear {S}chr{\"o}dinger evolution equations.
\newblock {\em Nonlinear Analysis: Theory, Methods \& Applications},
  4(4):677--681, 1980.

\bibitem{BGT04}
N.~Burq, P.~G\'erard, and N.~Tzvetkov.
\newblock Strichartz inequalities and the nonlinear {S}chr\"odinger equation on
  compact manifolds.
\newblock {\em Amer. J. Math.}, 126(3):569--605, 2004.

\bibitem{cannizzaro2015}
Giuseppe Cannizzaro and Khalil Chouk.
\newblock Multidimensional {SDE}s with singular drift and universal
  construction of the polymer measure with white noise potential.
\newblock {\em arXiv preprint arXiv:1501.04751}, 2015.

\bibitem{cazenave2003semilinear}
Thierry Cazenave.
\newblock {\em Semilinear {Schr\"odinger} {Equations}}.
\newblock American Mathematical Soc., 2003.

\bibitem{debussche2016schr}
Arnaud Debussche and Hendrik Weber.
\newblock The {Schr\"odinger} equation with spatial white noise potential.
\newblock {\em arXiv preprint arXiv:1612.02230}, 2016.

\bibitem{Dra03}
Sever~Silvestru Dragomir.
\newblock {\em Some {G}ronwall type inequalities and applications}.
\newblock Nova Science Publishers, Inc., Hauppauge, NY, 2003.

\bibitem{Evans10}
Lawrence~C. Evans.
\newblock {\em {P}artial {D}ifferential {E}quations}, volume~19 of {\em
  Graduate Studies in Mathematics}.
\newblock American Mathematical Society, Providence, RI, second edition, 2010.

\bibitem{gubinelli_renormalization_2017}
M.~Gubinelli, H.~Koch, and T.~Oh.
\newblock Renormalization of the two-dimensional stochastic nonlinear wave
  equation.
\newblock {\em arXiv:1703.05461 [math]}, March 2017.
\newblock arXiv: 1703.05461.

\bibitem{GP17}
M.~{Gubinelli} and N.~{Perkowski}.
\newblock {An introduction to singular SPDEs}.
\newblock {\em ArXiv e-prints}, February 2017.

\bibitem{gubinelli2015paracontrolled}
Massimiliano Gubinelli, Peter Imkeller, and Nicolas Perkowski.
\newblock Paracontrolled distributions and singular {PDE}s.
\newblock In {\em Forum of Mathematics, Pi}, volume~3. Cambridge University
  Press, 2015.

\bibitem{gubinelli2017kpz}
Massimiliano Gubinelli and Nicolas Perkowski.
\newblock {KPZ} reloaded.
\newblock {\em Communications in Mathematical Physics}, 349(1):165--269, 2017.

\bibitem{hairer_theory_2014}
M.~Hairer.
\newblock A theory of regularity structures.
\newblock {\em Inventiones mathematicae}, 198(2):269--504, March 2014.

\bibitem{HairerLabbe15}
Martin Hairer and Cyril Labb\'e.
\newblock A simple construction of the continuum parabolic {A}nderson model on
  {${\bf R}^2$}.
\newblock {\em Electron. Commun. Probab.}, 20:no. 43, 11, 2015.

\bibitem{labbe_2018}
C.~Labb\'e.
\newblock The continuous {A}nderson hamiltonian in $d \le 3$.
\newblock 2018.

\bibitem{reedsimon1}
Michael Reed and Barry Simon.
\newblock {\em Methods of modern mathematical physics. {I}. {F}unctional
  {A}nalysis}.
\newblock Academic Press, New York-London, 1972.

\bibitem{reedsimon2}
Michael Reed and Barry Simon.
\newblock {\em Methods of modern mathematical physics. {II}. {F}ourier
  {A}nalysis, self-adjointness}.
\newblock Academic Press Harcourt Brace Jovanovich, Publishers, New
  York-London, 1975.

\bibitem{tao2006nonlinear}
Terence Tao.
\newblock {\em Nonlinear dispersive equations: local and global analysis}.
\newblock Number 106. American Mathematical Soc., 2006.

\end{thebibliography}
\bibliographystyle{plain}

\Addresses
\end{document}